\documentclass[a4paper,reqno]{amsart}

\usepackage[utf8]{inputenc}
\usepackage[T1]{fontenc}
\usepackage[english]{babel}

\usepackage{amsfonts, color}
\usepackage[margin=3.5cm]{geometry}
\usepackage{todonotes}

\theoremstyle{plain}
\begingroup
\newtheorem{theorem}{Theorem}[section]
\newtheorem{lemma}[theorem]{Lemma}
\newtheorem{proposition}[theorem]{Proposition}
\newtheorem{corollary}[theorem]{Corollary}
\endgroup

\theoremstyle{definition}
\begingroup
\newtheorem{definition}[theorem]{Definition}
\newtheorem{remark}[theorem]{Remark}

\endgroup

\theoremstyle{remark}

\DeclareMathOperator{\tr}{tr}
\newcommand{\al}{\alpha}
\newcommand{\be}{\beta}
\DeclareMathOperator{\Div}{div}
\newcommand{\e}{\epsilon}
\newcommand{\gd}{\gamma_d}
\newcommand{\gn}{\gamma_n}
\newcommand{\Gd}{\Gamma_d}
\newcommand{\HH}{\mathcal{H}}
\newcommand{\la}{\lambda}
\newcommand{\LL}{\mathcal{L}}
\newcommand{\Mtwo}{\mathbb{M}_\mathrm{sym}^{2 \times 2}}
\newcommand{\Mthree}{\mathbb{M}_\mathrm{sym}^{3 \times 3}}
\newcommand{\Mn}{\mathbb{M}_{sym}^{n \times n}}

\newcommand{\om}{\omega}
\newcommand{\Om}{\Omega}
\newcommand{\R}{\mathbb{R}}
\newcommand{\wto}{\rightharpoonup}
\renewcommand{\doteq}{:=}
\newcommand{\eps}{\epsilon}
\allowdisplaybreaks

\mathchardef\emptyset="001F
\numberwithin{equation}{section}

\title[Existence and regularity]
{Existence and regularity of solutions for an evolution model of perfectly plastic plates \\}
\author[P. Gidoni, G.B. Maggiani,  and R. Scala]{P. Gidoni, G.B. Maggiani,  and R. Scala}

\address[P. Gidoni]{Centro de Matemática, Aplicações Fundamentais e Investigação Operacional (CMAF-CIO), Universidade de Lisboa,\\
Campo Grande, Edifício C6, 1749-016 Lisboa, Portugal
}
\email{pgidoni@fc.ul.pt}

\address[G.B.~Maggiani]{Dipartimento di Matematica, Universit\`a di Pavia, Via Ferrata 1, 27100 Pavia, Italy}
\email{giovannibattis.maggiani01@universitadipavia.it}

\address[R. Scala]{Centro de Matemática, Aplicações Fundamentais e Investigação Operacional (CMAF-CIO), Universidade de Lisboa\\
	Campo Grande, Edifício C6, 1749-016 Lisboa, Portugal
}
\email{rscala@fc.ul.pt}

\begin{document}

\begin{abstract}
We continue the study of a dynamic evolution model for perfectly plastic plates, recently derived in \cite{MaggianiMora} from three-dimensional Prandtl--Reuss plasticity. We extend the previous existence result by introducing non-zero external forces in the model, and we discuss the regularity of the solutions thus obtained. In particular, we show that the first derivatives with respect to space of the stress tensor are locally square integrable.

\end{abstract}

\maketitle
\vspace{2pt}

\noindent {\bf Mathematics Subject Classification (MSC2010):}  74C05, 74K20, 49J45.

\section{Introduction}

This paper studies the dynamic evolution  of a thin plate in Prandtl--Reuss plasticity, continuing the investigation undertaken in \cite{MaggianiMora}, where the authors rigorously derived a reduced model describing such situation.
Here, we pursue the analysis of that model by extending the existence result in \cite{MaggianiMora} to the case of nonzero external forces, and proving a regularity result for the stress field and for the vertical component of the displacement. Indeed, we point out that in  \cite{MaggianiMora} the stress is proved to be square integrable with respect to the space variables, without addressing any higher regularity; here we show $H^1_\mathrm{loc}$-regularity. Moreover, we also prove that the vertical component of the displacement field satisfies $H^1_\mathrm{loc}$-regularity in space uniformly in time.

We remark that an analogous investigation has been previously carried out in the simpler case of quasistatic evolution, namely neglecting inertial effects and assuming that the system is in dynamic equilibrium at every time. Indeed, this paper may be seen as the dynamic counterpart of  \cite{DavMor2}, where the authors analyse the stress regularity for the reduced model derived in \cite{DavMor}, describing the quasistatic evolution of a plastic plate. The regularity of the elastic stress is a classical issue in perfect plasticity, investigated for instance also in \cite{Bensoussan}, \cite{Demyanov1}, and \cite{Demyanov3} (see the related results in \cite{BM18}). For general treatment of quasistatic evolutions of rate-independent systems we refer to \cite{MM05,MieRou}.  

\subsubsection*{Formulation of the model} Before to discuss our regularity result, we briefly describe the mechanical model studied in this paper; more details are provided in Sections \ref{preliminaries} and \ref{sec:problem}.  This model was proposed and rigorously justified in \cite{MaggianiMora}, as the limit of a dynamic evolution for a three-dimensional plate, when the thickness tends to zero. We emphasize however that the limit model is purely three-dimensional, because the dependence of the stress on the variable $x_3$ is in general not trivial; a similar situation appears also in the quasistatic setting  (cf.~\cite[Section 5]{DavMor2} for an explicit example).

An elastoplastic plate is described in the reference configuration by the set
$ \Om \doteq \om \times \left( -\frac{1}{2},\frac{1}{2} \right)$, where $\om\subset\R^2$ is an open, bounded set corresponding to the base of the plate. 
The state of the plate at time $t$ is described by the triplet $(u(t),e(t),p(t))$, where $u(t)$ is the \textit{displacement}, $e(t)$ is the \textit{elastic strain} and $p(t)$ is the \textit{plastic strain}.
We assume that the displacement $u(t)$ is of \emph{Kirchhoff--Love} type: namely that $u_3(t)$ is independent of the transverse variable $x_3$ and there exists $\bar{u}(t): \om \to \R^2$ with
$$ u_\al (t)=\bar{u}_\al (t)-x_3 \partial_\al u_3 (t) \quad \mbox{ in } \Om,  \mbox{ for } \al=1,2. $$
As observed in \cite{Ciarlet}, the physical interpretation of this condition is that straight lines normal to the mid-surface remain straight and normal after the deformation, within the first order. 

The evolution of the system is guided by two main elements: a 
time-dependent Dirichlet boundary datum $w(t)$, prescribing the displacement field on a portion $\Gd=\gd\times\left( -\frac{1}{2},\frac{1}{2} \right)$ of the lateral boundary $\partial \om \times \left( -\frac{1}{2},\frac{1}{2} \right)$, and an external force $\mathcal L$, which in turn acts on the body, an will be decomposed as the sum of vertical and horizontal loads $g$ and $f$.

To characterize the dynamic evolution $[0,T]\ni t\rightarrow (u(t),e(t),p(t))$, we require that,  for every  $t \in [0,T]$, the following conditions (cf1)--(cf5) are satisfied.

\begin{itemize}
	\item[(cf1)] \textit{Kinematic admissibility:} 	\begin{align*}
	&Eu(t)=e(t)+p(t) \;\;\text{ in $\Om$,} \\
	&p(t)=(w(t)-u(t)) \odot \nu_{\partial \Om}  \quad \mathcal{H}^2-\text{a.e. on $\Gd$,}\\
	&e_{i3}(t)=p_{i3} (t)=0\;\;\text{ in $\Om$,  for $i=1,2,3$.} 
	\end{align*}
\end{itemize}
Here $E u(t) \doteq \frac{Du(t)+Du^T (t)}{2}$ is the symmetrized gradient of $u$, $\odot$ is the symmetrized tensor product, $\nu_{\partial \Om}$ is the outer unit normal to $\partial \Om$, and $\mathcal{H}^2$ denotes the two-dimensional Hausdorff measure. The third equation expresses the Kirchhoff--Love structure of the displacement. In view of such property, in the following we identify $e$ and $p$ with their nontrivial components taking values in $\Mtwo$.

 Regarding the second equation, one may expect instead the usual boundary condition  $u(t)=w(t)$ on $\Gd$; indeed, this will be our requirement on the approximating evolution problems. Yet, as happens already in the quasistatic setting, the existence of the solution  $(u(t),e(t),p(t))$ to the limit evolution problem is provided only in the product space 
$$ BD(\Om) \times L^2(\Om; \Mtwo) \times M_b (\Om \cup \Gd; \Mtwo), $$
where $BD(\Om)$ is the space of functions with bounded deformation in $\Om$, and $M_b (\Om \cup \Gd; \Mtwo)$ denotes the space of bounded Borel measures on $\Om \cup \Gd$. For this reason the boundary condition has to be formulated in a weak sense, as above. The mechanical interpretation of the second equation in (cf1) is that $u(t)$ may not attain the boundary condition and, if this is the case, a plastic slip of strength proportional to $w(t)-u(t)$ develops along $\Gd$.

	\begin{itemize}
	\item[(cf2)] \textit{Constitutive law:} $$\sigma (t) =\mathbb{C}_r e(t),$$ where $\mathbb{C}_r$ is the elasticity tensor.
	\item[(cf3)] \textit{Evolution equations:} given $\bar{\sigma}$ and $\hat{\sigma}$ be respectively the stretching and the bending components of $\sigma$ (cf.~Definition \ref{definitionmomentE}), we have $$-\Div \bar{\sigma}(t)=f(t)\text{ in }\Om$$ and $$\ddot{u}_3 (t)-\frac{1}{12} \Div \Div \hat{\sigma}(t)=g(t)\text{ in }\Om,$$ together with corresponding homogeneous Neumann boundary conditions on $\partial \om \setminus \gd$, corresponding to absent boundary tractions. 
	\item[(cf4)] \textit{Stress constraint:} $$\sigma (t) \in K_r \doteq \lbrace \xi \in \Mtwo: \vert \xi \vert_r \leq \al_0 \rbrace, $$
	where
	$$ \vert \xi \vert_r \doteq \sqrt{\vert \xi \vert -\frac{1}{3}(\tr\xi)^2} \quad \mbox{ for every } \xi \in \Mtwo, $$
	 so that $K_r$ is a convex, compact ellipsoid in $\Mtwo$ containing the origin (cf.~Section~\ref{preliminaries}).
\end{itemize}	
This assumption, adopted also in the quasistatic setting in \cite{DavMor2}, follows from the restriction, in the original three dimensional problem, to the case of von Mises yield criterion  (see, e.g., \cite{Lubliner}), namely assuming that the set of admissible stresses  is a cylinder $B_{\al_0}+\mathbb{R} I_{3 \times 3}$, with $B_{\al_0}$ being the ball of radius $\al_0$ centered at the origin of the space of trace-free $\Mthree$ matrices and $I_{3 \times 3}$ being the identity matrix in $\Mthree$.
\begin{itemize}	
	\item[(cf5)] \textit{Flow rule:} $\dot{p}(t)$ belongs to the normal cone to $K_r$ at $\sigma (t)$.
\end{itemize}
It is easy to see that condition (cf5) can be written in an equivalent way as
\begin{itemize}
	\item[(cf5')] \textit{Maximum dissipation principle:} $$H_r (\dot{p}(t))=\sigma (t): \dot{p}(t),$$ where $H_r$ is the support function of $K_r$ (for a precise definition see \eqref{supportfunction}).
\end{itemize}
 We remark that, for the lack of spatial regularity in $p(t)$ mentioned above, this expression  has to be read with the following precise sense. The left-hand side is defined using the theory of convex functions of measures as
$$ \int_{\Om \cup \Gd} H_r \left( \frac{d p(t)}{d \vert p (t) \vert} \right) d \vert p(t) \vert,  $$ 
where $d p(t)/ d \vert p(t) \vert$ is the Radon--Nikodym derivative of $p(t)$ with respect to its total variation $\vert p(t) \vert$. The right-hand side of (cf5') requires an ad-hoc elastic  stress-plastic strain duality notion, that was originally given in \cite{DavMor}, and is summarized here in Section \ref{preliminaries}.

\subsubsection*{Main result}
The main result of the paper is  Theorem \ref{mainresult}, stating the existence of a solution to the dynamic evolution  model discussed above, and additional regularity for the stress field and third component of the displacement; both these variables are shown to be locally $H^1$ with respect to the space variables. More precisely, we prove that for every open set $\om'$ compactly contained in $\om$ there exists a positive constant $C_1(\om')$ such that, for $\alpha=1,2$,
\begin{equation} \label{stima1}
\sup_{t \in [0,T]} \Vert D_\al \sigma (t) \Vert_{L^2\left( \om' \times \left( -\frac{1}{2},\frac{1}{2} \right) \right)} \leq C_1 (\om'),
\end{equation}
and for every open set $\Om'$ compactly contained in $\Om$  there exists a positive constant $C_2(\Om')$ such that
\begin{equation} \label{stima2}
\sup_{t \in [0,T]} \Vert D_3 \sigma (t) \Vert_{L^2\left( \Om'  \right)} \leq C_2 (\Om').
\end{equation}
Analogous estimates hold for the vertical component $u_3$ of the  displacement $u$.
We observe that the estimate \eqref{stima1} is stronger than \eqref{stima2}, since it is global in the direction $x_3$. We remark that all the arguments used are purely local, thus cannot be used to study the behaviour of stress up to the boundary $\partial \Om$. 

Finally, we recall that also the existence result presents some novelty, due to the presence of general external forces in the bulk, not considered in \cite{MaggianiMora}.

\subsubsection*{Strategy of the proof}
The proof is developed in several steps, which can be resumed into two partial existence results of approximating solutions to $(u,e,p)$, which are obtained considering progressive approximations of the flow rule (cf5). 

The key point is the study of the so-called approximating Norton--Hoff problem, see Theorems \ref{TheoremNHoff} and \ref{TheoremNHoff2}. In these theorems we show that for any natural numbers $N\geq4$ there exists a triplet $(u_N,e_N,p_N)$ solving (cf1)--(cf3) and the condition
\begin{equation}
\dot p_N(t)=D\phi_N(\sigma(t)),
\end{equation}
where $\phi_N$ is a convex potential approximating the indicator function of  $K_r$ as $N\rightarrow+\infty$.
 These results, together with a priori estimates on the norms of the solutions independent of $N$, can be considered of individual interest and are one of the  main novelties of this work.

In turn, these Norton--Hoff problems are studied by considering a second approximation of flow rule, obtained by truncating the potential $\phi_N$ at a certain threshold. The corresponding approximating solutions, obtained in Lemma \ref{Lemmalambda}, depend on a truncation parameter $\lambda>0$, which will be sent to $+\infty$ at the limit. The proof of Lemma \ref{Lemmalambda} is carried out by standard time discretization and using an implicit Euler scheme. 

In Proposition \ref{propositionregular} we recover additional regularity in the Norton--Hoff problems, specifically for the stress and the third component of the displacement. The main result is obtained by letting  $N\rightarrow+\infty$,  showing that the approximating solutions converge to the solution of the original problem (cf1)--(cf5), and that the additional regularity is preserved at the limit.

\subsubsection*{Comparison with the quasistatic case}
As mentioned above, Theorem \ref{mainresult} can be seen as the dynamic correspondent of the regularity result of \cite{DavMor2} for the quasistatic case, and the two approaches present several analogies and a similar overall structure. However, the introduction of inertial effects produces new difficulties, requiring a different proof for the existence of the approximating solutions, and the use of several ad-hoc techniques and part integration formulas. 
In particular we remark that, due to the inertial terms, the kinematic admissibility and the evolution equations are now coupled, so that the argument adopted in the quasistatic context cannot be repeated.

Differences between the quasistatic and dynamic models may be found also in the requirements on the  external data.
In \cite{DavMor2} the boundary Dirichlet datum $w$ is taken in the space $H^1([0,T]; H^1(\Om; \R^3)\cap KL (\Om))$.
In the dynamic framework such regularity is no longer sufficient. Indeed, we need a control also on the second derivative in time of $w$ and the third derivative of $w_3$. Moreover, in order to treat the spatial regularity of the stress we need a local control of the $H^2$ norm of $w$ (see \eqref{regularityw} for the precise assumption). 
Moreover, we have to assume an uniform safe-load condition on the external forces. Again, due to the presence of inertia, we need a control on the velocity and acceleration of the external forces via the safe-load variable. These are stronger hypotheses than the ones in the quasistatic case, but at the presence stage seem necessary in the dynamic one.

\medbreak

The paper is organized as follows. In Section 2 we recall some mathematical preliminaries. In Section 3 we give some mechanical preliminaries and describe the setting of the problem. In Section 4 we present the existence and regularity results for the approximating problems. By Lemma \ref{Lemmalambda} and then Theorem \ref{TheoremNHoff} we prove the existence of solutions to the Norton--Hoff problem. In Proposition \ref{propositionregular} we prove the  a priori estimates for the space derivatives of the stress. Finally we prove that these estimates hold true for the original problem in Theorem \ref{mainresult}, where the existence of a solution to (cf1)--(cf5) is stated.

\section{Mathematical preliminaries and notations}

\subsubsection*{Measures} 
We denote with $\LL^n$ the Lebesgue measure on $\R^n$, and with 
$\HH^{n-1}$ the $(n-1)$-dimensional Hausdorff measure.
For any given  Borel set $B\subset\R^n$, we denote  the space of bounded Borel measures on $B$ with values in a finite dimensional Hilbert space $X$ as $M_b(B;X)$; such space is 
endowed with the norm $\|\mu\|_{M_b} \doteq |\mu|(B)$, where $|\mu|\in M_b(B;\R)$ denotes the variation of the measure $\mu$. 
For every measure $\mu \in M_b(B;X)$, by Lebesgue decomposition Theorem there exist a measure $\mu^a$, absolutely continuous with respect to the Lebesgue measure  $\LL^n$, and a measure $\mu^s$ singular with respect to $\LL^n$, such that $\mu=\mu^a+\mu^s$.  When $\mu^s=0$, we identify $\mu$ with its density with respect to $\LL^n$. 

Moreover, if  $B$ is locally compact in the relative topology, 
by Riesz--Markov Theorem we can identify $M_b(B;X)$ with the topological dual space of $C_0(B;X)$. We recall that 
$C_0(B;X)$ is the space of  continuous functions $f\colon B\to X$ such that, for every $\e>0$, the set $\{ |f| \geq \e \}$ is compact. Using this property, we define the weak$^*$ topology of $M_b(B;X)$. We will denote the dual pairings between measures and continuous functions, as well as between other couples of spaces, as usual, by~$\langle \cdot, \cdot \rangle$.

\subsubsection*{Functions with bounded deformation} 
Let us denote with $\Mn$ the space of $n\times n$ real symmetric matrices. Given an open set $U \subset \R^n$, we define the space $BD(U)$ of functions with bounded deformation as the space whose elements are the functions $u \in L^1(U;\R^n)$ with symmetric gradient (in the sense of distributions) $Eu \doteq \frac12(Du+Du^T)$ belonging to $M_b(U; \Mn)$. 
The space $BD(U)$, endowed with the norm
$$
\| u \|_{BD} \doteq \| u \|_{L^1}+\| Eu \|_{M_b},
$$
is a Banach space. 
We refer to \cite{Temam} for the properties of the space $BD(U)$; we recall here the most relevant for our purposes.

If the set $U$ is bounded with Lipschitz boundary, then the space $BD(U)$ is continuously embedded in $L^{n/(n-1)}(U; \R^n)$ and compactly embedded in $L^p(U; \R^n)$, for every $p<n/(n-1)$. 

A sequence $(u_k)_k$ is said to converge weakly$^*$ in $BD(U)$ to $u$  if $u_k\wto u$ weakly in $L^1(U; \R^n)$ and $Eu_k \wto Eu$ weakly$^*$ in $M_b(U; \Mn)$. Every bounded sequence in $BD(U)$ admits a weakly$^*$ converging subsequence.

Furthermore, to every function $u \in BD(U)$ we can associate a trace, still denoted by $u$, belonging to $L^1(\partial U ; \R^n)$. Given a subset $\Gamma$ of $\partial U$ with positive $\mathcal H^{n-1}$-measure, we have
$$
\| u \|_{BD} \leq C( \| u \|_{L^1(\Gamma)}+ \| Eu \|_{M_b}).
$$
where the constant $C>0$, depends only on $U$ and $\Gamma$

\subsubsection*{Functions with bounded Hessian} 
Given an open set $U\subset \R^n$, we define the space $BH(U)$ of functions with bounded Hessian as the space whose elements are the functions $u \in W^{1,1}(U)$, with Hessian $D^2 u$ (in the sense of distributions) belonging to $M_b(U; \Mn)$. The space $BH(U)$,  endowed with the norm
$$
\| u \|_{BH} \doteq \| u \|_{W^{1,1}}+\| D^2 u \|_{M_b},
$$
is a Banach space. We refer to \cite{Demengel1} for the main properties of $BH(U)$.

According to the regularity of $U$, we can deduce several properties of $BH(U)$
If $U$ has the cone property, then the space $BH(U)$ coincides with the space of functions in $L^1(U)$ whose Hessian belongs to $M_b(U; \Mn)$. If $U$ is bounded with Lipschitz boundary, the space $BH(U)$ can be embedded into $W^{1,n/(n-1)}(U)$. Moreover, if $U$ is bounded with $C^2$ boundary, then  the traces of $u$ and $\nabla u$, still denoted by $u$ and $\nabla u$ are well defined for every $u \in BH(U)$; moreover we have  $u \in W^{1,1}( \partial U)$, $\nabla u \in L^1(\partial U; \R^n)$, and $\frac{\partial u}{ \partial \tau}= \nabla u \cdot \tau \in L^1(\partial U)$ for every $\tau$ tangent vector to $\partial U$. Finally, if $n=2$, then $BH(U)$ embeds into the space $C(\overline{U})$, of continuous functions on $\overline{U}$. 

\subsubsection*{Maximal monotone operators} Let $X$ be a Banach space and let $X'$ be its dual. Let  $A:X \to X'$ be an operator, possibly multivalued, and let $ D(A) \doteq \lbrace x \in X: Ax \neq   \emptyset \rbrace$ be its domain. $A$ is \textit{monotone} if 
\begin{equation*}
\langle y_1-y_2,x_1-x_2 \rangle \geq 0 \mbox{ for every } x_1,x_2 \in D(A) \mbox{ and } y_1 \in A x_1, y_2 \in A x_2. 
\end{equation*}
A monotone operator is said to be \textit{maximal} if it satisfies the following property: if $(x,y) \in X \times X'$ are such that 
\begin{equation*}
\langle y-\eta, x-\xi \rangle \geq 0 \mbox{ for every } \xi \in D(A) \mbox{ and } \eta \in A \xi, 
\end{equation*}
then $y \in Ax$. We also recall the following useful property of maximal monotone operators (see \cite[Chapter II, Lemma 1.3]{Barbu}):
\begin{proposition} \label{mintytrick}
Let $A:X \to X'$ be a maximal monotone operator, possibly multivalued. We assume that
\begin{align*}
&y_k \in Ax_k, \\
&x_k \wto x \mbox{ weakly in } X, \\
&y_k \wto y \mbox{ weakly* in } X', \\
&\limsup_{k \to +\infty} \langle y_k,x_k \rangle \leq \langle y,x \rangle.
\end{align*}
Then $y \in Ax$.
\end{proposition}
For the general properties of maximal monotone operators we refer, e.g.,  to \cite{Barbu} and  \cite{Brezis}.

\section{Mechanical preliminaries and the setting of the problem}\label{preliminaries}

\subsubsection*{Reference configuration.}
In the reference configuration, we consider a thin plate described by the set $\Om \doteq \om \times \left(-\frac{1}{2},\frac{1}{2} \right)$, where $\om$ is  an open, connected and bounded set with  boundary of class $C^2$. We denote the points in  $\Om$ as $x=(x',x_3)$, where $x'=(x_1,x_2)\in\omega$ and $x_3\in \left(-\frac{1}{2},\frac{1}{2} \right)$. 

We partition the boundary $ \partial \om$ into two disjoint open subsets $\gd,\gn$, with common boundary $\partial_{| \partial \omega} \gd$; we write $\Gamma_d:=\gamma_d\times\left(-\frac{1}{2},\frac{1}{2} \right)$ and $\Gamma_n:=\gamma_n\times\left(-\frac{1}{2},\frac{1}{2} \right)$. Moreover we assume that 
$\mathcal H^1(\gamma_d)>0$ and $\partial_{| \partial \omega} \gd=\{ P_A,P_B \}$, where $P_A$ and $P_B$ are two points of $\partial \omega$.

\subsubsection*{Symmetric matrices.}
Let $\Mtwo$ be the vectorial space of symmetric $2\times2$-matrices.
Let $\vert \cdot \vert_r$ be an anisotropic norm on $\Mtwo$ given by
$$ \vert \xi \vert_r \doteq \sqrt{\vert \xi \vert^2-\frac{1}{3}\left( \tr \xi \right)^2} \mbox{ for every } \xi \in \Mtwo.$$
This norm satisfies
\begin{equation} \label{normarequiv}
\frac{1}{\sqrt{3}} \vert \xi \vert_r  \leq \vert \xi \vert \leq \vert \xi \vert_r  \mbox{ for every } \xi \in \Mtwo.
\end{equation}
The scalar product associated to $\vert \cdot \vert_r$ is
$$ \left( \xi, \zeta \right)_r \doteq \xi: \eta-\frac{1}{3}\tr \xi \tr \zeta\;\; \mbox{ for every } \;\xi,\eta \in \Mtwo. $$
A direct computation shows that the dual norm $\mid \cdot \mid_*$ of $\mid \cdot \mid_r$ is
\begin{equation} \label{dualnorm}
\vert \xi \vert_* \doteq \sup_{\vert \eta \vert_r \leq 1} \xi : \eta = \sqrt{\vert \xi \vert^2+\left( \tr \xi \right)^2} \mbox{ for every } \xi \in \Mtwo.
\end{equation}
Moreover, it is easy to check that for every $\xi \in \Mtwo$
\begin{eqnarray} \label{normadualnormar}
&\vert \xi \vert_* = \vert \xi + (\tr\xi)I \vert_r, \\  
& \vert \xi \vert_r = \vert \xi -\frac{1}{3} (\tr\xi)I \vert_* \label{normarnormadual}. 
\end{eqnarray}

\subsubsection*{The elasticity  tensor and its inverse.} 

We denote by $\mathbb{C}_r$ the \textit{elasticity tensor}, which we recall to be a symmetric positive definite linear operator $\mathbb{C}_r: \Mtwo \to \Mtwo$. Morever we denote by $\mathbb{A}_r: \Mtwo \to \Mtwo$ its inverse $\mathbb{A}_r \doteq \mathbb{C}_r^{-1}$. It follows that there exist two constants $\alpha_\mathbb{A}$ and $\beta_\mathbb{A}$, with $0 < \alpha_\mathbb{A} \leq \beta_\mathbb{A}$,  such that
\begin{equation} \label{Acoercivity}
\al_\mathbb{A} \vert \xi \vert_r^2 \leq \frac{1}{2} \mathbb{A}_r \xi: \xi \leq \be_\mathbb{A} \vert \xi \vert_r^2  \mbox{ for every } \xi \in \Mtwo. 
\end{equation}
In particular it holds
\begin{equation} \label{Acoercivity2}
 \vert \mathbb{A}_r \xi \vert \leq 2 \be_\mathbb{A} \vert \xi \vert \mbox{ for every } \xi \in \Mtwo. 
\end{equation}

\subsubsection*{Dissipation potential.}

Let $K$ be a closed convex set of $\mathbb M^{3\times3}_\text{sym}$ such that there exist two constants $r_H$ and $R_H$, with $0 < r_K \leq R_K$, such that
$$ \lbrace \xi \in \mathbb M^{3\times3}_\text{sym}: \vert \xi \vert \leq r_K \rbrace \subseteq K \subseteq\lbrace \xi \in \mathbb M^{3\times3}_\text{sym}: \vert \xi \vert \leq R_K \rbrace. $$
The set $K_r\subset\Mtwo$ represents the set of admissible stresses in the reduced problem and can be characterised as follows:
\begin{equation} \label{daKaKridotto}
\xi \in K_r \quad \Leftrightarrow \quad \begin{pmatrix}
\xi_{11} & \xi_{12} & 0 \\
\xi_{12} & \xi_{22} & 0 \\
0 & 0 & 0
\end{pmatrix}-\frac{1}{3} (\tr\xi)I_{3 \times 3} \in K,
\end{equation}
(see \cite[Section~3.2]{DavMor}).
In particular we note that if 
$$ K= \left\lbrace \xi \in \Mthree: \vert \xi-\frac{1}{3} \tr \xi I_{3 \times 3} \vert \leq \al_0 \right\rbrace $$
for some $\al_0>0$, then \eqref{daKaKridotto} gives
\begin{equation} \label{formaKr}
K_r \doteq \left\lbrace \xi \in \Mtwo: \vert \xi \vert_r \leq \al_0 \right\rbrace.
\end{equation}
We define the set
$$ \mathcal{K}_r (\Om) \doteq \lbrace \sigma \in L^2(\Om; \Mtwo): \sigma (x) \in K_r \mbox{ for a.e. } x \in \Om \rbrace . $$

The \textit{plastic dissipation potential} is given by the support function of $K_r$, namely $H_r\colon \Mtwo \to [0,+\infty)$  defined as
\begin{equation} \label{supportfunction}
 H_r (\xi) \doteq \sup_{\sigma \in K_r} \sigma : \xi \;\;\;\;\mbox{ for every } \xi \in \Mtwo.
\end{equation}
It follows that $H_r$ is convex and positively one-homogeneous and there are two constants $0<r_H<R_H$ such that 
\begin{equation}
r_H \vert \xi \vert \leq H_r (\xi) \leq R_H \vert \xi \vert\;\;\;\; \mbox{ for every } \xi \in \Mtwo.
\end{equation}
Therefore $H_r$ satisfies the triangle inequality 
\begin{equation}
H_r (\xi+\zeta) \leq H_r (\xi)+H_r (\zeta)\;\;\;\; \mbox{ for every } \xi,\zeta \in \Mtwo.
\end{equation}
With this definition it is easy to deduce the important property
$$K_r=\partial H_r(0),$$
where $\partial H_r(0)$ is the classical subdifferential of $H_r$ at the origin.

\subsubsection*{Kirchhoff--Love admissible triples} 
We now focus on the admissible configurations of the plate. Regarding the displacement $u$, we require 
it to be in the set $KL(\Omega)$ of Kirchhoff--Love displacements:
$$
KL(\Omega) \doteq \{ u \in BD(\Omega): \ (Eu)_{i3}=0, \ i=1,2,3 \}.
$$
We have the following alternative characterization of Kirchhoff--Love displacements:  $u\in KL(\Omega)$ if and only if $u_3 \in BH(\omega)$ and there exists $\bar{u} \in BD(\omega)$ such that
$$
u_\alpha(x) =\bar{u}_\alpha (x')-x_3 \partial_\alpha u_3(x') \quad \mbox{ for } x=(x',x_3) \in \Omega, \ \alpha=1,2.
$$
The terms $\bar{u}$, $u_3$ are called the \textit{Kirchhoff--Love components} of $u$. 

For any given $w \in H^1(\Omega;\R^3) \cap KL(\Omega)$, used to prescribe the Dirichlet boundary condition, we define the set $\mathcal{A}_{KL}(w)$ of \textit{Kirchhoff--Love admissible triplets},  as the class of all triplets
$$
(u,e,p) \in KL(\Omega) \times L^2(\Omega; \mathbb M^{3\times3}_\text{{\rm sym}}) \times M_b (\Omega \cup \Gd; M^{3\times3}_\text{{\rm sym}})
$$
such that
$$
\begin{array}{c}
Eu=e+p \quad \mbox{ in } \Omega, \qquad
p=(w-u)\odot \nu_{\partial \Omega}\HH^2 \quad \mbox{ on } \Gd,\medskip\\
e_{i3}=0 \quad \mbox{ in } \Omega, \qquad p_{i3}=0 \quad \mbox { in } \Omega \cup \Gd, \qquad i=1,2,3.
\end{array}
$$
 In view of this latter property, in the following, given $(u,e,p)\in \mathcal{A}_{KL}(w)$,
we will always identify $e$ with a function in $L^2(\Omega; \Mtwo)$ and $p$ with a measure in $M_b(\Omega \cup \Gd; \Mtwo)$.

To provide a useful characterisation of admissible triplets in $\mathcal{A}_{KL}(w)$, let us first recall the definitions of zeroth and first order moments.

\begin{definition} \label{definitionmomentE}
Let $f \in L^2(\Omega;\Mtwo)$. We define $\bar{f}$, $\hat{f} \in L^2(\omega;\Mtwo)$ and $f_{\perp} \in  L^2(\Omega; \Mtwo)$ as the following orthogonal  components (with respect to the scalar product of  $L^2(\Omega; \Mtwo)$) of $f$:
$$
\bar{f}(x') \doteq \int_{-\frac{1}{2}}^{\frac{1}{2}} f(x',x_3)\, dx_3, \qquad
\hat{f}(x') \doteq 12 \int_{-\frac{1}{2}}^{\frac{1}{2}}x_3 f(x',x_3)\,dx_3
$$
for a.e.\ $x' \in \omega$, and
$$
f_\perp(x) \doteq f(x)-\bar{f}(x')-x_3\hat{f}(x')
$$
for a.e.\ $x \in \Omega$.
We name $\bar{f}$ the \textit{zero-th order moment} of $f$ and $\hat{f}$ the \textit{first order moment} of~$f$.
\end{definition}

\begin{definition} \label{definitionmomentP}
Let $q \in M_b(\Omega \cup \Gd;\Mtwo)$. We define $\bar{q}$, $\hat{q} \in M_b(\omega \cup \gd;\Mtwo)$ 
and $q_{\perp} \in  M_b(\Omega \cup \Gd;\Mtwo)$ as follows:
$$
\int_{\omega \cup \gd} \varphi: d\bar{q}\doteq \int_{\Omega \cup \Gd} \varphi: dq , \qquad
\int_{\omega \cup \gd} \varphi: d\hat{q} \doteq 12\int_{\Omega \cup \Gd} x_3 \varphi: dq
$$
for every $\varphi \in C_0(\omega \cup \gd; \Mtwo)$,
and
$$
q_{\perp} \doteq q-\bar{q} \otimes \LL^1-\hat{q} \otimes x_3 \LL^1,
$$
where $\otimes$ is the usual product of measures, and $\mathcal L^1$ is the Lebesgue measure restricted to the third component of $\R^3$. We name $\bar{q}$ the \textit{zero-th order moment} of $q$ and $\hat{q}$ the \textit{first order moment} of $q$.
\end{definition}

We are now ready to state the following characterisation of  $\mathcal{A}_{KL}(w)$.

\begin{proposition}[{cf.~\cite[Proposition 4.3]{DavMor}}] \label{charKLterne}
Let $w \in H^1(\Omega;\R^3) \cap KL(\Omega)$ and let $(u,e,p) \in  {KL}(\Omega) \times L^2(\Omega; \Mtwo) \times M_b(\Omega \cup \Gd; \Mtwo)$. Then  $(u,e,p) \in \mathcal{A}_{KL}(w)$ if and only if the following three conditions are satisfied:
\begin{itemize}
\item[(i)] $E\bar{u}=\bar{e}+\bar{p}$ in $\omega$ and $\bar{p}=(\bar{w}-\bar{u}) \odot \nu_{\partial \omega}\HH^1$ on $\gd$;\smallskip
\item[(ii)] $D^2 u_3=-( \hat{e}+\hat{p})$ in $\omega$, $u_3=w_3$ on $\gd$, and $\hat{p}=(\nabla u_3-\nabla w_3) \odot \nu_{\partial \omega}\HH^1$ on $\gd$;\smallskip
\item[(iii)] $p_\perp=-e_\perp$  in $\Omega$ and $p_\perp=0$ on $\Gd$.
\end{itemize}
\end{proposition}

\subsubsection*{Stress-strain duality}  
We define the set $\Sigma(\Omega)$ of \textit{admissible stresses} as
$$
\Sigma(\Omega) \doteq \{ \sigma \in L^\infty(\Omega;\Mtwo): \ \Div\, \bar{\sigma} \in L^2(\omega;\R^2), \ \Div\,\Div\,\hat{\sigma} \in L^2(\omega) \}.
$$

We define the space of \textit{admissible plastic strains} $\Pi_{\Gd}(\Omega)$ as the set of all the measures $p \in M_b(\Omega \cup \Gd; \Mtwo)$ for which there exists a triplet $(u,e,w) \in BD(\Omega) \times L^2(\Omega; \Mtwo) \times 
(H^1(\Omega;\R^3)\cap KL(\Omega))$ such that $(u,e,p) \in  \mathcal{A}_{KL}(w)$.

For any 
 $\sigma \in \Sigma (\Omega)$, since $\bar\sigma\in L^\infty(\omega;\Mtwo)$, the trace $[ \bar{\sigma} \nu_{\partial\omega}] \in L^\infty(\partial\omega; \R^2)$ of its zeroth order moment normal component  can be properly defined as
\begin{equation} \label{traccia0mom}
\langle[ \bar{\sigma} \nu_{\partial \omega}], \psi\rangle
\doteq \int_{\omega} \bar{\sigma} : E \psi \,dx'+\int_{\omega} \Div\,\bar{\sigma} \cdot \psi \,dx'
\end{equation} 
for every $\psi \in W^{1,1}(\omega;\R^2)$. 

Let us denote with $T(W^{2,1}(\omega))$  the space of the  traces on $\partial\om$ of functions in $W^{2,1}(\omega)$, and with $(T(W^{2,1}(\omega)))'$ its dual space. 
Since $\hat{\sigma} \in L^\infty(\omega;\Mtwo)$, also the traces $b_0(\hat{\sigma}) \in ( T(W^{2,1}(\omega)))'$
and $b_1(\hat{\sigma}) \in L^\infty(\partial \omega)$ of the first order moment of any 
$\sigma \in \Sigma (\Omega)$ can be properly defined as

\begin{equation} \label{traccia1mom}
-\langle b_0(\hat{\sigma}),\psi\rangle+\big\langle b_1(\hat{\sigma}), \frac{\partial \psi}{\partial \nu_{\partial \omega}} \big\rangle \doteq 
\int_{\omega} \hat{\sigma}: D^2 \psi\, dx'-\int_\omega \psi\,\Div\,\Div\, \hat{\sigma}\, dx'
\end{equation}
for every $\psi \in W^{2,1}(\omega)$. Moreover, if 
If $\hat{\sigma} \in C^2(\overline{\omega}, \Mtwo)$, we have
$$
\begin{array}{l}
b_0(\hat{\sigma})= \Div\, \hat{\sigma} \cdot \nu_{\partial \omega}+\dfrac{\partial}{\partial \tau_{\partial \omega}} \left( \hat{\sigma} \tau_{\partial\omega} \cdot \nu_{\partial\omega} \right), \smallskip\\
b_1(\hat{\sigma})=\hat{\sigma} \nu_{\partial \omega} \cdot \nu_{\partial \omega},
\end{array}
$$
where $\tau_{\partial \omega}$ denotes the tangent vector to $\partial \omega$.  

We remark that, since $[ \bar{\sigma} \nu_{\partial\omega}] \in L^\infty(\partial\omega; \R^2)$ and  $b_1(\hat{\sigma}) \in L^\infty(\partial \omega)$, the null-trace conditions $[ \bar{\sigma} \nu_{\partial\omega}]=0$ and $b_1(\hat{\sigma})=0$ on $\gn$ are clearly defined. On the other hand, for $b_0(\hat{\sigma})$, we say that $b_0(\hat{\sigma})=0$ on $\gn$ holds if $\langle b_0(\hat{\sigma}),\psi\rangle=0$
for every $\psi\in W^{2,1}(\omega)$ such that $\psi=0$ on~$\gd$.

\begin{definition}
Given any $h \in H^{-\frac{1}{2}}(\partial \om; \R^2)$ and $m=(m_0,m_1) \in H^{-\frac{3}{2}}(\partial \om; \R^2) \times H^{-\frac{1}{2}}(\partial \om) $ we define $\Theta(\gn,h,m)$ as the class of all $\sigma \in \Sigma (\Om)$ such that
$$ \langle [\bar{\sigma} \nu_{\partial \om}]-h,\varphi \rangle=0 $$
for every $\varphi \in H^{\frac{1}{2}}(\partial \om; \R^2)$ with $\varphi=0$ on $\gd$, and 
$$ \langle b_0(\hat{\sigma}) -m_0,\psi_0 \rangle=\langle b_1(\hat{\sigma}) -m_1,\psi_1 \rangle=0 $$
for every $\psi_0 \in H^{\frac{3}{2}}(\partial \om)$ with $\psi=0$ on $\gd$, and  every $\psi_1 \in H^{\frac{1}{2}}(\partial \om)$ with $\psi_1=0$ on $\gd$. 
\end{definition}

As discussed in the introduction, since the plastic strain exists only as a measure, to properly use the stress-strain duality between $\Sigma(\Omega)$ and $\Pi_{\Gd}(\Omega)$ we need first to consider a suitable notion of duality pairing (cf.~\cite{DavMor}).

For every $\sigma \in \Sigma(\Omega)$ and $\xi \in BD(\omega)$, we  define the distribution $[\bar{\sigma}: E \xi]$ on $\omega$ as
$$
\langle [ \bar{\sigma}: E \xi ], \varphi \rangle \doteq -  \int_\omega \varphi \, \Div\,\bar{\sigma} \cdot \xi\, dx'-\int_\omega \  \bar{\sigma}: (\nabla \varphi \odot \xi) \,dx'
$$
where $ \varphi \in C_c^\infty(\omega)$. From \cite[Theorem~3.2]{KohnTemam} it follows that $[ \bar{\sigma}: E \xi ] \in M_b(\omega)$ and its variation satisfies
$$
|[\bar{\sigma}: E \xi ] | \leq \| \bar{\sigma} \|_{L^\infty} | E \xi |\quad \mbox{ in } \omega.
$$
For every $\sigma \in \Sigma(\Omega)$ and $p \in \Pi_{\Gd}(\Omega)$, we set the measure $[\bar{\sigma}: \bar{p} ] \in M_b(\omega \cup \gd)$ as
$$
[\bar{\sigma}: \bar{p} ] \doteq\begin{cases}
[\bar{\sigma}: E\bar{u} ]-\bar{\sigma}: \bar{e} & \mbox { in } \omega, \smallskip \\
[ \bar{\sigma} \nu_{\partial \omega} ] \cdot (\bar{w}-\bar{u})\,\HH^1 & \mbox { on } \gd.
\end{cases}
$$
For every $\sigma \in \Sigma(\Omega)$ and $v \in BH(\omega)$, we set the distribution $[ \hat{\sigma}: D^2 v ]$ on $\omega$ as
$$
\langle [ \hat{\sigma}: D^2 v ], \psi \rangle \doteq \int_\omega \psi v \,\Div\,\Div\, \hat{\sigma} \,dx'-2 \int_\omega \hat{\sigma}: (\nabla v \odot \nabla \psi ) \,dx'-\int_\omega v \hat{\sigma}:D^2 \psi \,dx'
$$
for every $\psi \in C_c^\infty(\omega)$.
From \cite[Proposition~2.1]{Demyanov2} it follows that $[ \hat{\sigma}: D^2 v ] \in M_b(\omega)$ and its variation satisfies
$$
| [ \hat{\sigma}: D^2 v ] | \leq \| \hat{\sigma} \|_{L^\infty} | D^2 v | \quad \mbox{ in } \omega.
$$
Given $\sigma \in \Sigma(\Omega)$ and $p \in \Pi_{\Gd}(\Omega)$, we define the measure $[\hat{\sigma}: \hat{p} ] \in M_b(\omega \cup \gd)$ as
$$
[\hat{\sigma}: \hat{p} ] \doteq\begin{cases}
-[\hat{\sigma}: D^2 u_3 ]-\hat{\sigma}: \hat{e} & \mbox { in } \omega, \smallskip \\
b_1(\hat{\sigma})\dfrac{\partial (u_3-w_3)}{\partial \nu_{\partial \omega}}\, \HH^1 &\mbox { on } \gd.
\end{cases}
$$

Finally, we combine the notions above for zeroth and first order moments to define, for every $\sigma \in \Sigma(\Omega)$ and $p \in \Pi_{\Gd}(\Omega)$, the measure $[ \sigma: p ]_r \in M_b(\Omega \cup \Gd)$ as
$$
[ \sigma: p ]_r \doteq [ \bar{\sigma}: \bar{p} ] \otimes \LL^1+ \frac{1}{12} [ \hat{\sigma}: \hat{p} ] \otimes \LL^1-\sigma_\perp:e_\perp.
$$

We are now ready to define the stress-strain duality pairings as
$$
\langle \bar{\sigma}, \bar{p} \rangle \doteq [ \bar{\sigma}: \bar{p}  ] ( \omega \cup \gd ), \qquad
\langle \hat{\sigma}, \hat{p} \rangle \doteq  [ \hat{\sigma}: \hat{p} ](\omega \cup \gd )
$$
and
\begin{equation}\label{dual-def}
\langle \sigma,p \rangle_r \doteq [ \sigma: p ]_r( \Omega \cup \Gd )
= \langle \bar{\sigma}, \bar{p} \rangle +\frac{1}{12}\langle \hat{\sigma}, \hat{p} \rangle -\int_\Omega \sigma_\perp:e_\perp\, dx .
\end{equation}

To conclude, we recall the following integration by parts formula.
\begin{proposition}[{see \cite[Proposition~3.5]{DavMor2}}]
Let $\sigma \in \Sigma(\Omega)$, $w\in H^1(\Omega; \R^3)\cap KL(\Omega)$, and $(u,e,p)\in\mathcal{A}_{KL}(w)$. Then
\begin{eqnarray*}
\lefteqn{\int_{\Omega \cup \Gd} \varphi\, d[ \sigma: p]_r+\int_\Omega \varphi \sigma: ( e-Ew)\, dx} \\
& = & -\int_\omega \bar{\sigma}:(\nabla \varphi \odot (\bar{u}-\bar{w}))\, dx'-\int_\omega \Div\, \bar{\sigma} \cdot \varphi(\bar{u}-\bar{w}) \,dx' 
\\
&& {}+\langle  [ \bar{\sigma} \nu_{\partial \omega}], \varphi(\bar{u}-\bar{w}) \rangle+ \frac{1}{12} \int_\omega \hat{\sigma}:(u_3-w_3)D^2\varphi\, dx'  
\\
&& {}+\frac{1}{6} \int_\omega \hat{\sigma}:(\nabla \varphi \odot (\nabla u_3 -\nabla w_3)) \,dx'- \frac{1}{12} \int_\omega \varphi(u_3-w_3) \Div\, \Div\, \hat{\sigma}\, dx' 
\\
&& {}+\frac{1}{12} \langle b_0 (\hat{\sigma}), \varphi(u_3-w_3) \rangle-\frac{1}{12} \langle b_1 (\hat{\sigma}), \frac{\partial (\varphi(u_3-w_3))}{\partial \nu_{\partial \omega}} \rangle
\end{eqnarray*}
for every $\varphi \in C^2(\overline{\omega})$.
\end{proposition}

Since we shall consider homogeneous Neumann boundary condition for the stress, it is convenient to state the previous result in this setting.
\begin{corollary}\label{formulazzaintparti}
Let $\sigma \in \Sigma(\Omega)$ such that $\sigma\in \Theta(\gamma_n,0,0)$, $w\in H^1(\Omega; \R^3)\cap KL(\Omega)$, and $(u,e,p)\in\mathcal{A}_{KL}(w)$. Then
\begin{align*}
\int_{\Omega \cup \Gd} \varphi\, d[ \sigma: p]_r&+\int_\Omega \varphi \sigma: ( e-Ew)\, dx= -\int_\omega \bar{\sigma}:(\nabla \varphi \odot (\bar{u}-\bar{w}))\, dx'\\
&-\int_\omega \Div\, \bar{\sigma} \cdot \varphi(\bar{u}-\bar{w}) \,dx'+ \frac{1}{12} \int_\omega \hat{\sigma}:(u_3-w_3)D^2\varphi\, dx'\\  
&+\frac{1}{6} \int_\omega \hat{\sigma}:(\nabla \varphi \odot (\nabla u_3 -\nabla w_3)) \,dx'- \frac{1}{12} \int_\omega \varphi(u_3-w_3) \Div\, \Div\, \hat{\sigma}\, dx',
\end{align*}
for every $\varphi \in C^2(\overline{\omega})$.
 
\end{corollary}

\subsection*{The functions $\Phi_N$ and $\Psi_\la$}

Let $N \in \mathbb{N}$, $N \geq 4$ be a fixed parameter and let $\alpha_0>0$. We consider the function $\phi_N: \Mtwo \to [0,+\infty)$ given by
$$ \phi_N(\xi) \doteq \frac{1}{N \alpha_0^{N-1}} \vert \xi \vert_r^N \mbox{ for every } \xi \in \Mtwo. $$
The function $\phi_N$ is convex and of class $C^1$ with differential
$$ D \phi_N(\xi)=\frac{1}{\alpha_0^{N-1}} \vert \xi \vert_r^{N-2} \left( \xi -\frac{1}{3}(\tr \xi ) I_{2 \times 2} \right) \mbox{ for every } \xi \in \Mtwo.$$
For every $\la>0$ we define $\psi_\lambda: \Mtwo \to [0,+\infty)$ as
$$ \psi_\la (\xi) \doteq \frac{1}{N \al_0^{N-1}}(\vert \xi \vert_r^N \wedge \la^N )+\frac{1}{2 \al_0^{N-1}} \la^{N-2}(\vert \xi \vert_r^2-\la^2)^+\mbox{ for every } \xi \in \Mtwo,$$
where $(z)^+:=z\wedge 0$ denotes the positive part of $z$.
 Therefore
$$ D \psi_\la (\xi) =\frac{1}{\al_0^{N-1}} (\vert \xi \vert_r^{N-2} \wedge \la^{N-2} ) \left( \xi -\frac{1}{3} ( \tr \xi) I_{2 \times 2} \right) \mbox{ for every } \xi \in \Mtwo. $$
Hence, it is easy to see that
\begin{eqnarray} \label{Dpsilambdalipschitz}
&\vert D \psi_\la (\xi) \vert \leq \frac{\la^{N-2}}{\al_0^{N-1}} \vert \xi \vert, \\ \label{scalarDpsilambda}
&D \psi_\la  (\xi): \zeta= \frac{1}{\al_0^{N-1}} (\vert \xi \vert_r^{N-2} \wedge \la^{N-2} ) (\xi,\zeta)_r, \\ \label{DpsilambdaNN-1}
&\vert D \psi_\la (\xi) \vert^{N/(N-1)} \leq \frac{1}{\al_0} D \psi_\la (\xi): \xi.
\end{eqnarray}
Note that \eqref{Dpsilambdalipschitz} implies that $D\psi_\la$ is Lipschitz continuous.

Let us also introduce the functions 
$$\Phi_N:L^N(\Om; \Mtwo)\rightarrow [0,+\infty),$$ 
and
$$\Psi_\lambda: L^2(\Om; \Mtwo) \to [0,+\infty),$$ 
defined by
\begin{align}\label{PhiPsi}
 \Phi_N(\eta):=\int_\Om\phi_N(\eta(x))dx\;\;\;\;\;\mbox{ for every } \eta \in L^N(\Om; \Mtwo),\\
 \Psi_\lambda(\eta):=\int_\Om\psi_\lambda(\eta(x))dx\;\;\;\;\;\mbox{ for every } \eta \in L^2(\Om; \Mtwo).\label{PhiPsi2}
\end{align}

\subsubsection*{The conjugate of $\psi_\la$.}

In this subsection we compute the conjugate function of $\psi_\la$. First of all we recall this definition.
\begin{definition}
Let $X$ be a Banach space and let $f: X \to \left( -\infty, +\infty \right] $ with $f \not\equiv +\infty$. The \textit{conjugate function} $f^*: X^* \to \left( -\infty, +\infty \right] $ of $f$ is defined as
\begin{equation*}
f^* (x) \doteq \sup_{y \in X} \left\lbrace \langle x,y \rangle-f(y) \right\rbrace,
\end{equation*}
for every $x \in X^* $.
\end{definition}

We need the following three lemmas.

\begin{lemma}\label{14giugno_a}
 The function $D\psi_\lambda:\Mtwo \rightarrow \Mtwo$ is injective.
\end{lemma}
\begin{proof}
 Let $\xi_1,\xi_2\in \Mtwo$ and assume that
\begin{equation}  \label{dpsilambdaxi1dpsilambdaxi2}
  D\psi_\lambda(\xi_1)=D\psi_\lambda(\xi_2).
\end{equation} 
We will prove that $\xi_1=\xi_2$. Let us distinguish three cases:
 \begin{itemize}
  \item[(i)]  Assume $|\xi_1|_r,|\xi_2|_r\leq\lambda$. Then taking the norm $|\cdot|_*$ of both members of the equality \eqref{dpsilambdaxi1dpsilambdaxi2} and owing to \eqref{normarnormadual},  we infer 
$$|\xi_1|^{N-1}_r=|\xi_2|^{N-1}_r,$$ 
and then, from \eqref{dpsilambdaxi1dpsilambdaxi2}, we obtain 
$$\xi_1-\frac{1}{3} (\tr\xi_1)I_{2\times2}=\xi_2-\frac{1}{3}(\tr\xi_2)I_{2\times2}.$$ 
From this it easily follows that $\xi_1=\xi_2$.
  \item[(ii)] Assume $|\xi_1|_r\leq\lambda<|\xi_2|_r$. In such a case taking the norm $|\cdot|_*$ of both members of \eqref{dpsilambdaxi1dpsilambdaxi2} we infer 
$|\xi_1|^{N-1}_r=\lambda^{N-2}|\xi_2|_r>\lambda^{N-1},$
 which is a contradiction since by hypothesis  $|\xi_1|^{N-1}_r\leq\lambda^{N-1}$.
  \item[(iii)] Assume $|\xi_1|_r,|\xi_2|_r>\lambda$. In such a case we have 
$$\lambda^{N-2}|\xi_1|_r=\lambda^{N-2}|\xi_2|_r$$ 
from which $|\xi_1|_r=|\xi_2|_r$, and then we conclude as in case (i).
 \end{itemize}

\end{proof}
\begin{lemma}\label{14giugno_b}
Let $X$ be a  reflexive Banach space.
 Let $\Psi:X\rightarrow \mathbb R$ be convex and of class $C^1$. If $D\Psi:X\rightarrow X'$ is injective, then $\Psi^*$ is univalued and in particular differentiable, namely $\partial \Psi^*=D\Psi^*$. 
\end{lemma}
\begin{proof}
 It is well-known from convex analysis that $y=D\Psi(x)$ if and only if $x\in\partial \Psi^*(y)$. Suppose that there are two elements $x_1,x_2\in\partial \Psi^*(y)$, then $y=D\Psi(x_1)=D\Psi(x_2)$ implies $x_1=x_2$. This shows that $\partial \Psi^*$ is univalued, then $\Psi^*$ is differentiable and $D  \Psi^*= \partial \Psi^*$.
\end{proof}

\begin{lemma}\label{14giugno_c}
 Let $X$ be a  reflexive Banach space and let $\Psi:X\rightarrow \mathbb R$ be convex, of class $C^1$, with $D \Psi:X\rightarrow X'$ injective, and satisfying the condition
 \begin{equation} \label{coerccond}
\lim_{\|x\|\rightarrow\infty}\frac{\Psi(x)}{\|x\|}= +\infty. 
 \end{equation}
Let $F:X'\rightarrow \mathbb R$ be a convex function of class $C^1$ satisfying the following condition: for all $x\in X$, $y\in X'$ we have that $y=D\Psi(x)$ if and only if $x=DF(y)$. Then $F=\Psi^*+C$ for some constant $C\in\mathbb R$. 
\end{lemma}
\begin{proof}
 Since $D\Psi^*$ is univalued by Lemma \ref{14giugno_b}, it holds $y=D\Psi(x)$ if and only if $x=D\Psi^*(y)$. Therefore, if $y=D\Psi(x)$ we also have $x=DF(y)$ and thus $DF(y)=D\Psi^*(y)$. In particular we obtain that $DF$ and $D\Psi^*$ coincide on the range of $D\Psi$. We then conclude if we show that $D\Psi$ is surjective. But this follows from hypothesis \eqref{coerccond} and, e.g., \cite[Proposition 2.6, Chapter II]{Barbu}. 
\end{proof}
We are now ready to compute $\psi_\la^*$.

\begin{proposition} \label{propoconjugate}
Let $\la>0$. The conjugate function of $\psi_\la$ is 
\begin{equation} \label{conjugateofDpsila}
\psi_\la^* (y) = F_\lambda(y):= \frac{N-1}{N} \left( \vert y \vert_*^{\frac{N}{N-1}} \wedge \frac{\la^N}{\al_0^N} \right) +\frac{\al_0^{N-1}}{2 \la_0^{N-2}} \left( \vert y\vert_*^2-\frac{\la^{2N-2}}{\al_0^{2N-2}} \right)^+ 
\end{equation}
for every $y \in \Mtwo$.
\end{proposition}

\begin{proof}
From \eqref{conjugateofDpsila} it can be computed
$$ D F_\la (y)= \al_0 \left( \vert y \vert_*^{\frac{2-N}{N-1}} \wedge \frac{\la^{2-N}}{\al_0^{2-N}}  \right) \left( y + ( \tr y ) I_{2 \times 2} \right) $$
for every $y \in \Mtwo$.
 It is clear that $\psi_\la$ satisfies \eqref{coerccond}.  Let $x,y \in \Mtwo$. We want to prove that  $x=DF_\lambda (y)$ if and only if $y=D \psi_\la (x)$. Then, assume $x=DF_\lambda (y)$, namely,
\begin{equation}
x=
\begin{cases}
\al_0 \left| y \right|  _*^{\frac{2-N}{N-1}} \left( y+(\tr y) I_{2 \times 2} \right) \mbox{ if } \left| y \right|  _* \leq \frac{\la^{N-1}}{\al_0^{N-1}}, \\
\frac{\al_0^{N-1}}{\la^{N-2}} \left( y+(\tr y) I_{2 \times 2} \right), \mbox{ if } \left| y \right|  _* \geq \frac{\la^{N-1}}{\al_0^{N-1}}.
\end{cases}
\end{equation}
We take the norm $\vert \cdot \vert_r$ of this quantity. Owing to \eqref{normadualnormar} we get 
\begin{equation}
\vert x \vert_r=
\begin{cases}
\al_0 \left| y \right|  _*^{\frac{1}{N-1}}  \mbox{ if } \left| y \right|  _* \leq \frac{\la^{N-1}}{\al_0^{N-1}}, \\
\frac{\al_0^{N-1}}{\la^{N-2}} \left| y \right|  _*, \mbox{ if } \left| y \right|  _* \geq \frac{\la^{N-1}}{\al_0^{N-1}}.
\end{cases}
\end{equation}
Inverting this expression we arrive to 
\begin{equation}
\vert y \vert_*=
\begin{cases}
\frac{1}{\al_0^{N-1}} \left| x \right|  _r^{N-1}  \mbox{ if } \left| x \right|  _r \leq \la, \\
\frac{1}{\al_0^{N-1}} \la^{N-2} \left| x \right|  _r ,  \mbox{ if } \left| x \right|  _r \geq \la.
\end{cases}
\end{equation} 
It follows from \eqref{normarnormadual} that this is equivalent to $y=D \psi_\la (x)$, since 
$$a=b+(\tr b) I_{2 \times 2}\;\;\;\text{ if and only if }\;\;\;b=a-\frac13(\tr a) I_{2 \times 2},$$
for any $a,b\in\Mtwo$.
Now, thanks to Lemma \ref{14giugno_a} all the hypotheses of Lemma \ref{14giugno_c} are satisfied. Consequently  $F_\la= \psi_\la^*+C$. Since $F_\la (0)= \psi_\la^*(0)=0$, we have that $C=0$. This concludes the proof.
\end{proof}

As a consequence of Proposition \ref{propoconjugate} we deduce that the conjugate function of $\Psi_\lambda$ in \eqref{PhiPsi2} is
\begin{align}
 \Psi_\lambda^*(\eta):=\int_\Om F_\lambda(\eta(x))dx\;\;\;\;\;\mbox{ for every } \eta \in L^2(\Om; \Mtwo).
\end{align}

\section{The dynamic evolution problem: regularity} \label{sec:problem}

In this section we introduce the approximate problem to the dynamic model. We prove some preliminary lemmas before Theorem \ref{TheoremNHoff} which states the existence of a solution to the Norton--Hoff approximation. Then in Proposition \ref{propositionregular} we study the regularity of the obtained solutions.

The Dirichlet datum  of the problem is realized by a prescribed boundary displacement $w$ with the following regularity
\begin{align} \label{regularityw}
 &w \in W^{2,1}([0,T]; H^1(\Om; \R^3) \cap KL (\Om)) \cap W^{1,1}([0,T]; H^2_{\text{loc}}(\om\times[-\frac12,\frac12];\R^3)),\nonumber\\
 &\mbox{ with } w_3 \in W^{3,1}([0,T];L^2(\om))\cap W^{2,1}([0,T];H^1_{\text{loc}}(\om)).
\end{align} 
Here the space $H^2_{\text{loc}}(\om\times[-\frac12,\frac12];\R^3)$ is  the space of maps $w\in H^1(\Om;\R^3)$ such that for all $\om'\subset\subset\om$ the function $w$ belongs to $ H^2(\om'\times (-\frac12,\frac12);\R^3)$.

The total external load is a function $\mathcal L:[0,T]\rightarrow L^2(\Om;\R^3)$ which we decompose as horizontal and vertical forces $f$ and $g$, namely 
$$\langle \mathcal L(t),\varphi\rangle :=\langle f(t),\bar\varphi\rangle+\langle g(t), \varphi_3\rangle.  $$
We assume that the external forces satisfy the following uniform safe-load condition. Namely, we suppose the existence of $\varrho:[0,T]\rightarrow L^2(\Om;\Mtwo)$ such that
\begin{align}\label{regularityvarrho}
 \varrho\in W^{2,\infty}([0,T];L^\infty(\Om;\Mtwo))\cap L^2([0,T]; W^{2,\infty}_\text{loc}(\om\times[-\frac12,\frac12];\Mtwo)),
\end{align}
and satisfying, for all $t\in[0,T]$, 
\begin{align}\label{regularityvarrho*}
&-\Div \bar\varrho(t)=f(t),\;\;\;\;\;\;\;-\frac{1}{12}\Div\Div \hat\varrho(t)=g(t),\nonumber\\
& |\varrho(t)|_r\leq \al_0(1-\gamma),
\end{align}
for some $\gamma>0$ fixed and independent of $t$. The space $W^{2,\infty}_\text{loc}(\om\times[-\frac12,\frac12];\Mtwo)$ is the subspace of $ W^{1,\infty}(\Om;\Mtwo)$ whose elements belong to $ W^{2,\infty}(\om'\times (-\frac12,\frac12);\Mtwo)$ for all $\om'\subset\subset\om$. Finally we make the following technical assumption on the external vertical force $g$, namely
\begin{align}\label{regularityg}
g \in W^{1,1}([0,T]; L^2(\om)).
\end{align}
Notice that such condition does not follow from the safe-load condition.  We assume for simplicity that there are no external loads on the Neumann boundary $\Gamma_n$.

Finally let us suppose that the initial data $u_0$, $\sigma_0$, and $(v_0)_3$, satisfy the following properties
\begin{align}\label{initialdata}
 &u_0 \in H^1(\Om; \R^3)\cap KL(\Om),\nonumber\\
 &\sigma_0\in L^\infty(\Om; \Mtwo) \text{ is such that }-\Div \bar{\sigma}_0= f(0)\;\text{ and }\; -\frac{1}{12}\Div \Div \hat{\sigma}_0= g(0) \mbox{ in } \om, \nonumber\\
&\sigma_0 \in \mathcal{K}_r(\Om) \cap \Theta(\gn,0,0), \text{ and }\nonumber\\
 &\text{there exists } \hat v_0 \in H^1(\Om;\R^3)\cap KL(\Om)\text{ such that }(v_0)_3=(\hat v_0)_3.
\end{align}
We will denote $ \hat v_0$ simply by $v_0$. Moreover we set $p_0:=Eu_0-\mathbb A_r\sigma_0\in L^2(\Om;\Mtwo)$.

Now we can state and prove the following result.
\begin{lemma} \label{Lemmalambda}
Let $T>0$, $K_r$ be of the form \eqref{daKaKridotto}, assume \eqref{regularityw}, and  that  $f$ and $g$ satisfy \eqref{regularityvarrho}-\eqref{regularityg}. Assume that the initial conditions satisfy \eqref{initialdata}. 
 Then for every integer $N \geq 4$ and $\lambda>0$ the problem
\begin{equation}
\begin{cases} \label{NHLambdaproblem}
\mathbb{A}_r \dot{\sigma} (t)+D \psi_\la(\sigma(t))= E \dot{u} (t) \mbox{ in } \Om, \\
-\Div \bar{\sigma}(t)= f(t), \quad \ddot{u}_3(t)-\frac{1}{12} \Div \Div \hat{\sigma}(t)= g(t) \mbox{ in } \om, \\
\sigma(t) \in \Theta (\gn,0,0), \\
u(t)=w(t) \mbox{ on } \Gd.
\end{cases}
\end{equation}
for a.e. $t\in[0,T]$, has a unique solution 
\begin{align*}
(\sigma^\la, u^\la) \in  H^1([0,T];L^2(\Om; \Mtwo)) \times 
  H^1([0,T]; H^{1}(\Om; \R^3)\cap KL(\Om)),
\end{align*}
with
\begin{equation*}
u_3^\la \in  H^2([0,T]; L^2 (\om)),
\end{equation*}
such that $$(u^\la(0),\sigma^\la(0))=(u_0,\sigma_0) \text{ and }\dot{u}_3^\la(0)=(v_0)_3.$$
Moreover the following estimate holds true
\begin{align} 
&\Vert \dot{u}_3^\la \Vert_{L^\infty(L^2)}^2+\Vert \sigma^\la  \Vert_{L^\infty(L^2)}^2+ \Vert D \psi_\la ( \sigma^\la)  \Vert_{L^{N'}(L^{N'})}^{N'}\leq C,\label{firstlambdaestimate} 
\end{align}
for a constant $C>0$ independent of $\la$ and $N$.
\end{lemma}

Moreover we can prove the following additional regularity of solution of the problem \eqref{NHLambdaproblem}.

\begin{lemma}
Let $(u^\la,\sigma^\la,p^\la)$ be the solution of Lemma \ref{Lemmalambda}. Then there exists a constant $C>0$ independent of $\la$ and $N$ such that 
\begin{align} 
&\Vert \dot{\sigma}^\la \Vert_{L^2 (L^2)}^2+ \Vert \ddot{u}_3^\la  \Vert_{L^2 (L^2)}^2 \leq C\label{secondlambdaestimate}.
\end{align}
Moreover, if we assume that  the function $f$ is independent of time $t$, then 
$$(\sigma^\la, u_3^\la) \in  W^{1,\infty}([0,T];L^2(\Om; \Mtwo)) \times  W^{2,\infty}([0,T]; L^2 (\om))$$ and
there exists a constant $C>0$ independent of $\la$ and $N$ such that 
\begin{align} 
&\Vert \dot{\sigma}^\la \Vert_{L^\infty (L^2)}^2+ \Vert \ddot{u}_3^\la  \Vert_{L^\infty (L^2)}^2\leq C\label{secondlambdaestimate*}
\end{align}
\end{lemma}

\begin{remark}
 From estimate \eqref{firstlambdaestimate} it is possible to show that 
 $$\|D\psi_\la(\sigma^\la)\|_{L^1(L^1)}\leq C,$$
 for a constant $C>0$ independent of $\la$ and $N$. Coupling this with the estimate for $\dot\sigma^\la$ in \eqref{secondlambdaestimate} we infer
 that 
 \begin{align}
 \|u^\la\|_{W^{1,1}(BD)}\leq C, 
 \end{align}
 for a constant $C>0$ independent of $\la$ and $N$. Using the fact that $u^\la$ is a Kirchhoff--Love function we conclude
 \begin{align}\label{sanvalentino}
  \|u^\la\|_{W^{1,1}(L^2)}\leq C.
 \end{align}

\end{remark}
We will prove the two lemmas in a unique proof.

\begin{proof}

We proceed in six steps and we use a standard time discretization technique, with the aid of an implicit Euler scheme.

\medbreak
\noindent\textit{Step 1: Time discretization.} For every integer $k>0$ we consider a partition of the time interval $[0,T]$ into $k$ subintervals of equal length $\delta_k\doteq \frac{T}{k}$, i.e.,    $$t_k^i \doteq i \delta_k  \mbox{ for every } i=0, \dots, k.$$ 
We define 
\begin{align}\label{in.data}
 &(u_k^{-1},\sigma_k^{-1},p_k^{-1}) \doteq (u_0-\delta_k v_0,\sigma_0-\delta_k \mathbb C_r (E v_0-D\psi_\la(\sigma_0)),p_0-\delta_kD\psi_\la(\sigma_0)),\nonumber\\
&u_k^{-2}\doteq u_k^{-1}-\delta_k v_0,\nonumber\\
 &(u_k^0,\sigma_k^0,p_k^0)=(u_0, \sigma_0, p_0).
\end{align}
For $i= 1,\dots,k$ the triplet
$(u_k^i,\sigma_k^i,p_k^i)$ is defined recursively as the solution of the minimum problem 
\begin{equation} \label{minimumproblem}  
\min_{(u, \sigma,p) \in \mathcal{A}_{reg}(w_k^i) } \mathcal{F}_i(u,\sigma,p),
\end{equation} 
where 
\begin{align*}  
\mathcal{F}_i(u,\sigma,p) \doteq  \frac{1}{2}   \Big{\Vert} \frac{ u_3-2 (u_3)_k^{i-1}+ (u_3)_k^{i-2}}{\delta_k} \Big{\Vert}_{L^2}^2+\frac{1}{2}\langle \mathbb{A}_r \sigma, \sigma \rangle +\delta_k \Psi_\la^* \left( \frac{p-p_k^{i-1}}{\delta_k} \right) -\langle \mathcal L_k^i,u\rangle, 
\end{align*}
and
\begin{align*}
\mathcal{A}_{reg}(w)  \doteq & \lbrace (u,\sigma,p) \in (H^{1}(\Om; \R^3) \cap KL(\Om)) \times L^2(\Om; \Mtwo) \times L^{2}(\Om; \Mtwo): \\  &    Eu=\mathbb{A}_r \sigma+p \mbox{ in } \Om, \quad u=w \mbox{ on } \Gd \rbrace.   
\end{align*}
  In \eqref{minimumproblem} we have defined $w_k^i:=w(t_k^i)$ for $i=0,\dots,k$, and 
\begin{align}
 \langle \mathcal L_k^i,u\rangle=\langle f_k^i,\bar u\rangle +\langle g_k^i,u_3\rangle,
\end{align}
where $f_k^i:=f(t_k^i)$, $g_k^i:=g(t_k^i)$.
We also set $w^{-1}_k:=2w^{0}_k-w^{1}_k$ and $w^{-2}_k:=2w^{-1}_k-w^{0}_k$.

It follows from Proposition \ref{propoconjugate} that $\Psi_\la^*$ is coercive with respect to the $L^2-$norm and has more than linear growth. This fact, together with \eqref{Acoercivity}, implies that $\mathcal{F}_i$ is coercive, lower semicontinuous and strictly convex. Then  Korn inequality in $H^1$ ensures the existence and uniqueness of a solution for the problem \eqref{minimumproblem}. 

\medbreak
\noindent\textit{Step 2: Euler-Lagrange equations.} For all $ i =0,\dots,k $ we claim that the solution 
$$(u_k^i, \sigma_k^i, p_k^i) \in \mathcal{A}_{reg} (w_k^i)$$ 
of \eqref{minimumproblem}  satisfies the following system:
\begin{equation} \label{discrete problem}
\begin{cases}
\frac{Eu_k^i-E u_k^{i-1}}{\delta_k}=\frac{e_k^i-e_k^{i-1}}{\delta_k}+\frac{p_k^i-p_k^{i-1}}{\delta_k} \mbox{ in } \Om, \\ 
-\Div \bar{\sigma}_k^i= f^i_k, \quad \frac{(u_3)_k^i-2(u_3)_k^{i-1}+(u_3)_k^{i-2}}{\delta_k^2}-\frac{1}{12} \Div \Div \hat{\sigma}_k^i= g^i_k \mbox{ in } \om, \\
u_k^i=w_k^i \mbox{ on } \Gd, \\
\sigma_k^i \in \Theta(\gn, 0,0), \\
\frac{p_k^i-p_k^{i-1}}{\delta_k}=D \psi_\la (\sigma_k^i),
\end{cases}
\end{equation}
where $e_k^i \doteq \mathbb{A}_r \sigma_k^i$. 

The system \eqref{discrete problem} is satisfied by definition for the index $i=0$. Moreover, by admissibility, $ E u_k^i=e_k^i+p_k^i$, $ E u_k^{i-1}=e_k^{i-1}+p_k^{i-1}$ and the first line in \eqref{discrete problem} follows.

Let $\eps \in (-1,1)$ and let $\varphi \in H^1(\Om; \R^3) \cap KL(\Om)$ with $\varphi=0$ on $\Gd$. Clearly $(u_k^i+\eps \varphi, \sigma_k^i+\eps \mathbb{C}_r E \varphi, p_k^i) \in \mathcal{A}_{reg} (w_k^i) $, so that by minimality of $(u_k^i, \sigma_k^i, p_k^i)$ we can differentiate the energy $\mathcal F_i(u_k^i+\eps \varphi, \sigma_k^i+\eps \mathbb{C}_r E \varphi, p_k^i)$ in $\eps=0$ obtaining 
\begin{equation} \label{variationalmotion}
\langle \frac{(u_3)_k^i-2(u_3)_k^{i-1}+(u_3)_k^{i-2}}{\delta_k^2}, \varphi_3 \rangle +\langle \sigma_k^i, E \varphi \rangle=\langle \mathcal L_k^i,\varphi \rangle.
\end{equation}
This, by arbitrariness of $\varphi$, entails
\begin{equation*} 
\begin{cases}
-\Div \bar{\sigma}_k^i= f^i_k, \quad 
\frac{(u_3)_k^i-2(u_3)_k^{i-1}+(u_3)_k^{i-2}}{\delta_k^2}-\frac{1}{12} \Div \Div \hat{\sigma}_k^i= g^i_k \mbox{ in } \om, \\
u_k^i=w_k^i \mbox{ on } \Gd, \\
\sigma_k^i \in \Theta(\gn, 0,0),
\end{cases}
\end{equation*}
where the Neumann boundary conditions derive from \eqref{traccia0mom} and \eqref{traccia1mom}.
It remains to prove the last condition in  \eqref{discrete problem}.
Let $\eps \in (-1,1)$ and let $\eta \in L^2(\Om;\Mtwo)$. We have 
$$(u_k^i,\sigma_k^i-\eps \mathbb{C}_r(\eta-p_k^i+p_k^{i-1}),p_k^i+\eps (\eta-p_k^i+p_k^{i-1})) \in \mathcal{A}_{reg} (w_k^i),$$ 
so that  we can compare the values of $\mathcal F_i$ at this point and  at the  minimum point $(u_k^i, \sigma_k^i, p_k^i)$, obtaining
\begin{align*}
&\frac{\eps}{\delta_k} \langle \mathbb{A}_r\sigma_k^i,  \mathbb{C}_r(\eta-p_k^i+p_k^{i-1})\rangle -\frac{\eps^2}{2\delta_k} \langle \mathbb{A}_r( \eta-p_k^i+p_k^{i-1}),  \mathbb{C}_r(\eta-p_k^i+p_k^{i-1}) \rangle  \\
&\leq  \Psi_\la^* \left( \frac{p_k^i-p_k^{i-1}+ \eps(\eta-p_k^i+p_k^{i-1}) }{\delta_k} \right)-  \Psi_\la^* \left( \frac{p_k^i-p_k^{i-1} }{\delta_k} \right)\\
&\leq \eps\Big( \Psi_\la^* \left( \frac{\eta }{\delta_k} \right)-\Psi_\la^* \left( \frac{p_k^i-p_k^{i-1} }{\delta_k} \right)\Big),
\end{align*}
the last inequality following by convexity.
From this,  dividing by $\eps$ and letting $\eps \to 0^+$, we get
\begin{equation} \label{differentialinclusion}
\langle \sigma_k^i, \frac{\eta-p_k^i+p_k^{i-1}}{\delta_k} \rangle \leq \Psi_\la^* \left( \frac{\eta }{\delta_k} \right)-\Psi_\la^* \left( \frac{p_k^i-p_k^{i-1} }{\delta_k} \right), 
\end{equation}
where we also used that
$$ \langle \mathbb{A}_r \sigma, \mathbb{C}_r \eta \rangle=\langle \sigma, \eta \rangle \mbox{ for every } \sigma, \eta \in \Mtwo. $$
By arbitrariness of $\eta$ expression \eqref{differentialinclusion} leads to
$$ \sigma_k^i : \frac{\eta-p_k^i+p_k^{i-1}}{\delta_k}  \leq \psi_\la^* \left( \frac{\eta }{\delta_k} \right)-\psi_\la^* \left( \frac{p_k^i-p_k^{i-1} }{\delta_k} \right)\mbox{ a.e.  in } \Om,$$
that is
$$ \sigma_k^i \in \partial  \psi_\la^* \left( \frac{p_k^i-p_k^{i-1}}{\delta_k} \right) \mbox{ a.e.  in } \Om. $$
This fact together with Lemma \ref{14giugno_a} and Lemma \ref{14giugno_b} yields
\begin{equation} \label{NHinversediscreteflowrule}
\sigma_k^i=D \psi_\la^* \left( \frac{p_k^i-p_k^{i-1}}{\delta_k} \right) \mbox{ a.e. in } \Om,
\end{equation}
which is equivalent to the last condition in \eqref{discrete problem}.

\medbreak
\noindent\textit{Step 3: Compactness estimates.} 

For every $i=-2, \dots, k$ we set
$$ \om_k^i \doteq \frac{w_k^{i+1}-w_k^i}{\delta_k}, \quad (v_3)_k^i \doteq \frac{(u_3)_k^{i+1}-(u_3)_k^i}{\delta_k}.  $$
We define two types of interpolations. The piecewise constant interpolations $(u_k,e_k,p_k):[0,T]\rightarrow (H^{1}(\Om; \R^3) \cap KL(\Om)) \times L^2(\Om; \Mtwo) \times L^{2}(\Om; \Mtwo)$ are  given by
$$ (u_k(0),e_k(0),p_k(0)) \doteq (u_0,e_0,p_0) $$
$$ u_k(t) \doteq u_k^{i+1} \quad \sigma_k(t) \doteq e_k^{i+1} \quad p_k(t) \doteq p_k^{i+1} \mbox{ for } t \in (t_k^{i},t_k^{i+1}], $$
and similarly are defined $\om_k$ and $(v_3)_k$. In the same way we denote by $g_k$ and $f_k$ the piecewise constant function defined as
$$ f_k(t) \doteq f_k^{i+1},\;\;  g_k(t) \doteq g_k^{i+1}\mbox{ for } t \in (t_k^{i},t_k^{i+1}],$$
and similarly for $\varrho_k$.
Setting $f_k^{-1}:=f^0_k$, and analogously $g_k^{-1}$ and $\varrho_k^{-1}$, we see that the expression above makes sense also for $i=-1$.

The piecewise affine interpolations are instead
$$ \tilde{u}_k(t)=u_k^{i}+\frac{t-t_k^{i}}{\delta_k}(u_k^{i+1}-u_k^{i}) \mbox{ for } t \in [t_k^{i},t_k^{i+1}], $$
and analogous expressions for $\tilde w_k$, $\tilde{\sigma}_k$, $\tilde{p}_k$, $ (\tilde v_3)_k$, and $(\tilde \omega_3)_k$, $\tilde g_k$, $\tilde f_k$, and $\tilde\varrho_k$. Notice that, thanks to \eqref{in.data}, setting $t_k^{-1}=-\delta_k$, all the functions introduced so far are naturally defined on the interval $[-\delta_k,T]$ and it holds
$ (v_3)_k=(\dot{\tilde{u}}_3)_k$ and $ (\omega_3)_k=(\dot{\tilde{w}}_3)_k$ on $[-\delta_k,T]$. 

Let us show that there exists a constant $C>0$ independent of $k$, $\la$, and $N$, such that
\begin{align} 
&\Vert (\dot{\tilde{u}}_3)_k  \Vert_{L^\infty(L^2)}^2+\Vert \sigma_k  \Vert_{L^\infty(L^2)}^2+\Vert D \psi_\la ( \sigma_k)  \Vert_{L^{N'}(L^{N'})}^{N'} \leq C,\label{primastimapriori}
\end{align}
and 
\begin{align} 
&\Vert \dot{\tilde{\sigma}}_k \Vert_{L^2 (L^2)}^2+ \Vert (\dot{\tilde{v}}_3)_k  \Vert_{L^2 (L^2)}^2  \leq C \label{secondprioriestimate}.
\end{align}
To prove \eqref{primastimapriori} we test \eqref{variationalmotion} by $\varphi=u_k^i-u_k^{i-1}-w_k^i+w_k^{i-1}$. Then, owing to the first and last conditions in \eqref{discrete problem}, we infer
\begin{align}\label{numero}
& \frac{1}{2} \Big{\Vert} \frac{(u_3)_k^i-(u_3)_k^{i-1}}{\delta_k} \Big{\Vert}_{L^2}^2 -\frac{1}{2} \Big{\Vert} \frac{(u_3)_k^{i-1}-(u_3)_k^{i-2}}{\delta_k} \Big{\Vert}_{L^2}^2+\frac{1}{2} \langle \mathbb{A}_r \sigma_k^i, \sigma_k^i \rangle \nonumber\\
&- \frac{1}{2} \langle \mathbb{A}_r \sigma_k^{i-1}, \sigma_k^{i-1} \rangle+ \delta_k \langle  \sigma_k^i, D \psi_\la (\sigma_k^i) \rangle\nonumber \\ 
&   \leq \langle \frac{(u_3)_k^i-2(u_3)_k^{i-1}+(u_3)_k^{i-2}}{\delta_k^2}, (w_3)_k^i-(w_3)_k^{i-1} \rangle+  \langle \sigma_k^i,  E w_k^i-E w_k^{i-1} \rangle\nonumber\\
&+\langle g^i_k,(u_3)_k^i-(u_3)_k^{i-1}-(w_3)_k^i+(w_3)_k^{i-1}\rangle +\langle f^i_k,\bar u_k^i-\bar u_k^{i-1}-\bar w_k^i+\bar w_k^{i-1} \rangle. 
\end{align}
We now write 
\begin{align*}
 &\langle f^i_k,\bar u_k^i-\bar u_k^{i-1}-\bar w_k^i+\bar w_k^{i-1}\rangle=\langle \bar\varrho_k^i,E\bar u_k^i-E\bar u_k^{i-1}-E\bar w_k^i+E\bar w_k^{i-1}\rangle\\
 &=\langle \bar\varrho_k^i,e_k^i-e_k^{i-1}\rangle +\langle \bar\varrho_k^i,p_k^i-p_k^{i-1}\rangle -\langle \bar\varrho_k^i,E\bar w_k^i-E\bar w_k^{i-1}\rangle
\end{align*}
and, in turn,
\begin{align*}
 &\langle \bar\varrho_k^i,e_k^i-e_k^{i-1}\rangle=\langle \bar\varrho_k^i,e_k^i\rangle-\langle \bar\varrho_k^{i-1},e_k^{i-1}\rangle-\langle \bar\varrho_k^i-\bar\varrho_k^{i-1},e_k^{i-1}\rangle,\\
 &\langle \bar\varrho_k^i,p_k^i-p_k^{i-1}\rangle =\delta_k\langle \bar\varrho_k^i,D \psi_\la (\sigma^i_k)\rangle.
\end{align*}

For $j\geq2$, summing expression \eqref{numero} on $i=1,\dots, j$ and rearranging terms we get
\begin{align*}
& \frac{1}{2} \Big{\Vert} \frac{(u_3)_k^j-(u_3)_k^{j-1}}{\delta_k} \Big{\Vert}_{L^2}^2+\frac{1}{2} \langle \mathbb{A}_r \sigma_k^j, \sigma_k^j \rangle+ \frac{\delta_k}{\al_0^{N-1}} \sum_{i=1}^j  \int_\Om  ( \vert \sigma_k^i\vert_r^{N-2} \wedge \la^{N-2}) \vert \sigma_k^i \vert_r^2 dx \\
& \leq \frac{1}{2} \Vert ( v_0)_3 \Vert_{L^2}^{2}+\frac{1}{2} \langle \mathbb{A}_r \sigma_0, \sigma_0 \rangle+ \langle \frac{(u_3)_k^j-(u_3)_k^{j-1}}{\delta_k}, \frac{(w_3)_k^j-(w_3)_k^{j-1}}{\delta_k} \rangle-\langle(v_0)_3,(\omega_0)_3\rangle \\
&-\sum_{i=1}^{j} \langle (u_3)_k^{i-1}-(u_3)_k^{i-2},  \frac{(w_3)_k^i-2(w_3)_k^{i-1}+(w_3)_k^{i-2}}{\delta_k^2}  \rangle +\sum_{i=1}^j\langle (u_3)_k^{i}-(u_3)_k^{i-1},g_k^i\rangle \\
&+\langle \bar\varrho_k^j,e_k^j\rangle-\langle \bar\varrho_k^{0},e_k^{0}\rangle-\sum_{i=1}^j\langle \bar\varrho_k^i-\bar\varrho_k^{i-1},e_k^{i-1}\rangle+\delta_k\sum_{i=1}^j\langle \bar\varrho_k^i,D \psi_\la (\sigma^i_k)\rangle\\
&+\sum_{i=1}^j  \langle\sigma_k^i, E w_k^i-E w_k^{i-1}\rangle -\sum_{i=1}^j \langle g^i_k,(w_3)_k^i-(w_3)_k^{i-1}\rangle-\sum_{i=1}^j\langle \bar\varrho_k^i,E\bar w_k^i-E\bar w_k^{i-1}\rangle, 
\end{align*}
where we have used that $w^{-2}_k=w^{-1}_k-\delta_k(\omega_0)_3$.
From \eqref{DpsilambdaNN-1} it follows that
$$ \int_\Om \vert D \psi_\la (\sigma_k(t,x)) \vert^{N'}dx \leq \frac{1}{\al_0^N} \int_\Om \left( \vert \sigma_k(t,x) \vert_r^{N-2} \wedge \la^{N-2} \right) \vert \sigma_k (t,x) \vert_r^2 dx. $$
Consequently,  we rewrite the previous expression as
\begin{align*}
& \frac{1}{2\alpha_0} \Big{\Vert} (\dot{\tilde u}_3)_k(t_k^{j-1})\Big{\Vert}_{L^2}^2+\frac{1}{2\alpha_0} \langle \mathbb{A}_r \sigma_k^j, \sigma_k^j \rangle+\int_0^{t_k^{j-1}} \int_\Om  \vert D \psi_\la (\sigma_k(t,x)) \vert^{N'} dx dt \\
& \leq \frac{1}{2\alpha_0} \Vert ( v_0)_3 \Vert_{L^2}^{2}+\frac{1}{2\alpha_0} \langle \mathbb{A}_r \sigma_0, \sigma_0 \rangle+ \frac{1}{\alpha_0}\langle (\dot{\tilde u}_3)_k(t_k^{j-1}),(\dot{\tilde w}_3)_k(t_k^{j-1}) \rangle-\frac{1}{\alpha_0}\langle(v_0)_3,(\omega_0)_3\rangle  \\
&-\frac{1}{\alpha_0}\int_0^{t_k^{j-2}}\langle (\dot{\tilde u}_3)_k(t),  (\dot{\tilde \omega}_3)_k(t-\delta_k)  \rangle dt-\frac{1}{\alpha_0}\int_0^{t_k^j}\langle(\dot{\tilde u}_3)_k(t),g_k(t)\rangle dt  \\
&+\frac{1}{\alpha_0}\langle \mathbb A_r\bar\varrho_k^j,\sigma_k^j\rangle-\frac{1}{\alpha_0}\langle \mathbb A_r\bar\varrho_k^{0},\sigma_k^{0}\rangle-\frac{1}{\alpha_0}\int_0^{t_k^j}\langle \mathbb A_r\dot{\bar\varrho}_k(t),\sigma_k(t)\rangle dt+\int_0^{t_k^j}\langle \frac{\bar\varrho_k(t)}{\alpha_0},D \psi_\la (\sigma_k(t))\rangle dt\\
 &+\frac{1}{\alpha_0}\int_0^{t_k^j} \langle \sigma_k(t),  E \dot{\tilde w}_k(t-\delta_k)\rangle dt +\frac{1}{\alpha_0}\int_0^{t_k^j} \langle g_k(t),(\dot{\tilde w}_3)(t-\delta_k)\rangle dt-\frac{1}{\alpha_0}\int_0^{t_k^j}\langle \bar\varrho_k(t),\partial_tE\bar{\tilde w}_k(t-\delta_k)\rangle dt . 
\end{align*}
Standard applications of H\"older and Young inequalities, together with the Gronwall Lemma gives the following estimates 
\begin{align}
 \Vert (\dot{\tilde{u}}_3)_k  \Vert_{L^\infty(L^2)}^2+\Vert \sigma_k  \Vert_{L^\infty(L^2)}^2+  \int_{\delta_k}^T \int_\Om  \vert D \psi_\la (\sigma_k (t,x)) \vert^{N'} dxdt\leq C,
\end{align}
where the constant $C>0$ depends on the following quantities
\begin{align*}
 \|(v_0)_3\|_{L^2}, \;\|\sigma_0\|_{L^2},\;\|\ddot{w}_3\|_{L^1(L^2)},\;\;\|\dot{w}_3\|_{L^\infty(L^2)},\;\;\|g\|_{L^1(L^2)},\;\|\bar{\varrho}\|_{W^{1,1}(L^2)},\;\|E\dot{ w}\|_{L^1(L^2)},
\end{align*}
and is independent of  $N$, $k$, and $\la$. 
To obtain the previous estimate we have treated the term $\int_0^{t_k^j}\langle \bar\varrho_k,D \psi_\la (\sigma_k)\rangle dt$ as follows.
We have used the fact, ensured by the safe-load condition \eqref{regularityvarrho*}, that $|\frac{\bar\varrho_k}{\alpha_0}|\leq|\frac{\bar\varrho_k}{\alpha_0}|_r\leq(1-\gamma)$ so that 
\begin{align}
 &\int_0^{t_k^j}\langle \frac{\bar\varrho_k}{\alpha_0},D \psi_\la (\sigma_k)\rangle dt\leq(1-\gamma)\int_0^{t_k^j}\int_\Om|D \psi_\la (\sigma_k)|dx dt\nonumber\\&\leq\frac{(1-\gamma)}{N'}\int_0^{t_k^{j-1}} \int_\Om  \vert D \psi_\la (\sigma_k(t,x)) \vert^{N'} dx dt+\frac{(1-\gamma)|\Om|T}{N}\nonumber\\
 &\leq (1-\gamma)\int_0^{t_k^{j-1}} \int_\Om  \vert D \psi_\la (\sigma_k(t,x)) \vert^{N'} dx dt+C.
\end{align}

Now let us prove \eqref{secondprioriestimate}. Subtracting the first expression in \eqref{discrete problem} at step $i$ with the one at step $i-1$,  dividing by $\delta_k$, and then testing the result with $\sigma_k^i-\sigma_k^{i-1}$, we infer 
\begin{align} 
&\langle \mathbb{A}_r \frac{\sigma_k^i-2\sigma_k^{i-1}+\sigma_k^{i-2}}{\delta_k^2}, \sigma_k^i-\sigma_k^{i-1} \rangle+\langle  \frac{p_k^i-2p_k^{i-1}+p_k^{i-2}}{\delta_k^2}, \sigma_k^i-\sigma_k^{i-1} \rangle  \notag \\ \label{kinederivata}
&=\langle \frac{E u_k^i-2 E u_k^{i-1}+E u_k^{i-2}}{\delta_k^2}, \sigma_k^i-\sigma_k^{i-1} \rangle
\end{align} 
The first term in the left-hand side of \eqref{kinederivata} can be estimated from below by
$$ \frac{1}{2} \langle \mathbb{A}_r \frac{\sigma_k^{i}-\sigma_k^{i-1}}{\delta_k}, \frac{\sigma_k^{i}-\sigma_k^{i-1}}{\delta_k} \rangle- \frac{1}{2} \langle \mathbb{A}_r \frac{\sigma_k^{i-1}-\sigma_k^{i-2}}{\delta_k}, \frac{\sigma_k^{i-1}-\sigma_k^{i-2}}{\delta_k} \rangle. $$
Using the last condition in \eqref{discrete problem}, we note that the second term in the left-hand side of \eqref{kinederivata}  is equal to
$$ \langle  \frac{D \psi_\la (\sigma_k^i)-D \psi_\la (\sigma_k^{i-1})}{\delta_k}, \sigma_k^i-\sigma_k^{i-1} \rangle, $$
which is nonnegative, since  the differential of a convex function is a monotone operator.

From  \eqref{variationalmotion} we can write the right-hand side of \eqref{kinederivata} as 
\begin{align*}
&\langle \sigma_k^i-\sigma_k^{i-1}, \frac{E \om_k^{i-1}-E \om_k^{i-2}}{\delta_k} \rangle+\langle \frac{E\bar u_k^i-2 E\bar u_k^{i-1}+E\bar u_k^{i-2}}{\delta_k^2}-\frac{E\bar w_k^i-2 E\bar w_k^{i-1}+E\bar w_k^{i-2}}{\delta_k^2}, \bar\sigma_k^i-\bar\sigma_k^{i-1} \rangle \\
&-\frac{1}{12} \langle \hat{\sigma}_k^i-\hat{\sigma}_k^{i-1}, \frac{ D^2 (v_3)_k^{i-1}-D^2 (v_3)_k^{i-2}}{\delta_k}- \frac{ D^2 (\om_3)_k^{i-1}-D^2 (\om_3)_k^{i-2}}{\delta_k} \rangle \\
&=\langle \sigma_k^i-\sigma_k^{i-1}, \frac{E \om_k^{i-1}-E \om_k^{i-2}}{\delta_k} \rangle+\langle \frac{E\bar u_k^i-2 E\bar u_k^{i-1}+E\bar u_k^{i-2}}{\delta_k^2}-\frac{E\bar w_k^i-2 E\bar w_k^{i-1}+E\bar w_k^{i-2}}{\delta_k^2}, \bar\varrho_k^i-\bar\varrho_k^{i-1} \rangle  \\
&- \langle \frac{(v_3)_k^{i-1}-2(v_3)_k^{i-2}+(v_3)_k^{i-3}-(\om_3)_k^{i-1}+2(\om_3)_k^{i-2}-(\om_3)_k^{i-3}}{\delta_k}, \\
& \frac{  (v_3)_k^{i-1}- (v_3)_k^{i-2}}{\delta_k}- \frac{  (\om_3)_k^{i-1}- (\om_3)_k^{i-2}}{\delta_k} \rangle \\
&- \langle \frac{(\om_3)_k^{i-1}-2(\om_3)_k^{i-2}+(\om_3)_k^{i-3}}{\delta_k}-\frac{g_k^i-g_k^{i-1}}{\delta_k}, 
\frac{  (v_3)_k^{i-1}- (v_3)_k^{i-2}}{\delta_k}- \frac{  (\om_3)_k^{i-1}- (\om_3)_k^{i-2}}{\delta_k} \rangle.
\end{align*} 
Let us first consider the case that $f$ is independent of $t$. This implies that the second term of the right-hand side is null. Using this expression, let $j\geq2$ and sum \eqref{kinederivata} on  $i=1, \dots, j$. We get
\begin{align}\label{f_ind_time}
&\Big{\Vert} \frac{\sigma_k^j-\sigma_k^{j-1}}{\delta_k} \Big{\Vert}_{L^2}^2+ \frac{1}{2} \Big{\Vert} \frac{  (v_3)_k^{j-1}- (v_3)_k^{j-2}-(\om_3)_k^{j-1}+ (\om_3)_k^{j-2}}{\delta_k} \Big{\Vert}_{L^2}^{2}\nonumber \\
&\leq C+ \sum_{i=1}^j \langle \frac{\sigma_k^i-\sigma_k^{i-1}}{\delta_k}  ,  \frac{E \om_k^{i-1}-E \om_k^{i-2}}{\delta_k} \rangle\nonumber \\
&- \sum_{i=1}^j  \langle \frac{(\om_3)_k^{i-1}-2(\om_3)_k^{i-2}+(\om_3)_k^{i-3}}{\delta_k}-\frac{g_k^i-g_k^{i-1}}{\delta_k}, 
\frac{  (v_3)_k^{i-1}- (v_3)_k^{i-2}}{\delta_k}- \frac{  (\om_3)_k^{i-1}- (\om_3)_k^{i-2}}{\delta_k} \rangle,
\end{align}
where we have used $(v_3)_k^{0}=(v_3)_k^{-1}$ and $(\omega_3)_k^{0}=(\omega_3)_k^{-1}$, and $C=\| \frac{\sigma_k^0-\sigma_k^{-1}}{\delta_k} \|_{L^2}^2=\|\mathbb C_r (Ev_0-D\psi_\la(\sigma_0))\|_{L^2}^2$.
This leads to the estimate
\begin{align} 
&\Vert \dot{\tilde{\sigma}}_k \Vert_{L^\infty (L^2)}^2+ \Vert (\dot{\tilde{v}}_3)_k  \Vert_{L^\infty (L^2)}^2  \leq C \label{thirdprioriestimate},
\end{align}
where the constant is independent of $k$, $\la$, and $N$, and depends on 
\begin{align}
\|E\ddot{w}\|_{L^1(L^2)},\;\;\|\dddot w_3 \|_{L^1(L^2)},\;\;\|\dot g\|_{L^1(L^2)}.
\end{align}
Let us now consider the general case in which $f$ might depend on time $t$. This means that in the right-hand side of \eqref{f_ind_time} there is the additional term 
\begin{align}
 &\sum_{i=1}^j\langle \frac{E\bar u_k^i-2 E\bar u_k^{i-1}+E\bar u_k^{i-2}}{\delta_k^2}-\frac{E\bar w_k^i-2 E\bar w_k^{i-1}+E\bar w_k^{i-2}}{\delta_k^2}, \bar\varrho_k^i-\bar\varrho_k^{i-1} \rangle=\nonumber\\
 &=\langle E{\bar v}_k^{j-1},\frac{\bar\varrho_k^j-\bar\varrho_k^{j-1}}{\delta_k}\rangle-\langle E{\bar v}_k^{0},\dot{\bar\varrho}(0)\rangle-\sum_{i=1}^j\langle E{\bar v}_k^{i-2},\ddot{\bar\varrho}_k^i\rangle-\sum_{i=1}^j\langle\frac{E\bar\omega^{i-1}_k-E\bar\omega^{i-2}_k}{\delta_k},\bar\varrho_k^i-\bar\varrho_k^{i-1}\rangle\nonumber\\
 &=\langle \dot{\tilde p}_k(t_k^{j-1}),\frac{\bar\varrho_k^j-\bar\varrho_k^{j-1}}{\delta_k}\rangle+\langle \dot{\tilde \sigma}_k^{j}(t_k^{j-1}),\mathbb A_r\frac{\bar\varrho_k^j-\bar\varrho_k^{j-1}}{\delta_k}\rangle-\langle E{\bar v}_k^{0},\dot{\bar\varrho}(0)\rangle-\sum_{i=1}^j\langle \dot{\tilde p}_k^{i-2}+\dot{\tilde e}_k^{i-2},\ddot{\bar\varrho}_k^i\rangle\nonumber\\
 &-\sum_{i=1}^j\langle\frac{E\bar\omega^{i-1}_k-E\bar\omega^{i-2}_k}{\delta_k},\bar\varrho_k^i-\bar\varrho_k^{i-1}\rangle\nonumber\\
 &\leq C\|\dot{\tilde p}_k(t_k^{j-1})\|_{L^1}+\frac{1}{2}\|\dot{\tilde \sigma}_k(t_k^{j-1})\|_{L^2}^2 +C +C\int_{0}^T\|\dot{\tilde \sigma}_k\|_{L^2}ds, 
\end{align}
where we have denoted $\ddot{\bar\varrho}_k^i:=\frac{\bar \varrho_k^i-2 \bar \varrho_k^{i-1}+\bar \varrho_k^{i-2}}{\delta_k^2}\in L^{\infty}([0,T]\times \Om)$ and $\dot{\bar\varrho}(0):=\frac{\bar\varrho_k^0-\bar\varrho_k^{-1}}{\delta_k}$. Here the constant $C$ depends on the norms
\begin{align}
 \|\dot{\bar \varrho}\|_{L^\infty(L^\infty)},\;\;\|\ddot{\bar \varrho}\|_{L^\infty(L^\infty)},\;\;\|E\ddot w\|_{L^1(L^2)}.
\end{align}
Coupling this last inequality with \eqref{f_ind_time} we arrive at
\begin{align*}
&\frac{1}{2}\Big{\Vert} \dot{\tilde \sigma}_k(t) \Big{\Vert}_{L^2}^2+ \frac{1}{2} \Big{\Vert}  (\dot{\tilde v}_3)_k(t)- (\dot{\tilde\omega}_3)_k(t) \Big{\Vert}_{L^2}^{2}\\
&\;\;\;\leq C+C \Big(\|\dot{\tilde p}_k(t)\|_{L^1}+\int_{0}^Ta_1(s)\|\dot{\tilde \sigma}_k\|_{L^2}ds+\int_0^Ta_2(s)\|  (\dot{\tilde v}_3)_k- (\dot{\tilde\omega}_3)_k\|_{L^2}ds \Big),
\end{align*}
for all $t\in[0,T]$, where $a_1$ and $a_2$ belong to $L^1([0,T])$ and are independent of $k$, $\la$, and $N$.
Integrating in the variable $t\in[0,T]$ and using Young inequality we conclude \eqref{secondprioriestimate}, thanks to \eqref{primastimapriori}.
Finally, writing $E\tilde u_k=\tilde e_k+\tilde p_k$, from the fact that $p_0:=Eu_0-e_0\in L^2(\Om; \Mtwo)$, \eqref{discrete problem}, \eqref{primastimapriori}, \eqref{secondprioriestimate},  and the Korn inequality, it follows that 
\begin{equation} \label{thirdestimate}
\Vert \tilde{u}_k \Vert_{W^{1,N'}(W^{1,N'})} \leq C,
\end{equation}
where $C$ is a constant that does not depend on $k$, $\lambda$, and $N$.

\medbreak
\noindent\textit{Step 4: Existence.}
 From \eqref{primastimapriori}, \eqref{secondprioriestimate}, and \eqref{thirdestimate} we deduce that there exist
\begin{equation}
\begin{split}
& u^\la \in W^{1,N'}([0,T];W^{1,N'}(\Omega; \R^3) \cap KL (\Om)) ,  \\
& \sigma^\la \in L^{\infty}([0,T]; L^2(\Om; \Mtwo)), \quad \tilde{\sigma}^\la \in  H^1([0,T]; L^2(\Om; \Mtwo)), \\
& v_3^\la \in H^1([0,T]; L^2(\omega)),
\end{split}
\end{equation}
 such that, up to subsequences,
\begin{align}
& \tilde{u}_k \wto u^\la \mbox{ weakly in } W^{1,N'}([0,T];W^{1,N'}(\Omega; \R^3)),  \label{convuk} \\
& (\tilde{u}_3)_k \wto u_3^\la \mbox{ weakly* in } W^{1,\infty}([0,T];L^2(\omega)), \label{convu3k} \\
& (\tilde{v}_3)_k \wto v_3^\la \mbox{ weakly in } H^1([0,T];L^2(\omega)), \label{convv3k} \\
& {\sigma}_k \wto {\sigma}^\la \mbox{ weakly* in } L^{\infty}([0,T];L^2(\Omega; \Mtwo)), \label{convsigmak} \\ 
& \tilde{\sigma}_k \wto \tilde{\sigma}^\la \mbox{ weakly in } H^1([0,T];L^2(\Omega; \Mtwo)), \label{convmintytrick1} \\ 
& \tilde{p}_k \wto p^\la \mbox{ weakly in } W^{1,N'}([0,T];L^{N'}(\Omega; \Mtwo)). \label{convpk}
\end{align}
Using the estimates in \eqref{secondprioriestimate} we have 
\begin{align}
& \int_0^T\Vert (\tilde{v}_3)_k(t)-(\dot{\tilde{u}}_3)_k(t)  \Vert_{L^2}dt=\sum_{i=1}^k\int_{t_k^{i-1}}^{t_k^{i}}\Vert (\tilde{v}_3)_k(t)-({\tilde{v}}_3)_k(t_k^{i-1})  \Vert_{L^2}dt \nonumber\\\label{v3ku3kstessolimite}
&\leq\delta_k^{1/2} \int_0^T\Vert (\dot{\tilde{v}}_3)_k(t) \Vert_{L^2}dt \leq C \delta_k^{1/2},  \\ \label{sigmasigmatildestessolimite}
& \int_0^T\Vert \tilde{\sigma}_k(t)-\sigma_k(t)  \Vert_{L^2}dt\leq \delta_k^{1/2}\int_0^T \Vert \dot{\tilde{\sigma}}_k(t) \Vert_{L^2}dt \leq C \delta_k^{1/2}.
\end{align}
This proves that $\sigma^\la=\tilde{\sigma}^\la$ and $v_3^\la=\dot{u}_3^\la$.
In particular $$ u_3^\la \in H^2([0,T];L^2(\om)).$$ 
Using that $p_0\in L^2(\Om; \Mtwo)$, estimate \eqref{Dpsilambdalipschitz}, and the first condition in \eqref{discrete problem}, we infer 
$$u^\la \in H^1([0,T]; H^1(\Om; \R^3) \cap KL (\Om)), \quad p^\la \in H^1([0,T]; L^2(\Om; \Mtwo)) ,$$ 
and
\begin{equation} \label{convformintytrick2}
\tilde{p}_k \wto p^\la \quad \mbox{ weakly in } H^1([0,T]; L^2(\Om; \Mtwo)).
\end{equation} 
Now it is convenient to introduce the notation
\begin{align}\label{hat-func}
\eta_k(t):=\tilde \omega_k(t-\delta_k),\;\;\;\;\;\xi_k(t):=(\tilde v_3)_k(t-\delta_k),\mbox{ for } t \in [0,T]. 
\end{align}
By \eqref{hat-func}, the estimate
\begin{align*}
&\int_0^T\Vert (\tilde{v}_3)_k(t)-\xi_k(t)  \Vert_{L^2}dt=\int_0^T\Vert \int_{t-\delta_k}^t(\dot{\tilde{v}}_3)_k(s)ds  \Vert_{L^2}dt\leq\delta_k^{1/2} C\int_0^T  \Vert (\dot{\tilde{v}}_3)_k \Vert_{L^2(L^2)} dt = C \delta_k^{1/2}, 
\end{align*}
implies, thanks to \eqref{convv3k}, that 
\begin{eqnarray}
 & \xi_k \wto \dot u_3^\la \mbox{ weakly in } H^1([0,T];L^2(\Omega)). \label{convvhat} 
\end{eqnarray}
Let us show that $(u,\sigma)$ satisfy \eqref{NHLambdaproblem}. The definitions of $\sigma_k$ and $\xi_k$  allow us to rewrite \eqref{variationalmotion} as
\begin{equation*}
\int_{0}^T \langle\dot\xi_k (t), \varphi_3 \rangle\psi(t) dt+\int_{0}^T\langle \sigma_k (t), E \varphi\rangle \psi (t) dt=\int_0^T\langle \mathcal L_{k}(t),\varphi\rangle\psi(t) dt,
\end{equation*}
for every $\varphi \in H^1(\Om;\R^3) \cap KL(\Om)$ with $\varphi=0$ on $\Gd$ and every $\psi\in C([0,T];\R)$. Thanks to \eqref{convv3k}, \eqref{convsigmak}, and \eqref{v3ku3kstessolimite}, we can pass to the limit in this equality as $k \to +\infty$ and get
\begin{equation} \label{variationalcontinuousmotion}
\int_{0}^T \langle \ddot{u}_3^\la (t),  \varphi_3 \rangle \psi (t) dt+\int_{0}^T \langle \sigma^\la (t), E \varphi \rangle \psi (t) dt=\int_0^T\langle \mathcal L(t),\varphi\rangle\psi(t) dt,.
\end{equation}
Thanks to the arbitrariness of $\varphi$ and $\psi$ this proves the second condition in \eqref{NHLambdaproblem}  with the corresponding Dirichlet and Neumann boundary conditions, owing to \eqref{traccia0mom} and \eqref{traccia1mom}. To conclude that the first equation in \eqref{NHLambdaproblem} is satisfied we have to check that
\begin{equation} \label{NHflowrule}
\dot{p}^\la(t)=D \psi_\la (\sigma^\la(t)) \mbox{ a.e.  in } \Om,
\end{equation}
for a.e. $t \in [0,T]$. 

First observe that \eqref{convsigmak} and  \eqref{convformintytrick2} together imply that
\begin{align}
&\sigma_k \wto \sigma^\la \quad \mbox{ weakly in } L^2([0,T]; L^2(\Om; \Mtwo)),\label{convsigmak_minty} \\
&\dot{\tilde{p}}_k \wto \dot p^\la \quad \mbox{ weakly in } L^2([0,T]; L^2(\Om; \Mtwo)).\label{convpk_minty}
\end{align}
We want to apply Proposition \ref{mintytrick} with $ A=D \Psi_\la$ and $X=X'=L^2([0,T]; L^2(\Om; \Mtwo))$. The first three hypotheses of Proposition \ref{mintytrick} hold true, owing to the last condition in \eqref{discrete problem}, \eqref{convsigmak_minty} and \eqref{convpk_minty}. To conclude  \eqref{NHflowrule}, we have to prove that
\begin{equation}
\limsup_{k \to +\infty} \int_0^T \langle \dot{\tilde{p}}_k(t), \sigma_k(t) \rangle dt \leq \int_0^T \langle \dot{p}^\la(t), \sigma^\la(t) \rangle dt.
\end{equation}
First of all, as a consequence of the first condition in \eqref{discrete problem}, we have that
\begin{align*}
\int_0^T \langle \dot{\tilde{p}}_k(t), \sigma_k(t) \rangle dt= \int_0^T \langle E \dot{\tilde{u}}_k(t), \sigma_k(t) \rangle dt-\int_0^T \langle  \dot{\tilde{e}}_k(t), \sigma_k(t) \rangle dt.
\end{align*}
Let us consider the first term on the right hand-side. From the second equation in \eqref{discrete problem} it follows that 
\begin{align} \label{sigmakEuk}
& \int_0^T \langle E \dot{\tilde{u}}_k(t), \sigma_k(t) \rangle dt=\nonumber\\
&=\int_0^T \langle E \dot{\tilde{w}}_k(t), \sigma_k(t) \rangle dt+ \int_0^T\langle E\bar v_k(t)-E\bar \omega_k(t),\bar \sigma_k(t)\rangle dt-\frac{1}{12} \int_0^T \langle D^2 (\dot{\tilde{u}}_3)_k(t)-D^2 (\dot{\tilde{w}}_3)_k(t), \hat{\sigma}_k(t) \rangle dt \nonumber\\
&= +\int_0^T \langle E \dot{\tilde{w}}_k(t), \sigma_k(t) \rangle dt+\int_0^T\langle \bar v_k(t)-\bar \omega_k(t),f_k(t)\rangle dt+\int_0^T \langle (\dot{\tilde{u}}_3)_k(t)-(\dot{\tilde{w}}_3)_k(t), g_k(t) \rangle dt\nonumber\\
& -\int_0^T \langle \xi_k(t)-\eta_k (t),\dot\eta_k (t)   \rangle dt - \frac{1}{2}   \Vert  \xi_k(T)-\eta_k (T) \Vert_{L^2}^{2}+\frac{1}{2}   \Vert  \xi_k(0)- \eta_k(0) \Vert_{L^2}^{2}\nonumber\\
& +\int_{0}^{T} \langle  \xi_k(t)-\eta_k (t)-({\dot{\tilde{u}}}_3)_k(t)+ (\dot{\tilde{w}}_3)_k(t),(\dot{\tilde{v}}_3)_k(t) \rangle dt.
\end{align}
Thanks to \eqref{convvhat}, we deduce that 
 $$ \xi_k(t) \wto \dot{u}_3^\la(t) \quad \mbox{ weakly  in }  L^2(\Om) \mbox{ for every } t \in [0,T]. $$  
From this we infer
\begin{align*}
 & \limsup_{k \to +\infty} \Big( - \frac{1}{2}   \Vert  \xi_k(T)-\eta_k (T) \Vert_{L^2}^{2}+\frac{1}{2}   \Vert  \xi_k(0)- \eta_k(0) \Vert_{L^2}^{2}  \Big) \\
& \leq - \frac{1}{2}   \Vert  \dot{{u}}_3^\la(T)- \dot{{w}}_3^\la(T) \Vert_{L^2}^{2} + \frac{1}{2}   \Vert  (v_0)_3- (\dot{{w}}_3)(0) \Vert_{L^2}^{2},
\end{align*}
where we have used that $\xi_k(0)=(v_0)_3$. Moreover
\begin{align*}
&\limsup_{k \to +\infty} -\int_{0}^{T} \langle \xi_k(t)-\eta_k (t),\dot\eta_k (t)   \rangle dt \\
&=\lim_{k \to +\infty} -\int_{0}^{T} \langle \xi_k(t)-\eta_k (t),\dot\eta_k (t)   \rangle dt = - \int_0^T \langle \dot{u}_3^\la(t)-\dot{w}_3 (t), \ddot{w}_3(t) \rangle dt,
\end{align*}
while, as a consequence of \eqref{v3ku3kstessolimite} and \eqref{convvhat}, the last term in the right-hand side of  \eqref{sigmakEuk} tends to $0$ as $k \to +\infty$. 
Finally, from \eqref{convmintytrick1} and \eqref{sigmasigmatildestessolimite},  we get
\begin{equation*}
\tilde{\sigma}_k (t) \wto \sigma^\la (t) \mbox{ weakly in } L^2(\Om; \Mtwo) \mbox{ for every } t \in [0,T].
\end{equation*}
Then
\begin{align*}
&\limsup_{k \to +\infty} \int_0^T \langle \sigma_k(t), \dot{\tilde{e}}_k (t) \rangle dt \\
&= \limsup_{k \to +\infty} \Big\{ -\frac{1}{2} \langle \mathbb{A}_r \tilde{\sigma}_k (T),\tilde{\sigma}_k (T) \rangle +\frac{1}{2} \langle \mathbb{A}_r {\sigma}_0, {\sigma}_0 \rangle -\int_0^T \langle \sigma_k(t)-\tilde{\sigma}_k (t), \dot{\tilde{e}}_k (t) \rangle dt  \Big\} \\
& \leq -\frac{1}{2} \langle \mathbb{A}_r {\sigma}^\la (T),{\sigma}^\la(T) \rangle +\frac{1}{2} \langle \mathbb{A}_r {\sigma}_0, {\sigma}_0 \rangle,
\end{align*} 
where in the inequality we have used \eqref{sigmasigmatildestessolimite}. All these considerations lead to
\begin{align*}
&\limsup_{k \to +\infty}  \int_0^T \langle  \dot{\tilde{p}}_k(t), \sigma_k(t) \rangle dt
\leq   - \frac{1}{2}   \Vert  \dot{{u}}_3^\la(T)- (\dot{{w}}_3)(T) \Vert_{L^2}^{2} + \frac{1}{2}   \Vert  \dot{{u}}_3^\la(0)- \dot{{w}}_3(0) \Vert_{L^2}^{2} \\
&+\int_0^T \langle \sigma^\la (t), E \dot{w}(t) \rangle dt+\int_0^T \langle \ddot{w}_3(t), \dot{w}_3 (t)-\dot{u}_3^\la (t) \rangle dt -\frac{1}{2} \langle \mathbb{A}_r {\sigma}^\la (T),{\sigma}^\la(T) \rangle +\frac{1}{2} \langle \mathbb{A}_r {\sigma}_0, {\sigma}_0 \rangle \\
&+\int_0^T\langle \dot{\bar u}^\la(t)-\dot{\bar w}(t),f(t)\rangle dt+\int_0^T \langle \dot{{u}}^\la_3(t)-\dot{\tilde{w}}_3(t), g(t) \rangle dt= \int_0^T \langle  \dot{p}^\la(t), \sigma^\la(t) \rangle dt.
\end{align*}
Note that the last equality in the previous formula derives from an integration by parts and from \eqref{variationalcontinuousmotion}. 

It remains to check that the initial conditions hold. The definition of the interpolate approximating functions gives
$$ (\tilde{u}_k(0),\tilde{\sigma}_k(0))=(u_0,\sigma_0), \qquad \xi_k(0)=(v_0)_3.  $$
Thus the initial conditions follow from the previous equalities together with \eqref{convuk},\eqref{convmintytrick1} and \eqref{convvhat}.
This proves the existence of solutions for \eqref{NHLambdaproblem} and concludes Step 4.

\medbreak
\noindent\textit{Step 5: Uniqueness.}  Let $(u^\la, \sigma^\la)$ and $(v^\la,\tau^\la)$ be two solutions with the same initial data $\sigma^\la(0)=\tau^\la(0)$, $u^\la(0)=v^\la(0)$. By the first equation in \eqref{NHLambdaproblem} we get
\begin{align*}
&\langle \mathbb{A}_r \dot{\sigma}^\la(t)-\mathbb{A}_r\dot{\tau}^\la(t), \sigma^\la(t)-\tau^\la(t) \rangle+\langle D \psi_\la( \sigma^\la (t))-D \psi_\la (\tau^\la (t)),\sigma^\la (t)-\tau^\la (t) \rangle \\
&=\langle E \dot{u}^\la (t)-E \dot{v}^\la (t),\sigma^\la (t)-\tau^\la (t) \rangle.
\end{align*} 
Integrating by parts the right-hand side and using the convexity of $D\psi_\la$ we obtain
$$\langle \mathbb{A}_r \dot{\sigma}^\la(t)-\mathbb{A}_r\dot{\tau}^\la(t), \sigma^\la(t)-\tau^\la(t) \rangle+\frac{1}{12} \langle \mathbb{A}_r \dot{\hat {\sigma}}^\la(t)-\mathbb{A}_r\dot{\hat{\tau}}^\la(t), D^2 u_3^\la(t)-D^2 v_3^\la(t) \rangle  \leq 0.$$
We integrate in the time interval $[0,T]$ and recalling \eqref{Acoercivity} we get
$$ \alpha_\mathbb{A} \Vert \sigma^\la(t)-\tau^\la(t)  \Vert_{L^2}^2+\frac{1}{2} \Vert \dot{u}_3^\la (t)-\dot{v}_3^\la (t) \Vert_{L^2}^2 \leq 0. $$
Then $\sigma^\la (t)=\tau^\la(t)$ and $u_3^\la (t)=v_3^\la (t)$ for every $t \in [0,T]$. Thanks to the first condition in \eqref{NHLambdaproblem} we get  $Eu^\la (t)=Ev^\la (t)$. Since $u^\la (t)=v^\la (t)$ on $\Gd$, by Korn inequality we conclude that $u^\la(t)=v^\la (t)$ for every $t \in [0,T]$.

\medbreak
\noindent\textit{Step 6: Conclusion. }
The estimates \eqref{firstlambdaestimate} and \eqref{secondlambdaestimate} are obtained by simply using \eqref{convuk}-\eqref{convpk} to pass to the limit as $k \to +\infty$ in \eqref{primastimapriori}, \eqref{secondprioriestimate}. Similarly, starting from \eqref{thirdprioriestimate} we infer \eqref{secondlambdaestimate*}.
\end{proof}

Now we are ready to prove the existence and uniqueness of solutions for an approximating problem of Norton--Hoff type.

\begin{theorem} \label{TheoremNHoff}
Assume that all the hypotheses of Lemma \ref{Lemmalambda} are satisfied.
 Then for every $N \in \mathbb{N}$, $N \geq 4$, the problem
\begin{equation}
\begin{cases} \label{NHproblem}
\mathbb{A}_r \dot{\sigma} (t)+D \phi_N(\sigma(t))= E \dot{u} (t) \mbox{ in } \Om, \\
-\Div \bar{\sigma}(t)=f(t), \quad \ddot{u}_3(t)-\frac{1}{12} \Div \Div \hat{\sigma}(t)=g(t) \mbox{ in } \om, \\
\sigma(t) \in \Theta (\gn,0,0), \\
u(t)=w(t) \mbox{ on } \Gd,
\end{cases}
\end{equation}
has a unique solution $(\sigma^N,u^N)$ with
\begin{align*}
&\sigma^N \in  H^1([0,T];L^2(\Om; \Mtwo)) \cap L^N ([0,T]; L^N (\Om; \Mtwo)), \\
&u^N \in W^{1,N'}([0,T]; W^{1,N'}(\Om; \R^3) \cap KL(\Om)),\\
&u_3^N\in  H^2([0,T];L^2(\om)),
\end{align*}
and $(u(0),\sigma(0))=(u_0,\sigma_0) $ and $\dot{u}_3(0)=(v_0)_3$.
Moreover the following estimates hold:
\begin{align} \label{stimasigmaW1infty}
&\Vert \sigma^N \Vert_{ H^1(L^2)} \leq C, \\ \label{stimasigmalinftylN}
&\Vert \sigma^N \Vert_{L^N(L^N)} \leq C \al_0^{N-1}, \\
&\Vert u_3^N \Vert_{ H^2 (L^2)} \leq C,\label{4.48}\\
& \Vert u^N \Vert_{ H^1(BD)} \leq C, \label{stimauN} 
\end{align}
where $C>0$ is a constant independent of $N$.
\end{theorem}
\begin{proof} \textit{Step 1: passage to the limit as $\la\rightarrow+\infty$.} Let $(u^\la, \sigma^\la)$ as in the thesis of Lemma \ref{Lemmalambda}. Since  \eqref{firstlambdaestimate} and \eqref{secondlambdaestimate} are uniform estimates with respect to $\la$, we can pass to the limit as   $\la \to +\infty$ and deduce the existence of
$$ (\sigma^N, u^N ) \in H^1([0,T]; L^2(\Om; \Mtwo)) \times W^{1,{N'}}([0,T]; W^{1,{N'}}(\Om; \R^3) \cap KL (\Om)) $$
with
 $$ u_3^N \in H^2([0,T]; L^2(\om)) $$
 such that, up to subsequences,
 \begin{align} \label{convsigmalasigmaN}
&\sigma^\la \wto \sigma^N \mbox{ weakly in }  H^1([0,T]; L^2(\Om; \Mtwo)), \\ \label{convulauN}
&u^\la\wto u^N \mbox{ weakly in } W^{1,N'}([0,T]; W^{1,N'}(\Om; \R^3)),\\
&u_3^\la \wto u_3^N \mbox{ weakly in }  H^2([0,T]; L^2(\om)).\label{convu_3infinito}
 \end{align} 
In particular   \eqref{stimasigmaW1infty} and \eqref{4.48} hold true thanks to \eqref{secondlambdaestimate}. 
 By \eqref{firstlambdaestimate} there exists 
 $$\beta^N \in L^{{N'}}([0,T]; L^{{N'}}(\Om; \Mtwo)), $$ 
satisfying, up to a subsequence,
\begin{equation} \label{convDpsilambdaN}
D \psi_\lambda (\sigma^\la) \wto \beta^N \quad \mbox{ weakly in } L^{{N'}}([0,T]; L^{{N'}}(\Om; \Mtwo)). 
\end{equation}
Passing to the limit in the first equation of \eqref{NHLambdaproblem} we get 
\begin{align}\label{+}
\mathbb{ A}_r \dot{\sigma}^N (t)+ \beta^N (t)= E \dot{u}^N (t) \mbox{ a.e. in } \Om, 
\end{align}
for a.e. $t \in [0,T]$.
By \eqref{convsigmalasigmaN} and \eqref{convulauN} we can pass to the limit in \eqref{variationalcontinuousmotion} as $\la$ tends to $+\infty$ and infer that 
\begin{equation} \label{variationalmotionN}
\int_{0}^T \langle \ddot{u}_3^N(t), \varphi_3\rangle \psi(t) dt+\int_{0}^T \langle \sigma^N (t), E \varphi\rangle \psi(t) dt=\int_0^T\langle \mathcal L(t),\varphi(t)\rangle \psi(t)dt,
\end{equation}
for every $\varphi \in H^1(\Om; \R^3) \cap KL(\Om)$ with $\varphi=0$ on $\Gd$, and every $\psi\in C^0([0,T])$. Therefore the equation of motions holds true  together with corresponding boundary conditions, as a consequence of \eqref{traccia0mom} and \eqref{traccia1mom}.

 Let us now show the estimate \eqref{stimasigmalinftylN}.
First of all, we multiply the equation $E\dot u^\lambda=\mathbb A_r\dot\sigma^\lambda+D\psi_\lambda(\sigma^\lambda)$ by $\sigma^\lambda$ and integrate in $[0,T]\times \Omega$; integrating by parts the term $\int_0^T\langle E\dot u^\la,\sigma^\la\rangle dt$ and exploiting \eqref{firstlambdaestimate} and the estimates on $w$ and the external forces, we arrive at
\begin{equation} \label{stimatroncata}
\frac{1}{\al_0^{N-1}} \int_0^T \int_\Om \left( \vert \sigma^\la (t,x) \vert_r^N \wedge \la^{N-2} \vert \sigma^\la (t,x) \vert_r^2  \right) dx dt \leq C,
\end{equation}
where $C$ is a positive constant independent of $\la$ and $N$. Set $\tau^\la:= \chi_{\lbrace \vert \sigma^\la \vert_r \leq \la \rbrace} \sigma^\la$. From \eqref{stimatroncata} it follows that 
\begin{equation} \label{stimatauNNtroncata}
\frac{1}{\al_0^{N-1}} \int_0^T \int_\Om \vert \tau^\la (t,x) \vert_r^N dx dt \leq C.
\end{equation} 
Using \eqref{stimatroncata} again, we infer
\begin{align*}
\int_0^T \int_\Om \vert \sigma^\la (t,x)-\tau^\la (t,x) \vert_r^2 dx dt=\int_0^T \int_{ \lbrace \vert \sigma^\la \vert_r \geq \la \rbrace} \vert \sigma^\la (t,x) \vert_r^2 dx dt \leq \frac{C \al_0^{N-1}}{\la^{N-2}},
\end{align*}
where the equality derives from the definition of $\tau^\la$. Hence
$$ \sigma^\la-\tau^\la \to 0 \mbox{ strongly in } L^2([0,T];L^2(\Om; \Mtwo)), $$
and then
\begin{equation}\label{348}
\tau^\la \wto \sigma^N \mbox{ weakly in } L^2([0,T];L^2(\Om; \Mtwo)).
\end{equation}
This, together with \eqref{convsigmalasigmaN} and \eqref{stimatauNNtroncata}, implies that $\sigma^N \in L^N([0,T]; L^N(\Om;\Mtwo))$ and   
\begin{equation*}
\tau^\la \wto \sigma^N \mbox{ weakly in } L^N([0,T];L^N(\Om; \Mtwo)).
\end{equation*}
Morever, by \eqref{stimatauNNtroncata}, the inequality in \eqref{stimasigmalinftylN} is achieved.

\medbreak
\noindent\textit{Step 2: identification of $\beta^N$.} To conclude that $(\sigma^N, u^N)$ satisfies \eqref{NHproblem}, we have to show that for a.e. $t \in [0,T]$
\begin{equation} \label{gammaNdphiN}
\beta^N (t)=D \phi_N (t) \mbox{ a.e. in }  \Om.
\end{equation}
First we show that
\begin{equation} \label{limsupDpsilaineq}
\limsup_{\la \to +\infty} \int_0^T \langle D \psi_\la (\sigma^\la (t)), \sigma^\la (t) \rangle dt \leq \int_0^T \langle \beta^N(t), \sigma^N (t) \rangle dt.
\end{equation}
Indeed, owing to  \eqref{convsigmalasigmaN}, \eqref{convulauN}, \eqref{convu_3infinito},  and \eqref{convDpsilambdaN}, we obtain
\begin{align*}
&\limsup_{\la \to +\infty} \int_0^T \langle D \psi_\la (\sigma^\la (t)), \sigma^\la (t) \rangle dt \\
&\leq\limsup_{\la \to +\infty} - \int_0^T \langle \mathbb{A}_r \dot{\sigma}^\la (t),\sigma^\la (t) \rangle dt+\limsup_{\la \to +\infty} \int_0^T \langle E \dot{u}^\la (t),  \sigma^\la (t)\rangle dt \\
& \leq \limsup_{\la \to +\infty} \Big( -\frac{1}{2} \langle \mathbb{A}_r \sigma^\la (T), \sigma^\la (T) \rangle+\frac{1}{2} \langle \mathbb{A}_r \sigma_0 , \sigma_0 \rangle\Big)+ \limsup_{\la \to +\infty} \Big(  -\frac{1}{2} \Vert \dot{u}_3^\la (T)-\dot{w}_3 (T) \Vert_{L^2}^{2}\\
& \;\;+\frac{1}{2} \Vert (v_0)_3-\dot{w}_3 (0) \Vert_{L^2}^{2}-\int_0^T \langle \ddot{w}_3 (t),\dot{u}_3^\la(t)-\dot{w}_3(t) \rangle dt +\int_0^T \langle \sigma^\la (t), E \dot{w}(t) \rangle dt   \Big)\\
&\;\;+ \int_0^T\langle \dot u_3^\la(t)-\dot w_3(t),g(t)\rangle dt+\int_0^T\langle \dot{\bar u}^\la(t)-\dot{\bar w}(t),f(t)\rangle dt\\
& \leq   -\frac{1}{2} \langle \mathbb{A}_r \sigma^N (T), \sigma^N (T) \rangle+\frac{1}{2} \langle \mathbb{A}_r \sigma_0 , \sigma_0 \rangle+\int_0^T \langle \sigma^N (t), E \dot{w}(t) \rangle dt   \\
&  \;\; -\frac{1}{2} \Vert \dot{u}_3^N (T)-\dot{w}_3 (T) \Vert_{L^2}^{2}+\frac{1}{2} \Vert (v_0)_3-\dot{w}_3 (0) \Vert_{L^2}^{2}-\int_0^T \langle \ddot{w}_3 (t),\dot{u}_3^N(t)-\dot{w}_3(t) \rangle dt \\
&\;\;+ \int_0^T\langle \dot u_3^N(t)-\dot w_3(t),g(t)\rangle dt+\int_0^T\langle \dot{\bar u}^N(t)-\dot{\bar w}(t),f(t)\rangle dt= \int_0^T \langle \beta^N(t), \sigma^N (t) \rangle dt,
\end{align*}
where in the last equality we  have used the second equation in \eqref{NHproblem} and \eqref{+}. We have proved \eqref{limsupDpsilaineq}. Let us fix $\tau \in L^\infty([0,T] \times \Om; \Mtwo)$. By convexity of $ \psi_\la$ we have
$$\int_0^T \langle D \psi_\la (\sigma^\la (t))-D \psi_\la (\tau(t)), \sigma^\la (t)-\tau (t) \rangle dt \geq 0,$$
hence using that $D \psi_\la (\tau (t))= D\phi_N (\tau (t))$ for $\la > \Vert \tau \Vert_{L^\infty}$, convergences \eqref{convsigmalasigmaN}   and \eqref{convDpsilambdaN} ensure
\begin{equation*}
\liminf_{\la \to +\infty} \int_0^T \langle D \psi_\la (\sigma^\la (t)), \sigma^\la (t)\rangle dt \geq \int_0^T \langle \beta^N (t), \tau (t) \rangle dt+\int_0^T \langle D \phi_N (\tau (t)), \sigma^N (t)-\tau (t) \rangle dt.
\end{equation*}
Combining this inequality with \eqref{limsupDpsilaineq} we get 
\begin{equation*}
\int_0^T \langle \beta^N (t)- D \phi_N (\tau (t)), \sigma^N (t)-\tau (t)\rangle dt \geq 0
\end{equation*}  
for every $\tau \in L^\infty ((0,T) \times \Om; \Mtwo)$. Since $D\phi_N$ is a maximal monotone operator and $\tau$ is arbitrary, \eqref{gammaNdphiN} is satisfied.

\medbreak
\noindent\textit{Step 3: Estimate \eqref{stimauN}.}
Following the lines of the proof in \cite[Theorem 4.2]{DavMor2} we test the kinematic compatibility at time $t$ by $\sigma^N(t) -\varrho(t)$. Integrating by parts we arrive at
\begin{align}
 &\langle D \phi_N(\sigma^N(t)),\sigma^N (t) -\varrho(t)\rangle\nonumber\\
 &=-\langle\dot u_3^N(t)-\dot w_3(t),\ddot u_3^N(t)\rangle+\langle E\dot w(t),\sigma^N(t)-\varrho(t)\rangle-\langle\mathbb A_r\dot \sigma^N(t),\sigma^N(t)-\varrho(t)\rangle\nonumber\\
 &\leq C(\|\ddot u_3^N(t)\|_{L^2}+\|\dot \sigma^N(t)\|_{L^2}+1),
\end{align}
where we have used \eqref{stimasigmaW1infty} and \eqref{4.48}.
Arguing as in \cite[Theorem 4.2]{DavMor2} we finally infer
\begin{align}
 \int_\Om|D\phi_N(\sigma^N(t))|dx\leq C(\|\ddot u_3^N(t)\|_{L^2}+\|\dot \sigma^N(t)\|_{L^2}+1),
\end{align}
which from estimates \eqref{stimasigmaW1infty} and \eqref{4.48} implies that $(D\phi_N(\sigma^N(t)))_N$ is uniformly bounded in $L^2([0,T];L^1(\Om;\Mtwo))$. This together with \eqref{stimasigmaW1infty} concludes the proof of \eqref{stimauN} by 
 Korn--Poincar\'e inequality in $BD(\Om)$. 

\medbreak
\noindent\textit{Step 4: uniqueness.} The proof of uniqueness is very similar to the one in Step 5 of Lemma \ref{Lemmalambda} that can be easily adapted.

\end{proof}

\begin{theorem} \label{TheoremNHoff2}
In the hypotheses of Theorem \ref{TheoremNHoff}, assume in addition that the force $f$ is independent of time $t$. Then the solution $(\sigma^N,u^N)$ also satisfies
\begin{align*}
&\sigma^N \in W^{1,\infty}([0,T];L^2(\Om; \Mtwo)) \cap L^\infty ([0,T]; L^N (\Om; \Mtwo)), \\
&u_3^N\in W^{2,\infty}([0,T];L^2(\om)).
\end{align*}
and the following estimates:
\begin{align} \label{stimasigmaW1inftybis}
&\Vert \sigma^N \Vert_{W^{1,\infty}(L^2)} \leq C, \\ \label{stimasigmalinftylNbis}
&\Vert \sigma^N \Vert_{L^\infty(L^N)} \leq C N \al_0^{N-1}  \\ \label{stimauNbis}
& \Vert u^N \Vert_{W^{1,\infty}(BD)} \leq C, \quad \Vert u_3^N \Vert_{W^{2,\infty}(L^2)} \leq C,
\end{align}
where $C>0$ is a constant independent of $N$.
\end{theorem}

\begin{remark}
Since $u^N$ is a Kirchhoff--Love displacement,  the first estimate in \eqref{stimauNbis} together with the continuous immersion of $BD(\om)$ into $L^2(\om; \R^2)$ and of $BH(\om)$ into $H^1(\om)$ implies that
\begin{equation} \label{estimateUNkirchoff}
\Vert {u}^N \Vert_{W^{1,\infty}(L^2)}+\Vert u_3^N \Vert_{W^{1,\infty}(H^1)} \leq C,
\end{equation}
where $C$ is a positive constant independent of $N$.
\end{remark}

\begin{proof}
The proof follows the same lines of that of Theorem \ref{TheoremNHoff}, with the following changes. In step one we observe that, thanks to estimate \eqref{secondlambdaestimate*} we conclude estimate \eqref{stimasigmaW1inftybis} and the second one in \eqref{stimauNbis}, which in particular imply the additional regularity 
$$ \sigma^N \in W^{1,\infty}([0,T]; L^2(\Om; \Mtwo)), $$
and
 $$ u_3^N \in W^{2,\infty}([0,T]; L^2(\om)). $$
We now prove \eqref{stimasigmalinftylNbis}.   Let us multiply the second equation in \eqref{discrete problem} by $\sigma_k^i-\sigma_k^{i-1}$. This yields
$$ \langle \mathbb{A}_r \frac{\sigma_k^i-\sigma_k^{i-1}}{\delta_k}, \sigma_k^i-\sigma_k^{i-1} \rangle+ \langle D \psi_\la  (\sigma_k^i), \sigma_k^i-\sigma_k^{i-1} \rangle= \langle  \frac{E u_k^i-E u_k^{i-1}}{\delta_k}, \sigma_k^i-\sigma_k^{i-1} \rangle. $$
It is clear that the first term on the left-hand side is positive, while
$$ \langle D \psi_\la  (\sigma_k^i), \sigma_k^i-\sigma_k^{i-1} \rangle \geq  \Psi_\la (\sigma_k^i) - \Psi_\la (\sigma_k^{i-1})  $$
as a consequence of the convexity of $\psi_\la$. Therefore, integrating by parts, we get
\begin{align*}
&\Psi_\la (\sigma_k^i) -\Psi_\la  (\sigma_k^{i-1})\leq  \langle \sigma_k^i-\sigma_k^{i-1}, \frac{E w_k^i- E w_k^{i-1}}{\delta_k} \rangle\\
&+ \langle E\bar v_k^{i-1}-E\bar\omega_k^{i-1},\bar \varrho_k^i-\bar \varrho_k^{i-1}\rangle+\langle ( v_3)_k^{i-1}-( \om_3)_k^{i-1},g_k^i-g_k^{i-1} \rangle\\
&\;\;- \langle \frac{(v_3)_k^{i-1}-2(v_3)_k^{i-2}+(v_3)_k^{i-3}}{\delta_k}, ( v_3)_k^{i-1}-( \om_3)_k^{i-1} \rangle.
\end{align*}
Let $j>1$ and sum on $i=1,\dots,j$. Rearranging terms we infer
\begin{align}\label{sum}
&\Psi_\la (\sigma_k^j) -\Psi_\la (\sigma_0) \leq \sum_{i=2}^j \langle{\sigma_k^i-\sigma_k^{i-1}}, E \om_k^{i-1} \rangle\nonumber\\
&+ \langle (v_3)_k^{j-1}-(\omega_3)_k^{j-1}, \frac{(v_3)_k^{j-1}-(v_3)_k^{j-2}}{\delta_k}-\frac{(\om_3)_k^{j-1}-(\om_3)_k^{j-2}}{\delta_k} \rangle\nonumber \\
&-\sum_{i=1}^{j} \langle \frac{(v_3)_k^{i-1}-(v_3)_k^{i-2}}{\delta_k}- \frac{(\om_3)_k^{i-1}-(\om_3)_k^{i-2}}{\delta_k}, (v_3)_k^{i-1}-(v_3)_k^{i-2}-(\om_3)_k^{i-1}+(\om_3)_k^{i-2}  \rangle \nonumber\\
&-\langle (v_3)_k^{0}-(\om_3)_k^{0}, \frac{(v_3)_k^{0}-(v_3)_k^{-1}}{\delta_k}-\frac{(\om_3)_k^{0}-(\om_3)_k^{-1}}{\delta_k} \rangle\nonumber\\
&- \sum_{i=1}^{j} \langle \frac{(\om_3)_k^{i-1}-2(\om_3)_k^{i-2}+(\om_3)_k^{i-3}}{\delta_k}, ( v_3)_k^{i-1}-( w_3)_k^{i-1} \rangle\nonumber\\
& +\sum_{i=1}^{j}\langle ( v_3)_k^{i-1}-( w_3)_k^{i-1},g_k^i-g_k^{i-1} \rangle+\sum_{i=1}^{j}\langle E\bar v_k^{i-1}-E\bar\omega_k^{i-1},\bar \varrho_k^i-\bar \varrho_k^{i-1}\rangle.
\end{align}
Hence, since $f$ is independent of time the last term is null and we can use  \eqref{secondlambdaestimate*} to estimate the right-hand side as
\begin{align*}
&\Psi_\la(\sigma_k(t))  \leq \Psi_\la (\sigma_0)+\Vert {\dot{\tilde{\sigma}}_k} \Vert_{L^2(L^2)}\Vert E \dot{w} \Vert_{L^2(L^2)}+\Vert (\tilde{v}_3)_k-\dot w_3 \Vert_{L^\infty(L^2)}\Vert (\dot{\tilde{ v}}_3)_k-\ddot w_3 \Vert_{L^\infty(L^2)} \\
& +\Vert (\dot{\tilde v}_3)_k-\ddot w_3 \Vert_{L^2(L^2)}^2 + \Vert \dddot{w}_3 \Vert_{L^1(L^2)} \Vert (\tilde{v}_3)_k-\dot w_3)_k \Vert_{L^\infty(L^2)}+\|(\tilde{v}_3)_k-\dot w_3 \|_{L^\infty(L^2)}\|\dot g\|_{L^1(L^2)}.
\end{align*}
By \eqref{primastimapriori},  \eqref{secondprioriestimate}, and \eqref{thirdprioriestimate}, we can pass  to the limit as $k \to +\infty$ and we get    
\begin{equation} \label{supsigmala}
\sup_{t \in [0,T]} \Psi_\la (\sigma^\la (t)) \leq C,
\end{equation}
for a constant $C$ independent of $N$.
Consequently 
\begin{equation} 
\frac{1}{ N \al_0^{N-1}} \sup_{t \in [0,T]} \int_\Om \vert \tau^\la (t,x) \vert_r^N dx  \leq C,
\end{equation}
where $\tau^\la$ has been first used in \eqref{stimatauNNtroncata}.
Therefore convergence \eqref{348} also takes place with respect to the weak star topology of the space $L^\infty([0,T];L^N(\Om; \Mtwo))$, so that $\sigma \in L^\infty([0,T];L^N(\Om; \Mtwo)) $ and  \eqref{stimasigmalinftylNbis} holds true.

Let us prove the first estimate in \eqref{stimauNbis}. We test the kinematic compatibility  by $\sigma^N$. We arrive at
\begin{align*} 
&\int_\Omega \vert D \phi_N (\sigma^N(t,x)) \vert dx\leq \|g(t)\|_{L^2}\|\dot u_3^N(t)-\dot w_3(t)\|_{L^2} \\
& + \Vert \sigma^N (t) \Vert_{L^2} ( \Vert \dot{\sigma}^N (t) \Vert_{L^2}+  \Vert E \dot{w} (t) \Vert_{L^2} )+\Vert \ddot{u}_3^N (t) \Vert_{L^2} \Vert \dot{u}_3^N(t)-\dot{w}_3(t) \Vert_{L^2}.
\end{align*}
We now observe that the estimates \eqref{stimasigmaW1inftybis} and the second in \eqref{stimauNbis} imply that the right-hand side is uniformly bounded. Hence we conclude
$$ \Vert D \phi_N (\sigma^N) \Vert_{L^\infty (L^1)} \leq C$$
for every $N \geq 4$, where $C$ is independent of $N$. Therefore 
$$ \Vert E \dot{u}^N \Vert_{L^\infty (L^1)} \leq C.$$
This yields the first estimate in \eqref{stimauNbis} as a consequence of Korn--Poincar\'e inequality in $BD(\Om)$. 
\end{proof}

Now we prove additional regularity for the solution of the Norton--Hoff problem.

\begin{proposition} \label{propositionregular}
 Let $N \in \mathbb{N}$ with $N \geq 4$. Assume that all the hypotheses of Lemma  \ref{Lemmalambda} are satisfied and in addition that $$\sigma_0 \in H^1(\Om; \Mtwo).$$ Then the stress component $\sigma^N$ and the velocity $\dot u_3^N$ satisfy the following estimates:
\begin{itemize}
\item for every open set $\om'$ compactly contained in $\om$ there exists a constant $C_1>0$ depending on $\om'$ but independent of $N$ such that
\begin{equation} \label{regestimatesigmaalN}
\sup_{t \in [0,T]} \Vert D_\alpha \sigma^N (t) \Vert_{L^2 ( \om ' \times \left( -\frac{1}{2},\frac{1}{2} \right))} \leq C_1,
\end{equation}   
\item for every open set $\Om'$ compactly contained in $\Om$ there exists a constant $C_2>0$ depending on $\Om'$ but independent of $N$ such that
\begin{equation} \label{regestimatesigma3N}
\sup_{t \in [0,T]} \Vert D_3 \sigma^N (t) \Vert_{L^2 ( \Om ' )} \leq C_2.
\end{equation}
\item for every open set $\om'$ compactly contained in $\om$ there exists a constant $C_3>0$ depending on $\om'$ but independent of $N$ such that
\begin{equation} \label{reg_Du3}
\sup_{t \in [0,T]} \Vert D_\alpha \dot u_3^N (t) \Vert_{L^2 ( \om ')} \leq C_3.
\end{equation}  
\end{itemize} 
\end{proposition}

 Before proving the result let us introduce the difference quotient operator, that for $h\in\R$ is defined as
$$ D_\al^h v (x) \doteq \frac{v(x+h e_\al)-v(x)}{h}, \mbox{ for } \alpha\in \{1,2,3\}, $$
for a given function $v: \R^3 \to \R$. We recall some important properties of the difference quotient. If $f\in L^p(\Om;\R)$ and $g\in L^q(\Om;\R)$, $\psi\in C^0(\Om;\R)$ with either $f$ or $g$ with compact support in $\Om$, then for $h$ sufficiently small we have
\begin{align}
 &\langle D^h_\alpha f, g\rangle=-\langle f,D^{-h}_\al g\rangle,\nonumber\\
 &\langle D^{-h}_\alpha(fg),\psi\rangle=\langle D^{-h}_\al f T^{-h}_\alpha g,\psi\rangle+\langle fD^{-h}_\al g,\psi\rangle,\label{D2}
\end{align}
where $T^{-h}_\alpha g(x):=g(x-he_\al)$.

\begin{proof} \textit{Step 1.}
Let us first prove higher regularity for $\sigma^\la$, the approximating solution of $\sigma^N$ found in Lemma \ref{Lemmalambda}. Let $\varphi \in C_c^\infty (\om)$, let $\al\in\{1,2\}$, and $h\in (-1,1)$ be sufficiently small. We multiply the first equation in \eqref{NHLambdaproblem} by $D_\al^{-h}(\varphi^2 D_\al^h \sigma (t))$ obtaining
\begin{equation} \label{lambdakinemadmissibility}
\langle  D_\al^h \mathbb{A}_r \dot{\sigma}^\la (t), \varphi^2 D_\al^h \sigma^\la (t) \rangle + \langle  D_\al^h D \psi_\la (\sigma^\la (t)), \varphi^2 D_\al^h \sigma^\la (t)  \rangle=\langle  D_\al^h E \dot{u}_\la (t),   \varphi^2 D_\al^h \sigma^\la (t) \rangle.
\end{equation}
We integrate this equation in the time interval $[0,t]$. From \eqref{Acoercivity} it follows that
\begin{align}
&\int_0^t \langle  D_\al^h \mathbb{A}_r \dot{\sigma}^\la (s), \varphi^2 D_\al^h \sigma^\la (s) \rangle ds  = \frac{1}{2} \langle D_\al^h  \mathbb{A}_r  \sigma^\la (t), \varphi^2 D_\al^h  \sigma^\la (t) \rangle-\frac{1}{2} \langle D_\al^h  \mathbb{A}_r \sigma_0, \varphi^2 D_\al^h  \sigma_0 \rangle \notag \\ \label{15lug1}
& \geq \al_\mathbb{A} \Vert \varphi D_\al^h \sigma^\la (t) \Vert_{L^2}^2-\be_\mathbb{A} \Vert \varphi D_\al^h \sigma_0 \Vert_{L^2}^2.
\end{align}
The second term on the left-hand side of \eqref{lambdakinemadmissibility} can be rewritten as
\begin{align}\label{15lug2}
&\nonumber\int_0^t \langle  D_\al^h D \psi_\la (\sigma^\la (s)), \varphi^2 D_\al^h \sigma^\la (s)  \rangle ds \\
&= \int_0^t \int_0^1 \int_\Om \varphi^2(x')D^2 \psi_\la (\sigma^\la (s,x)+rh D_\al^h \sigma^\la (s,x))D_\al^h \sigma^\la (s,x):D_\al^h \sigma^\la (s,x) dx dr ds.
\end{align}
In particular it is positive, because
\begin{equation} \label{coercivityD^2lambda}
D^2 \psi_\la (\xi) \eta: \eta \geq \frac{1}{\al_0^{N-1}} ( \vert \xi \vert_r^{N-2} \wedge \la^{N-2})\vert \eta \vert_r^2,
\end{equation}
for all $\xi$, $\eta\in \Mtwo$.
Let us focus on the right-hand side of \eqref{lambdakinemadmissibility}. Since $u^\la(t) \in KL (\Om)$  we get 
\begin{align} \label{secondthirdterm}
\langle  D_\al^h E \dot{u}^\la (t),   \varphi^2 D_\al^h \sigma^\la (t) \rangle \nonumber=&\langle  D_\al^h E \dot{w} (t),   \varphi^2 D_\al^h \sigma^\la (t) \rangle+\langle  D_\al^h E ( \dot{\bar{u}}^\la (t)-\dot{\bar{w}}(t)),   \varphi^2 D_\al^h \bar{\sigma}^\la (t) \rangle \nonumber\\ 
&-\frac{1}{12} \langle  D_\al^h D^2 ( \dot{{u}}_3^\la (t)-\dot{w}_3 (t)),   \varphi^2 D_\al^h \hat{\sigma}^\la (t) \rangle. 
\end{align}
Integration by parts yields
\begin{align*}
&\langle  D_\al^h E ( \dot{\bar{u}}^\la (t)-\dot{\bar{w}}(t)),   \varphi^2 D_\al^h \bar{\sigma}^\la (t) \rangle=-\int_\om \partial_\be \varphi^2 (x') D_\al^h( \dot{\bar{u}}^\la -\dot{\bar{w}})_\gamma(t,x') D_\al^h \bar{\sigma}^\la_{\be \gamma} (t,x') dx' \nonumber\\
& +\int_\om \varphi^2 (x') D_\al^h( \dot{\bar{u}}^\la -\dot{\bar{w}})_\gamma(t,x') D_\al^h f_{\gamma} (t,x') dx'\nonumber\\
\end{align*}
where we implicitely assume sum over the indices $\beta$ and $\gamma$, whereas $\alpha$ is kept fixed. This last expression can be rewritten as
\begin{align}\label{comb1}
& =-\int_\om \int_0^1 \partial_\be \varphi^2 (x') \partial_\al ( \dot{\bar{u}}^\la -\dot{\bar{w}})_\gamma(t,x'+rhe_\al) D_\al^h (\bar{\sigma}_{\be \gamma}^\la-\bar\varrho_{\be \gamma}) (t,x') dr dx' \nonumber\\
& +\int_\omega\varphi^2 (x') D_\al^h( E\dot{\bar{u}}^\la -E\dot{\bar{w}})_{\be \gamma}(t,x') D_\al^h \bar\varrho_{\be \gamma} (t,x') dx'\nonumber \\
&=- 2 \int_\om \int_0^1 \partial_\be \varphi^2 (x') E ( \dot{\bar{u}}^\la -\dot{\bar{w}})_{\al \gamma}(t,x'+rhe_\al) D_\al^h (\bar{\sigma}_{\be \gamma}^\la-\bar\varrho_{\be \gamma}) (t,x') dr dx' \nonumber\\
&+ \int_\om \int_0^1 \partial_\be \varphi^2 (x') \partial_\gamma ( \dot{\bar{u}}^\la -\dot{\bar{w}})_\al (t,x'+rhe_\al) D_\al^h (\bar{\sigma}_{\be \gamma}^\la-\bar\varrho_{\be \gamma}) (t,x') dr dx' \nonumber\\
& +\int_\omega\varphi^2 (x') D_\al^h( E\dot{\bar{u}}^\la -E\dot{\bar{w}})_{\be \gamma}(t,x') D_\al^h \bar\varrho_{\be \gamma} (t,x') dx'\nonumber \\
&=- 2 \int_\om \int_0^1 \partial_\be \varphi^2 (x') E ( \dot{\bar{u}}^\la -\dot{\bar{w}})_{\al \gamma}(t,x'+rhe_\al) D_\al^h (\bar{\sigma}_{\be \gamma}^\la-\bar\varrho_{\be \gamma})(t,x') dr dx' \nonumber\\
&- \int_\om \int_0^1 \partial_{\be\gamma}^2 \varphi^2 (x')  ( \dot{\bar{u}}^\la -\dot{\bar{w}})_\al (t,x'+rhe_\al) D_\al^h (\bar{\sigma}_{\be \gamma}^\la-\bar\varrho_{\be \gamma}) (t,x') dr dx' \nonumber\\
& +\int_\omega\varphi^2 (x') D_\al^h( E\dot{\bar{u}}^\la -E\dot{\bar{w}})_{\be \gamma}(t,x') D_\al^h \bar\varrho_{\be \gamma} (t,x') dx'\nonumber \\
&=- 2 \int_\om \int_0^1 \partial_\be \varphi^2 (x') E ( \dot{\bar{u}}^\la -\dot{\bar{w}})_{\al \gamma}(t,x'+rhe_\al) D_\al^h (\bar{\sigma}_{\be \gamma}^\la-\bar\varrho_{\be \gamma})(t,x') dr dx' \nonumber\\
&+  \int_\om \int_0^1 D_\al^{-h} \partial_{\be\gamma}^2 \varphi^2 (x')  ( \dot{\bar{u}}^\la -\dot{\bar{w}})_\al (t,x'+rhe_\al)  (\bar{\sigma}_{\be \gamma}^\la-\bar\varrho_{\be \gamma})(t,x') dr dx' \nonumber\\
&+  \int_\om \int_0^1 \int_0^1 \partial_{\be\gamma}^2 \varphi^2 (x'-he_\alpha) \partial_\al ( \dot{\bar{u}}^\la -\dot{\bar{w}})_\al (t,x'+(r-q)he_\al)  (\bar{\sigma}_{\be \gamma}^\la-\bar\varrho_{\be \gamma}) (t,x') dq dr dx'\nonumber\\
& +\int_\omega\varphi^2 (x') D_\al^h( E\dot{\bar{u}}^\la -E\dot{\bar{w}})_{\be \gamma}(t,x') D_\al^h \bar\varrho_{\be \gamma} (t,x') dx'.
\end{align}
where in the second equality we have  used that $- \partial_\al u_\gamma=-2 (E u)_{\al \gamma} + \partial_\gamma u_\al,$ while in the last equality we have used \eqref{D2}.

As for the third term on the right-hand side of \eqref{secondthirdterm}, we write
\begin{align}\label{comb2}
&-\frac{1}{12} \langle  D_\al^h D^2 ( \dot{{u}}_3^\la (t)-\dot{w}_3 (t)),   \varphi^2 D_\al^h \hat{\sigma}^\la (t) \rangle= \nonumber \\
&=- \langle \varphi^2 D_\al^h ( \dot{u}^\la_3 (t)-\dot{w}_3(t)),D_\al^h ( \ddot{u}^\la_3 (t)-\ddot{w}_3(t)) \rangle - \langle \varphi^2 D_\al^h ( \dot{u}^\la_3 (t)-\dot{w}_3(t)),D_\al^h  \ddot{w}_3(t)\rangle \nonumber\\
& + \frac{1}{6} \int_\om \int_0^1 \partial_{\be} \varphi^2 (x') \partial_{\al\gamma}^2 ( \dot{u}_3^\la -\dot{w}_3)(t,x'+rhe_\alpha) D_\al^h   (\hat{\sigma}_{\be \gamma}^\la-\hat\varrho_{\be \gamma}) (t,x') dr dx' \nonumber\\
&-\frac{1}{12} \int_\om \int_0^1 D_\al^{-h} \partial_{\be \gamma}^2 \varphi^2 (x')  \partial_\al (\dot{u}^\lambda_3 -\dot{w}_3 )(t,x'+rh e_\al )(\hat{\sigma}_{\be \gamma}^\la-\hat\varrho_{\be \gamma})  (t,x')dr dx'  \nonumber\\
&-\frac{1}{12} \int_\om \int_0^1 \int_0^1 \partial_{\be \gamma}^2 \varphi^2 (x'-he_\al)  \partial^2_\al (\dot{u}_3^\lambda-\dot{w}_3) (t,x'+(r-q)h e_\al ))(\hat{\sigma}_{\be \gamma}^\la-\hat\varrho_{\be \gamma})  (t,x') dqdr dx'\nonumber\\
& -\frac{1}{12}\int_\om  D_\al^h \partial^2_{\be \gamma} ( \dot{{u}}_3^\la (t,x')-\dot{w}_3 (t,x'))  \varphi^2(x') D_\al^h \hat{\varrho}_{\be \gamma} (t,x')dx' .  
\end{align}
Combining \eqref{comb1} and \eqref{comb2} and using that  $u \in KL (\Om)$ we obtain
\begin{align}
&\langle  D_\al^h E ( \dot{\bar{u}}^\la (t)-\dot{\bar{w}}(t)),   \varphi^2 D_\al^h \bar{\sigma}^\la (t) \rangle-\frac{1}{12} \langle  D_\al^h D^2 ( \dot{{u}}_3^\la (t)-\dot{w}_3 (t)),   \varphi^2 D_\al^h \hat{\sigma}^\la (t) \rangle=\nonumber\\
&=- 2 \int_\Om \int_0^1 \partial_\be \varphi^2 (x') E ( \dot{{u}}^\la -\dot{{w}})_{\al \gamma}(t,x+rhe_\al) D_\al^h ({\sigma}^\la_{\be \gamma} -\varrho_{\be \gamma}) (t,x) dr dx \nonumber\\
&+  \int_\Om \int_0^1 D_\al^{-h} \partial_{\be\gamma}^2 \varphi^2 (x')  ( \dot{{u}}^\la -\dot{{w}})_\al (t,x+rhe_\al)  ({\sigma}^\la_{\be \gamma} -\varrho_{\be \gamma})(t,x) dr dx \nonumber\\
&+  \int_\Om \int_0^1 \int_0^1 \partial_{\be\gamma}^2 \varphi^2 (x'-he_\al) \partial_\al ( \dot{{u}}^\la -\dot{{w}})_\al (t,x+(r-q)he_\al)  ({\sigma}^\la_{\be \gamma} -\varrho_{\be \gamma}) (t,x) dq dr dx\nonumber\\
&- \langle \varphi^2 D_\al^h ( \dot{u}^\la_3 (t)-\dot{w}_3(t)),D_\al^h ( \ddot{u}^\la_3 (t)-\ddot{w}_3(t)) \rangle - \langle \varphi^2 D_\al^h ( \dot{u}^\la_3 (t)-\dot{w}_3(t),D_\al^h  \ddot{w}_3(t)) \rangle. \nonumber\\
& +\int_\Omega\varphi^2 (x') D_\al^h( E\dot{{u}}^\la -E\dot{{w}})_{\be \gamma}(t,x') D_\al^h \varrho_{\be \gamma} (t,x') dx'.
\end{align}
Using this, \eqref{lambdakinemadmissibility},  \eqref{15lug1}, \eqref{15lug2}, \eqref{secondthirdterm}, the computations carried on so far  lead to
\begin{align} \label{primastimaregol}
&\al_\mathbb{A} \Vert \varphi D_\al^h \sigma^\lambda(t) \Vert_{L^2}^2-\be_\mathbb{A} \Vert \varphi D_\al^h \sigma_0 \Vert_{L^2}^2 \nonumber\\
 &+\int_0^t \int_0^1  \varphi^2(x')D^2 \psi_\la (\sigma^\la (s)+rh D_\al^h \sigma^\la (s))D_\al^h \sigma^\la (s):D_\al^h \sigma^\la (s) dr ds \nonumber\\
 &+\frac{1}{2} \Vert \varphi D_\al^h \dot{u}^\lambda_3 (t)- \varphi D_\al^h \dot{w}_3(t)) \Vert_{L^2}^2- \frac{1}{2} \Vert \varphi D_\al^h (v_0)_3 (t)-\varphi D_\al^h \dot{w}_3(0) \Vert_{L^2}^2\nonumber \\
 &\leq \int_0^t \Big( \langle D_\al^h E \dot{w}(s), \varphi^2 D_\al^h \sigma^\la (s) \rangle- \langle \varphi^2 D_\al^h \ddot{w}_3 (s), D_\al^h ( \dot{u}^\lambda_3 (s)- \dot{w}_3 (s) ) \rangle \Big) ds\nonumber \\
&- 2 \int_0^t \int_\Om \int_0^1 \partial_\be \varphi^2 (x')  ( \mathbb{A}_r \dot{\sigma}^\la +D \psi_\la (\sigma^\la))_{\al \gamma}(s,x+rhe_\al) D_\al^h ({\sigma}_{\be \gamma}^\la- {\varrho}_{\be \gamma}) (s,x) dr dx ds \nonumber\\
&+ \int_0^t \int_\Om \int_0^1 D_\al^{-h} \partial_{\be\gamma}^2 \varphi^2 (x')  ( \dot{{u}}^\la -\dot{{w}})_\al (s,x+rhe_\al)  ({\sigma}_{\be \gamma}^\la- {\varrho}_{\be \gamma}) (s,x) dr dx ds \nonumber\\
&+ \int_0^t \int_\Om \int_0^1 \int_0^1 \partial_{\be\gamma}^2 \varphi^2 (x'-he_\al) ( \mathbb{A}_r \dot{\sigma}^\la +D \psi_\la (\sigma^\la))_{\al \al} (s,x+(r-q)he_\al)  ({\sigma}_{\be \gamma}^\la- {\varrho}_{\be \gamma}) (t,x) dq dr dx ds\nonumber \\
&+ 2 \int_0^t \int_\Om \int_0^1 \partial_\be \varphi^2 (x')  ( E \dot{w})_{\al \gamma}(s,x+rhe_\al) D_\al^h ({\sigma}_{\be \gamma}^\la- {\varrho}_{\be \gamma}) (s,x) dr dx ds \nonumber\\  
&- \int_0^t \int_\Om \int_0^1 \int_0^1 \partial_{\be\gamma}^2 \varphi^2 (x'-he_\al) \partial_\al \dot{w}_\al  (s,x+(r-q)he_\al)  ({\sigma}_{\be \gamma}^\la- {\varrho}_{\be \gamma}) (s,x) dq dr dx ds\nonumber\\
&- \int_0^t \int_\Om \int_0^1\partial_\alpha \varphi^2(x')( \mathbb{A}_r \dot{\sigma}^\la +D \psi_\la (\sigma^\la)-E\dot w)_{\beta \gamma}(s,x+rhe_\al)D^h_\al\varrho_{\be\gamma}(s,x)drdxds\nonumber\\
&- \int_0^t \int_\Om \int_0^1\varphi^2(x')( \mathbb{A}_r \dot{\sigma}^\la +D \psi_\la (\sigma^\la)-E\dot w)_{\beta \gamma}(s,x+rhe_\al)D^h_\al\partial_\alpha \varrho_{\be\gamma}(s,x)drdxds,
\end{align}
for all $t\in[0,T]$.
From \eqref{Dpsilambdalipschitz} it follows that
\begin{align*}
&\al_\mathbb{A} \Vert \varphi D_\al^h \sigma^\lambda(t) \Vert_{L^2}^2-\be_\mathbb{A} \Vert \varphi D_\al^h \sigma_0 \Vert_{L^2}^2 +\frac{1}{2} \Vert \varphi D_\al^h \dot{u}^\lambda_3 (t)- \varphi D_\al^h\dot{w}_3(t)) \Vert_{L^2}^2 \\
&- \frac{1}{2} \Vert \varphi D_\al^h (v_0)_3 (t)-\varphi D_\al^h \dot{w}_3(0)) \Vert_{L^2}^2 \\
& \leq C \Big( \int_0^t (\Vert \varphi D_\al^h \sigma^\la (s) \Vert_{L^2}+\Vert  \sigma^\la (s) \Vert_{L^2})(\Vert  \dot{\sigma}^\la (s) \Vert_{L^2} +\frac{\la^{N-2}}{\al_0^{N-1}} \Vert  \sigma^\la (s) \Vert_{L^2} ) ds \\
& +  \int_0^t (\Vert \varphi D_\al^h \varrho (s) \Vert_{L^2}+\Vert  \varrho (s) \Vert_{L^2})(\Vert  \dot{\sigma}^\la (s) \Vert_{L^2} +\frac{\la^{N-2}}{\al_0^{N-1}} \Vert  \sigma^\la (s) \Vert_{L^2} ) ds \\
&+\int_0^t(\|\varphi D_\al^h\sigma^\la(s)\|_{L^2}+\|\varphi D_\al^h\varrho(s)\|_{L^2})\|E\dot w(s)\|_{L^2}+\|D_\al^hE\dot w(s)\|_{L^2}\|\varphi D^h_\al \sigma^\la(s)\|_{L^2}ds\nonumber\\
&+\int_0^t ( \Vert \dot{u}^\la (s) \Vert_{L^2}+\Vert  E \dot{w} (s) \Vert_{L^2}+ \Vert  \dot{w} (s) \Vert_{L^2})(\Vert  \sigma^\la (s) \Vert_{L^2}+ \Vert  \varrho (s) \Vert_{L^2}) ds \\
& + \int_0^t \Vert \varphi D_\al^h(\dot{u}^\lambda_3 (s)- \dot{w}_3 (s)) \Vert_{L^2} \Vert D_\al^h \ddot{w}_3 (s) \Vert_{L^2}ds\\
& +  \int_0^t (\Vert \varphi D_\al^h \partial_\alpha\varrho (s) \Vert_{L^2}+\|\varphi D^h_\al \varrho^\la(s)\|_{L^2})(\Vert  \dot{\sigma}^\la (s) \Vert_{L^2} +\frac{\la^{N-2}}{\al_0^{N-1}} \Vert  \sigma^\la (s) \Vert_{L^2} +\|E\dot w(s)\|_{L^2}) ds \Big).
\end{align*}
Thanks to \eqref{firstlambdaestimate} and \eqref{secondlambdaestimate} we have $(\sigma^\la, \dot{u}^\la) \in H^1([0,T]; L^2(\Om;\Mtwo)) \times L^2([0,T]; L^2(\Om; \R^3))$. Using the regularity of the external data \eqref{regularityw} and \eqref{regularityvarrho}, the previous inequality implies that $ D_\al^h \sigma^\la (t) \in L^2_\mathrm{loc}( \om \times [ -\frac{1}{2},\frac{1}{2} ]; \Mtwo) $ for all $t \in [0,T]$, while we already know that  $D_\al^h u_3^\la (t) \in L^2_\mathrm{loc}( \om \times [ -\frac{1}{2},\frac{1}{2} ]; \Mtwo) $ for all $t \in [0,T]$ because $u(t) \in KL (\Om)$ and $BH(\om)$ embeds into $H^1(\om)$. Notice that these estimates are uniform in $t\in[0,T]$. 
Therefore we can pass to the limit in \eqref{primastimaregol} as $h$ tends to $0$, and thanks to \eqref{coercivityD^2lambda} we get
\begin{align} \label{secondastimaregol}
&\al_\mathbb{A} \Vert \varphi D_\al \sigma^\lambda(t) \Vert_{L^2}^2-\be_\mathbb{A} \Vert \varphi D_\al \sigma_0 \Vert_{L^2}^2+\frac{1}{\al_0^{N-1}}\int_0^t \int_\Om  \varphi^2( \vert \sigma^\la (s) \vert_r^{N-2} \wedge \la^{N-2} ) \vert D_\al \sigma^\la (s) \vert_r^2  dx ds \nonumber\\
 &+\frac{1}{2} \Vert \varphi D_\al  \dot{u}^\lambda_3 (t)- \varphi D_\al \dot{w}_3(t)) \Vert_{L^2}^2- \frac{1}{2} \Vert \varphi D_\al (v_0)_3- \varphi D_\al \dot{w}_3(0) \Vert_{L^2}^2 \nonumber\\
 &\leq \int_0^t \Big( \langle D_\al E \dot{w}(s), \varphi^2 D_\al \sigma^\la (s) \rangle- \langle \varphi^2 D_\al \ddot{w}_3 (s), D_\al (\dot{u}_3^\lambda (s)- \dot{w}_3 (s) ) \rangle \Big) ds \nonumber\\
&- 2 \int_0^t \int_\Om  \partial_\be \varphi^2 (x')  ( \mathbb{A}_r \dot{\sigma}^\la +D \psi_\la (\sigma^\la))_{\al \gamma}(s,x) D_\al ({\sigma}_{ \be \gamma}^\la-\varrho_{\be\gamma}) (s,x)  dx ds \nonumber\\
&+ \int_0^t \int_\Om   \partial^3_{\al \be\gamma} \varphi^2 (x')  ( \dot{{u}}^\la -\dot{{w}})_\al (s,x)  ({\sigma}_{ \be \gamma}^\la-\varrho_{\be\gamma}) (s,x) dx ds \nonumber\\
&+ \int_0^t \int_\Om  \partial^2_{\be\gamma} \varphi^2 (x') ( \mathbb{A}_r \dot{\sigma}^\la +D \psi_\la (\sigma^\la))_{\al \al} (s,x)  ({\sigma}_{ \be \gamma}^\la-\varrho_{\be\gamma}) (s,x)  dx ds \nonumber\\
&+ 2 \int_0^t \int_\Om  \partial_\be \varphi^2 (x')  ( E \dot{w})_{\al \gamma}(s,x) D_\al ({\sigma}_{ \be \gamma}^\la-\varrho_{\be\gamma}) (s,x)  dx ds \nonumber\\  
&- \int_0^t \int_\Om  \partial^2_{\be\gamma} \varphi^2 (x') \partial_\al \dot{w}_\al  (s,x)  ({\sigma}_{ \be \gamma}^\la-\varrho_{\be\gamma}) (s,x)  dx ds\nonumber\\
&- \int_0^t \int_\Om\partial_\alpha \varphi^2(x')( \mathbb{A}_r \dot{\sigma}^\la +D \psi_\la (\sigma^\la)-E\dot w)_{\beta \gamma}(s,x)D_\al\varrho_{\be\gamma}(s,x)dxds\nonumber\\
&- \int_0^t \int_\Om \varphi^2(x')( \mathbb{A}_r \dot{\sigma}^\la +D \psi_\la (\sigma^\la)-E\dot w)_{\beta \gamma}(s,x)D_\al\partial_\alpha \varrho_{\be\gamma}(s,x)dxds.
\end{align} 
The first, sixth, and seventh term in the right-hand side of \eqref{secondastimaregol} can be estimated using Cauchy-Schwartz and H\"older inequality. Similarly the second term is estimated by 
\begin{align}
 \frac{1}{2}\|\varphi D_\al\ddot w_3\|^2_{L^2(L^2)}+\frac{1}{2}\int_0^t\Vert \varphi D_\al  \dot{u}^\lambda_3 (s)- \varphi D_\al \dot{w}_3(s)) \Vert_{L^2}^2ds.
\end{align}
Now we consider the third term on the right-hand side of \eqref{secondastimaregol}.
Using the expression of $D \psi_\lambda$, we get
\begin{align*}
&\int_0^t \int_\Om \vert \partial_\be \varphi^2 (x')  ( \mathbb{A}_r \dot{\sigma}^\la +D \psi_\la (\sigma^\la))_{\al \gamma}(s,x) D_\al {\sigma}_{ \be \gamma}^\la (s,x) \vert dx ds \\
& \leq C \int_0^t \Vert \varphi D_\al \sigma^\la (s) \Vert_{L^2} \Vert \dot{\sigma}^\la (s) \Vert_{L^2}ds \\
&+ \frac{C}{\al_0^{N-1}} \int_0^t \int_\Om   ( \vert \sigma^\la (s,x) \vert_r^{N-2} \wedge \la^{N-2} ) \vert  \sigma^\la (s,x) \vert_r \vert \varphi(x') D_\al \sigma^\la (s,x) \vert dx ds \\
& \leq C \left( \int_0^t \Vert \varphi D_\al \sigma^\la (s) \Vert_{L^2}^{2} ds \right)^{1/2}+\frac{C}{\al_0^{\frac{N-1}{2}}}\left( \int_0^t \int_\Om | \vert \sigma^\la (s) \vert_r^{N-2} \wedge \la^{N-2} |\;|\varphi D_\al \sigma^\la (s)|^2 dx ds \right)^{1/2}\\
&\leq  C \int_0^t \Vert \varphi D_\al \sigma^\la (s) \Vert_{L^2}^{2} ds+C+\frac{1}{2\al_0^{N-1}}\int_0^t \int_\Om | \vert \sigma^\la (s) \vert_r^{N-2} \wedge \la^{N-2} |\; |\varphi D_\al \sigma^\la (s)|_r^2 dx ds,
\end{align*}
where we have used \eqref{firstlambdaestimate}, \eqref{secondlambdaestimate}, and \eqref{stimatroncata}. In such a way the last term is absorbed by the corresponding one in the left-hand side of \eqref{secondastimaregol}. The part of the third term in the right-hand side of \eqref{secondastimaregol} containing $D_\alpha\varrho$ is treated together with the last two terms.

The estimate for the fifth term is analogous and is based on the boundedness \eqref{stimatroncata}. Let us treat the fourth term. To this aim we recall that $\dot{u}^\la$ is uniformly bounded with respect to $\la$ and $N$ in $L^1([0,T];L^2(\Om))$ by \eqref{sanvalentino}.

Therefore
\begin{align*}
& \int_0^t \int_\Om   \partial^3_{\al \be\gamma} \varphi^2 (x')  ( \dot{{u}}^\la -\dot{{w}})_\al (s,x)  ({\sigma}_{\be \gamma}^\la (s,x)-{\varrho}_{\be \gamma} (s,x)) dx ds\\
&\leq C \| \dot{{u}}^\la -\dot{{w}}\|_{L^1(L^2)} \Vert {\sigma}^\la-\varrho  \Vert_{L^\infty(L^2)}\leq C.
\end{align*}
where $C$ is independent of $\la$ and $N$. We finally study the last two terms in \eqref{secondastimaregol}. Since $D\psi_\la(\sigma_\lambda)$ is uniformly bounded in $L^1([0,T];L^1(\Om;\Mtwo))$, it suffices to observe that $\varrho$ belongs to $L^\infty([0,T];W^{2,\infty}_\text{loc}(\om\times [-\frac12,\frac12];\Mtwo))$.

Combining all these estimates and the hypotheses we conclude
\begin{align*}
&\al_\mathbb{A} \Vert \varphi D_\al \sigma^\lambda(t) \Vert_{L^2}^2-\be_\mathbb{A} \Vert \varphi D_\al \sigma_0 \Vert_{L^2}^2+\frac{1}{2\al_0^{N-1}}\int_0^t \int_\Om  \varphi^2( \vert \sigma^\la (s) \vert_r^{N-2} \wedge \la^{N-2} ) \vert D_\al \sigma^\la (s) \vert_r^2  dx ds \nonumber\\
 &+\frac{1}{2} \Vert \varphi D_\al  \dot{u}^\lambda_3 (t)- \varphi D_\al \dot{w}_3(t)) \Vert_{L^2}^2- \frac{1}{2} \Vert \varphi D_\al (v_0)_3- \varphi D_\al \dot{w}_3(0) \Vert_{L^2}^2 \nonumber\\ 
 &\leq C\Big(1+\int_0^t \Vert \varphi D_\al \sigma^\la (s) \Vert_{L^2}^{2} ds+\int_0^t r(s)\big(\Vert \varphi D_\al \sigma^\la (s) \Vert_{L^2}+\Vert \varphi D_\al  \dot{u}^\lambda_3 (s)- \varphi D_\al \dot{w}_3(s)) \Vert_{L^2}\big)ds\Big),
\end{align*}
where the constant $C$, as well as $\|r\|_{L^1([0,T])}$, are independent of $\la$ and $N$. Therefore, by applying the Gronwall Lemma we infer that there is a constant $C>0$ independent of $\la$ and $N$ such that 
\begin{align}\label{estimateH1}
  \Vert \varphi D_\al \sigma^\lambda(t) \Vert_{L^2}^2+\| \varphi D_\al  \dot{u}^\lambda_3 (t)\|_{L^2}^2\leq C,
\end{align}
for all $t\in[0,T]$. Hence, by arbitrariness of $\varphi$, we conclude that 
for every open set $\om'$ compactly contained in $\om$ the sequences $(D_\al \sigma^\la)_\la$ and $ (D_\al  \dot{u}^\lambda_3)_\la$ are both uniformly bounded in $ L^\infty ([0,T]; L^2( \om' \times [-\frac{1}{2},\frac{1}{2}]))$.  Owing to \eqref{convsigmalasigmaN} and \eqref{convu_3infinito},
$D_\al \sigma^N$ and $D_\al \dot u_3^N$ belong to  $ L^\infty ([0,T]; L^2( \om' \times [-\frac{1}{2},\frac{1}{2}]))$, for every open set $\om' $ compactly contained in $\om$. 

\medbreak
\noindent\textit{Step 2.} 
By \eqref{convsigmalasigmaN} and \eqref{convu_3infinito} we know that 
\begin{align}
 \varphi\sigma^\la(t)\rightharpoonup \varphi\sigma^N(t)\text{ weakly in }L^2(\Om),\\
 \varphi \dot u_3^\la(t)\rightharpoonup \varphi\dot u_3^N(t)\text{ weakly in }L^2(\om),
\end{align}
for all $t\in[0,T]$. Estimate \eqref{estimateH1} entails that the previous convergences take place in the $H^1(\Om)$ and $H^1(\om)$ weak topologies, respectively. Therefore by lower-semicontinuity we conclude 
\begin{align}
  \Vert \varphi D_\al \sigma^N(t) \Vert_{L^2}^2+\| \varphi D_\al \dot{u}^N_3 (t)\|_{L^2}^2\leq C,
\end{align}
for all $t\in[0,T]$, with the constant $C>0$ independent of $N$. We have proved that $(D_\al \sigma^N)_N$  and $(D_\al \dot{u}^N_3)_N$ are uniformly bounded in $ L^\infty ([0,T]; L^2( \om' \times [-\frac{1}{2},\frac{1}{2}]))$ with respect to $N$.

\medbreak
\noindent\textit{Step 3.} Now we prove higher regularity with respect to the third coordinate $x_3$, namely \eqref{regestimatesigma3N}. Given $\varphi \in C_c^\infty (\Om)$, we test the first equation in \eqref{NHproblem} by $D_3^{-h} (\varphi^2 D_3^h \sigma^N (t))$. We get
\begin{equation*}
\langle  D_3^h \mathbb{A}_r \dot{\sigma}^N (t), \varphi^2 D_3^h \sigma^N (t) \rangle + \langle  D_3^h D \phi_N (\sigma^N (t)), \varphi^2 D_3^h \sigma^N (t)  \rangle=-\langle  D^2 \dot{u}_3^N (t),   \varphi^2 D_3^h \sigma^N (t) \rangle.
\end{equation*}
 The right-hand side can be integrated by parts as
\begin{align}
 &-\langle  D^2 \dot{u}_3^N,   \varphi^2 D_3^h \sigma^N \rangle=-\langle  \partial^2_{\beta\gamma} \dot{u}_3^N,  \frac{1}{12} \varphi^2  (\hat\sigma^N -\hat\varrho)_{\beta\gamma}\rangle-\langle  \partial^2_{\beta\gamma} \dot{u}_3^N,  \frac{1}{12} \varphi^2  \hat\varrho_{\beta\gamma}\rangle\nonumber\\
 &=-\langle\varphi^2(\dot u^N_3-\dot w_3),\ddot u_3^N-\ddot w_3\rangle-\langle\varphi^2\ddot w_3,(\dot u^N_3-\dot w_3)\rangle-\langle D^2\dot w_3,\varphi^2\frac{1}{12}(\hat \sigma^N-\hat\varrho)\rangle\nonumber\\
 &-\langle D^2\dot u_3^N,\varphi^2\frac{1}{12}\hat\varrho\rangle-\langle \dot u^N_3-\dot w_3,\partial^2_{\beta\gamma}\varphi^2\frac{1}{12}(\hat \sigma^N-\hat\varrho)_{\beta\gamma}\rangle\nonumber\\
 &-\langle \dot u^N_3-\dot w_3,\partial_{\beta}\varphi^2\frac{1}{6}\partial_\gamma(\hat \sigma^N-\hat\varrho)_{\beta\gamma}\rangle
\end{align}

Integrating this equation with respect to time, using \eqref{Acoercivity}, \eqref{4.48}, and the regularity of $w$ and $\varrho$, we infer
\begin{align} \label{quartastimaregol}
&\al_\mathbb{A} \Vert \varphi D_3^h \sigma^N(t) \Vert_{L^2}^2-\be_\mathbb{A} \Vert \varphi D_3^h \sigma_0 \Vert_{L^2}^2\nonumber \\
&+\int_0^t \int_0^1 \varphi^2 (x') D^2 \phi_N (\sigma^N (s,x)+rh D_3^h \sigma^N (s,x))D_3^h \sigma^N (s,x):D_3^h \sigma^N (s,x) dr dx ds\nonumber\\
 &+\frac{1}{2}\|\varphi(\dot u_3(t)-\dot w_3(t))\|^2_{L^2}-\frac{1}{2}\|\varphi(\dot u_3)_0-\dot w_3(0))\|^2_{L^2}\nonumber\\
&\leq C+C\int_0^t\|\sigma^N(s)\|_{L^2}^2+\|\varphi D\sigma^N(s)\|_{L^2}^2ds.
 \end{align}
 Using the estimate obtained in Step 2, the Gronwall Lemma implies that the right hand-side of \eqref{quartastimaregol} is uniformly bounded with respect to $N$. Hence we have
\begin{align*}
&\al_\mathbb{A} \Vert \varphi D_3^h \sigma^N(t) \Vert_{L^2}^2 \leq \be_\mathbb{A} \Vert \varphi D_3^h \sigma_0 \Vert_{L^2}^2 +C,
\end{align*}
with $C$ independent of $h$ and $N$. Passing to the limit as $h$ tends to $0$ and using the assumptions on $\sigma_0$ we get
\begin{align*}
&\al_\mathbb{A} \Vert \varphi D_3 \sigma^N(t) \Vert_{L^2}^2 \leq C,
\end{align*}
with $C$ independent of $N$. This concludes the proof.
\end{proof}

\section{Existence and regularity result}
We are now in a position to prove the main result of the paper, which states the existence of a solution of the dynamic evolution together with higher spatial regularity for the stress field $\sigma$ and time regularity for the velocity $\dot u_3$. 
\begin{theorem} \label{mainresult}
Assume that the hypotheses of Lemma \ref{Lemmalambda} are satisfied and that $$\sigma_0 \in H^1(\Om; \Mtwo).$$ Then there exists a triplet
$$ (u,e,p) \in H^1([0,T]; KL (\Om) \times L^2 (\Om; \Mtwo) \times M_b (\Om \cup \Gd; \Mtwo)), $$
with $$u_3\in H^2([0,T];L^2(\om)),$$
 satisfying the initial conditions $(u(0),e(0),p(0))=(u_0, e_0,p_0)$ and $\dot{u}_3(0)=(v_0)_3$ and the following:
\smallskip

\begin{itemize}

\item[(i)] {\em kinematic admissibility:} $(u(t),e(t),p(t)) \in \mathcal{A}_{KL}(w(t))$ for every $t \in [0,T]$;
\medskip

\item[(ii)] {\em constitutive law:} $\sigma(t):=\mathbb C_r e(t)$;
\medskip

\item[(iii)] {\em equations of motion:} for every $t\in [0,T]$
\begin{equation} \label{limit-equilibrium}
-\Div\,\bar{\sigma}(t)= f(t) \quad\mbox{ a.e. in }  \omega, 
\end{equation}
and for a.e.\ $t \in [0,T]$
\begin{equation} \label{tangentialeqmotion}
\begin{cases}
\ddot{u}_3(t)-\frac{1}{12} \Div\,\Div\,\hat{\sigma}(t)= g(t) \mbox{ a.e. in }  \omega, \\
\sigma(t) \in \Theta (\gn, 0,0 );\\
\end{cases}
\end{equation}
\medskip

\item[(iv)] {\em stress constraint:} $\sigma(t) \in \mathcal{K}_r (\Om)$ for every $t\in [0,T]$;
\medskip
\item[(v)] {\em flow rule:} for a.e.\ $t\in [0,T]$
\begin{equation}\label{flow}
\HH_r(\dot{p}(t))= \left\langle \sigma(t),\dot{p}(t)\right\rangle_r.
\end{equation}

\end{itemize}
Moreover $\sigma (t)$ and $u_3 (t)$ are unique, and the following estimates hold true:
\begin{itemize}
\item for every open set $\om'$ compactly contained in $\om$ there exists a positive constant $C_1=C_1(\om')$   such that
\begin{equation} \label{regestimatesigmaal}
\sup_{t \in [0,T]} \Vert D_\alpha \sigma (t) \Vert_{L^2 ( \om ' \times \left( -\frac{1}{2},\frac{1}{2} \right))} \leq C_1,
\end{equation}   
\item for every open set $\Om'$ compactly contained in $\Om$ there exists a positive constant $C_2=C_2(\Om')$   such that
\begin{equation} \label{regestimatesigma3}
\sup_{t \in [0,T]} \Vert D_3 \sigma (t) \Vert_{L^2 ( \Om ' )} \leq C_2.
\end{equation}
\item for every open set $\om'$ compactly contained in $\om$ there exists a positive constant $C_3=C_3(\om')$ such that
\begin{equation} \label{reg_Du3finale}
\sup_{t \in [0,T]} \Vert D_\alpha \dot u_3(t) \Vert_{L^2 ( \om ')} \leq C_3.
\end{equation}   
\end{itemize} 
Finally, if the external force $f$ is independent of time $t$  also the following additional regularity conditions are satisfied
\begin{align}\label{higherregularity}
 &\sigma \in W^{1,\infty}([0,T];L^2(\Om; \Mtwo)),\;\; u\in W^{1,\infty}([0,T],BD(\Om)),\nonumber\\
 &u_3\in W^{2,\infty}([0,T];L^2(\om)),
\end{align}
and in particular
\begin{align*}
 (u,e,p) \in Lip([0,T]; KL (\Om) \times L^2 (\Om; \Mtwo) \times M_b (\Om \cup \Gd; \Mtwo)).
\end{align*}

\end{theorem} 

\begin{proof}
\textit{Step 1.} Let $(u^N,\sigma^N,p^N)$ be the solution of Theorem \ref{TheoremNHoff}. We want to pass to the limit as $N\rightarrow\infty$.
From \eqref{stimasigmaW1infty}, \eqref{4.48}, and \eqref{stimauN}, it follows that there exists
$$ (\sigma,u) \in H^1([0,T]; L^2(\Om; \Mtwo) \times  KL (\Om)) $$ 
with $u_3 \in H^2([0,T]; L^2(\om))$ such that, up to subsequences,
\begin{align} \label{sigmaNtosigma}
& \sigma^N \wto \sigma \quad \mbox{ weakly in } H^1([0,T]; L^2(\Om; \Mtwo)), \\
& u^N \wto u \quad \mbox{ weakly* in } H^1([0,T]; BD(\Om)),\label{5.9} \\ \label{u3Ntou3}
& \dot{u}_3^N \wto \dot{u}_3 \mbox{ weakly in } H^1([0,T]; L^2(\om)).
\end{align}
Condition (iii) then easily follows from \eqref{sigmaNtosigma} and \eqref{u3Ntou3}.
Moreover by \eqref{sigmaNtosigma} and \eqref{5.9}, for every $t \in [0,T]$, we have
\begin{align} 
& \sigma^N (t) \wto \sigma (t) \quad \mbox{ weakly in } L^2(\Om; \Mtwo),\label{5.11} \\
& u^N (t) \wto u (t) \quad \mbox{ weakly* in }  BD(\Om).\label{5.12}
\end{align}
Let us set $e(t) \doteq \mathbb A_r \sigma (t)$, $p(t) \doteq E u(t)-e(t)$ in $\Om$ and $p(t) \doteq (w(t)-u(t)) \odot \nu_{\partial \Om}$.
It is then clear that the initial conditions, and (i) and (ii) hold true. We now prove (iv). 
From  \eqref{stimasigmalinftylN}, for any Borel set $B\subset[0,T]\times\Om$ we infer by H\"older inequality
\begin{align}
 \int_B\frac{|\sigma^N|_r}{\alpha_0}dxdt\leq |B|^{\frac{N-1}{N}}\Big(\int_B\frac{|\sigma^N|_r^N}{\alpha_0^N}\Big)^{\frac{1}{N}}dxdt\leq \Big(\frac{C}{\alpha_0}\Big)^{\frac{1}{N}}|B|^{\frac{N-1}{N}}.
\end{align}
For every $\delta>0$ let $B_\delta$ be the Borel set defined as $B_\delta:=\{|\sigma|_r\geq \alpha_0(1+\delta)\}$. From the previous estimate we infer
\begin{align*}
 |B_\delta|(1+\delta)\leq\int_{B_\delta}\frac{|\sigma|_r}{\alpha_0}dxdt\leq \liminf_{N\rightarrow\infty}\int_{B_\delta}\frac{|\sigma^N|_r}{\alpha_0}dxdt\leq\liminf_{N\rightarrow\infty}\Big(\frac{C}{\alpha_0}\Big)^{\frac{1}{N}}|B_\delta|^{\frac{N-1}{N}}=|B_\delta|.
\end{align*}
Hence $|B_\delta|=0$ and by arbitrariness of $\delta$ we conclude (iv).

Let us now prove (v).
Let $\theta \in H^1([0,T]; L^2 (\Om; \Mtwo)) \cap L^\infty([0,T];L^N (\Om; \Mtwo))$ and let $\varphi \in C^2 (\bar{\om})$ with $\varphi \geq 0$. We test the first equality in \eqref{NHproblem} by $\varphi(\theta (t)-\sigma^N (t))$, namely
$$ \langle \mathbb{A}_r \dot{\sigma}^N (t)+ D \phi_N (\sigma^N (t))-E \dot{u}^N (t), \varphi(\theta (t)-\sigma^N (t)) \rangle=0. $$
By convexity of $\phi_N$ we infer
\begin{align} 
\langle \mathbb{A}_r \dot{\sigma}^N(t), \varphi (\sigma^N (t)-\theta(t)) \rangle+\int_\Om \varphi(x') \phi_N (\sigma^N (t,x)) dx \notag \\ \label{disug1flowrule}
 \leq \langle E \dot{u}^N (t), \varphi (\sigma^N (t)-\theta(t)) \rangle +\int_\Om \varphi(x') \phi_N (\theta (t,x)) dx,
\end{align}
for a.e. $t\in[0,T]$.
Now we choose $\theta=\sigma$ and $\varphi=1$ and we integrate on the time interval $[0,t]$, with $t\in[0,T]$. Since $\phi_N \geq 0$, rearranging terms and integrating by parts, we have
\begin{align}
&\frac{1}{2} \langle \mathbb{A}_r (\sigma^N (t)-\sigma (t), \sigma^N (t)-\sigma (t) \rangle+\frac{1}{2} \Vert \dot{u}_3^N (t)-\dot{u}_3 (t) \Vert_{L^2}^2 \nonumber\\
&\leq  \int_0^t \langle \sigma^N (s)-\sigma (s), E \dot{w}(s)-\mathbb{A}_r \dot{\sigma} (s) \rangle ds\nonumber\\
&\;\;+\int_0^t\langle \dot w_3-\dot u_3,\ddot u_3^N-\ddot u_3\rangle ds+\int_0^t \int_\Om \phi_N (\sigma (t,x)) dx dt. \label{strongconvformula}
\end{align}
From (iv) it follows that the last term on the right-hand side tends to $0$ as $N \to +\infty$.  Hence, taking the limit as $N \to +\infty$ we infer, thanks to \eqref{Acoercivity}, \eqref{sigmaNtosigma}, \eqref{u3Ntou3}, \eqref{5.11}, and \eqref{5.12},
\begin{align}
 &\sigma^N(t) \to \sigma(t) \quad\mbox{ strongly in }L^2(\Om; \Mtwo)\text{ for all }t\in[0,T],\\
 & \dot{u}_3^N(t) \to \dot{u}_3(t) \quad  \mbox{ strongly in } L^2(\om)\text{ for all }t\in[0,T].
\end{align}
Furthermore, integrating expression \eqref{strongconvformula} in time on $[0,T]$ we also conclude
\begin{align} \label{strongconvsigmaN}
& \sigma^N \to \sigma \quad  \mbox{ strongly in } L^2([0,T]; L^2(\Om; \Mtwo)), \\ \label{strongconvuN}
& \dot{u}_3^N \to \dot{u}_3 \quad  \mbox{ strongly in } L^2([0,T]; L^2(\om)).
\end{align}
We now go back to  \eqref{disug1flowrule} and choose an arbitrary $\theta \in \mathcal{K}_r (\Om) \cap \Sigma (\Om)$ independent of time and $\varphi\in C^2(\bar\omega)$ with $\varphi=0$ in a neighbourhood of $\gn$. Since $\varphi (\dot{u}^N(t)-\dot{w}(t))=0$ on $\partial \om \times \left( -\frac{1}{2},\frac{1}{2} \right)$, parts integration leads to
\begin{align*}
\langle E \dot{u}^N (t)&-E \dot{w} (t), \varphi (\sigma^N (t)-\theta) \rangle= \int_\om \varphi(x') (\dot{\bar{u}}^N (t,x')-\dot{\bar{w}}(t,x')) \cdot  (\Div \bar\theta(x') +f(t,x')) dx' \\
&-\int_\om \big(\nabla \varphi (x') \odot (\dot{{u}}^N (t,x')-\dot{{w}}(t,x'))\big): ( {\sigma}^N (t,x')-{\theta}(x'))dx'\\
&+\int_\om \varphi(x') (\dot{u}_3^N (t,x'))-\dot{w}_3 (t,x'))(g(t,x')- \ddot{u}_3^N(t,x')) dx'\\
&+\frac{1}{12} \int_\om \varphi(x') (\dot{u}_3^N (t,x'))-\dot{w}_3 (t,x'))\Div \Div \hat{\theta}(x') dx' \\
&+\frac{1}{12} \int_\om (\dot{u}_3^N (t,x')- \dot{w}_3 (t,x'))D^2 \varphi(x'): (\hat{\sigma}^N (t,x')-\hat{\theta}(x'))dx'.
\end{align*}
We substitute this expression into \eqref{disug1flowrule}, then we integrate with respect to time on an arbitrary time interval $[t_1,t_2]$,  and eventually we pass to the limit as $N \to+ \infty$. Notice that by \eqref{stimauN} we have $\dot u^N\rightharpoonup \dot u$ weakly in $L^2([0,T];L^2(\Om;\R^2))$, so that thanks to \eqref{u3Ntou3}, \eqref{strongconvsigmaN},  and \eqref{strongconvuN}, we arrive to
\begin{align*}
\int_{t_1}^{t_2}& \langle \mathbb{A}_r \dot{\sigma} (t)-E \dot{w}(t), \varphi (\sigma (t)- \theta) \rangle dt \\
&\leq \int_{t_1}^{t_2} \int_\om \varphi (\dot{\bar{u}} (t,x')-\dot{\bar{w}}(t,x')) \cdot  (\Div \bar\theta (x')- f(t,x')) dx'dt\\ &-\int_{t_1}^{t_2} \int_\om \big(\nabla \varphi(x') \odot (\dot{{u}} (t,x')-\dot{w}(t,x'))\big): ( {\sigma} (t,x')-{\theta}(x'))dx'dt\\
&+\int_{t_1}^{t_2}\int_\om \varphi(x') (\dot{u}_3^N (t,x'))-\dot{w}_3 (t,x'))(g(t,x')- \ddot{u}_3^N(t,x')) dx'dt\\
&+\frac{1}{12} \int_{t_1}^{t_2} \int_\om \varphi(x') (\dot{u}_3 (t,x')-\dot{w}_3 (t,x')) \Div \Div \hat{\theta}(x') dx'dt \\
&+\frac{1}{12} \int_{t_1}^{t_2} \int_\om (\dot{u}_3 (t,x')- \dot{w}_3 (t,x'))D^2 \varphi(x'): (\hat{\sigma} (t,x')-\hat{\theta}(x'))dx'dt.
\end{align*}
 Exploiting the integration by parts formula Corollary \ref{formulazzaintparti} and the fact that $-\frac{1}{12}\Div\Div \hat \sigma=g-\ddot u_3$, $-\Div\bar \sigma=f$, we see that this last expression is equivalent to
\begin{equation} \label{disugflowrule2}
\int_{t_1}^{t_2} \int_{\Om \cup \Gd} \varphi d[(\theta-\sigma (t)): \dot{p}(t)]dt \leq 0,
\end{equation}
for every $\theta \in \mathcal{K}_r (\Om) \cap \Sigma (\Om)$ and every $\varphi \in C^2 (\bar{\om})$, $\varphi \geq 0$, with $\varphi=0$ in a neighbourhood of $\gn$. For every $\delta > 0$ let now $\varphi_\delta \in C^2 (\bar{\om})$ be such that $0 \leq \varphi_\delta \leq 1$, $\varphi_\delta=0$ on the set $\lbrace x' \in \bar{\om}: \mathrm{dist}(x',\gn) < \delta \rbrace$ and $\varphi_\delta=1$ on the set $\lbrace x' \in \bar{\om}: \mathrm{dist}(x',\gn) > 2 \delta \rbrace$. Using $\varphi_\delta$ as a test function in \eqref{disugflowrule2} and sending $\delta$ to zero, we obtain
\begin{equation*} 
\int_{t_1}^{t_2}  \langle \theta-\sigma (t), \dot{p}(t) \rangle dt \leq 0.
\end{equation*}
By the arbitrariness of $t_1$ and $t_2$ we can localize such expression in time and then using \cite[Proposition 7.8]{DavMor} we infer \eqref{flow}, by arbitrariness of $\theta$.

\medbreak
\noindent\textit{Step 2.} Conditions  \eqref{regestimatesigmaal}, \eqref{regestimatesigma3}, and \eqref{reg_Du3finale} follow from \eqref{regestimatesigmaalN}, \eqref{regestimatesigma3N}, \eqref{reg_Du3}, \eqref{sigmaNtosigma}, and \eqref{u3Ntou3}.
Furthermore, if the force $f$ is independent of time $t$, then we can use Theorem \ref{TheoremNHoff2} to deduce that \eqref{higherregularity} holds true.

\medbreak
\noindent\textit{Step 3.}
To conclude the proof of the Theorem it remains to show that $\sigma$ and $u_3$ are unique. The proof is very similar to the uniqueness part of Lemma \ref{Lemmalambda}, but we sketch the lines for completeness.
Let $(u, e, p)$ and $(v,\eta,q)$ be two solutions, and let $\sigma(t)\doteq \mathbb{C}_r e(t)$ and $\tau(t)\doteq\mathbb{C}_r\eta(t)$. Subtracting the two equations of motion for $u_3$ and $v_3$ leads to 
$$ 
\ddot{u}_3(t)- \ddot{v}_3(t)-\frac{1}{12} \Div\,\Div\, (\hat\sigma(t)- \hat\tau(t))=0 \quad \text{ in } \omega
$$
for a.e.\ $t\geq0$. Multiplying this equation by $\dot{u}_3(t)- \dot{v}_3(t)$ and integrating on $[0, T]\times \omega$ yields
\begin{align} 
& \int_0^T \int_\omega (\ddot{u}_3(t,x')- \ddot{v}_3(t,x'))(\dot{u}_3(t.x')-\dot{v}_3(t,x'))\,dx'\,dt \\ \label{uniqu3eq}
& -\frac{1}{12} \int_0^T \int_\omega \Div\,\Div\, (\hat\sigma(t,x')- \hat\tau(t,x'))(\dot{u}_3(t,x')- \dot{v}_3(t,x'))\,dx'\,dt=0.
\end{align}
Since $\dot{u}_3 (0)=\dot{v}_3 (0)$, we have 
\begin{equation}\label{200}
\int_0^T \int_\omega (\ddot{u}_3(t,x')- \ddot{v}_3(t,x'))(\dot{u}_3(t,x')-\dot{v}_3(t,x'))\,dx'\,dt 
=\frac{1}{2} \| \dot{u}_3 (T)- \dot{v}_3 (T) \|_{L^2}^2 .
\end{equation}
On the other hand, by Corollary \ref{formulazzaintparti}, \eqref{limit-equilibrium}, and \eqref{tangentialeqmotion}, we obtain
\begin{align} 
&-\frac{1}{12} \int_0^T \int_\omega \Div\,\Div\, (\hat\sigma(t,x')- \hat\tau(t,x'))(\dot{u}_3(t,x')- \dot{v}_3(t,x'))\,dx'\,dt
\notag \\ \label{dualityformula}
&= \int_0^T \int_\Omega (\sigma(t,x)-\tau(t,x)): (\dot{e}(t,x)-\dot{\eta}(t,x))\,dx\,dt 
+\int_0^T \langle \sigma(t)-\tau(t),\dot{p}(t)-\dot{q}(t) \rangle \,dt,
\end{align}
where we have also used that $(\dot u(t) -\dot v(t), \dot e(t)-\dot\eta(t), \dot p(t)-\dot q(t))\in \mathcal{A}_{KL} (0)$ for a.e. $t\geq0$.
Since $e(0)=\eta(0)$, we have 
\begin{equation}
\int_0^T \int_\Omega (\sigma(t,x)-\tau(t,x)): (\dot{e}(t,x)-\dot{\eta}(t,x))\,dx\,dt 
=\frac{1}{2} \langle \mathbb{A}_r( \sigma(T)-\tau(T)), \sigma(T)-\tau(T) \rangle. 
\end{equation}
Moreover, using (v) and the fact that $\tau(t)\in K_r$ a.e.\ in $\Omega$, we infer that 
$$  
\langle \sigma(t)-\tau(t),\dot{p}(t) \rangle \geq 0
$$ 
for a.e.\ $t\geq0$. Similarly,
$$  
\langle \tau(t)-\sigma(t),\dot{q}(t) \rangle \geq 0
$$
for a.e.\ $t\geq0$. Summing up the previous inequalities and integrating in time yields 
\begin{equation} \label{signofp}
\int_0^T \langle \sigma(t)-\tau(t),\dot{p}(t)-\dot{q}(t) \rangle \,dt \geq 0. 
\end{equation} 
Gathering \eqref{uniqu3eq}--\eqref{signofp} we deduce that
$$ 
\frac{1}{2} \| \dot{u}_3 (T)- \dot{v}_3 (T) \|_{L^2}^2 +\frac{1}{2} \langle \mathbb{A}_r (\sigma(T)-\tau(T)),\sigma(T)-\tau(T)\rangle\leq 0. 
$$
By \eqref{Acoercivity} we conclude that $\dot{u}_3=\dot{v}_3$, hence $u_3=v_3$, and that $\sigma=\tau$.

\end{proof}

We conclude with two observations.
\begin{remark}
As pointed out in \cite[Remark 5.3]{MaggianiMora}, one cannot expect in general the uniqueness of the horizontal components of the displacement $u$ and of the plastic strain $p$.
\end{remark}
\begin{remark}
 Let us emphasize that from the condition $u\in H^1([0,T]; KL (\Om))$ it easily follows that $\dot u_3\in L^2([0,T]; H^1 (\om))$. This does not imply the regularity stated in \eqref{reg_Du3finale} since the latter involves an $L^\infty$-norm in time. Notice that the fact that $\ddot u_3\in L^2([0,T]; L^2(\om))$ only implies that $\dot u_3\in L^\infty([0,T]; L^2 (\om))$ without any additional spatial regularity.
\end{remark}

\subsubsection*{Acknowledgements}
P.G. and R.S. are supported by FCT--Fundação para a Ciência e
Tecnologia, under the project UID/MAT/04561/2013. The authors have been also partially supported by the
Gruppo Nazionale per l’Analisi Matematica, la Probabilità e le loro Applicazioni (GNAMPA) of the Istituto Nazionale di Alta Matematica (INdAM). 
We are grateful to Maria Giovanna Mora for suggesting the problem and for the helpful discussions.

\end{document}